\documentclass[reqno]{amsart}
\usepackage{setspace,tikz,xcolor,mathrsfs,listings,multicol}
\usepackage{rotating}
\usepackage[vcentermath]{youngtab}
\usepackage{fullpage}
\usepackage{enumerate}
\usepackage[all,cmtip]{xy}
\usetikzlibrary{arrows,matrix}

\usepackage[colorlinks=true, pdfstartview=FitV, linkcolor=blue, citecolor=blue, urlcolor=blue]{hyperref}

\newcommand{\g}{\mathfrak{g}}
\newcommand{\HH}{\mathcal{H}}
\newcommand{\clfw}{\overline{\Lambda}} 
\newcommand{\clsr}{\overline{\alpha}} 
\newcommand{\inner}[2]{\left\langle #1, #2 \right\rangle}
\newcommand{\iso}{\cong}

\newcommand{\absval}[1]{\left\lvert #1 \right\rvert}
\newcommand{\virtual}[1]{\widehat{#1}}

\newcommand{\case}[1]{\vspace{12pt}\noindent\underline{#1}:}

\DeclareMathOperator{\RC}{RC} 
\DeclareMathOperator{\wt}{wt} 
\DeclareMathOperator{\ls}{ls} 
\DeclareMathOperator{\rs}{rs} 
\DeclareMathOperator{\lt}{lb} 
\DeclareMathOperator{\rb}{rb} 
\DeclareMathOperator{\cc}{cc} 
\DeclareMathOperator{\fillmap}{fill} 
\DeclareMathOperator{\id}{id} 
\DeclareMathOperator{\ch}{ch} 
\DeclareMathOperator{\gr}{gr} 

\newcommand{\hwRC}{\RC^{HW}} 

\newcommand{\mcA}{\mathcal{A}}

\newcommand{\ZZ}{\mathbb{Z}}

\newcommand{\bon}{\overline{1}}
\newcommand{\btw}{\overline{2}}
\newcommand{\bth}{\overline{3}}
\newcommand{\bfo}{\overline{4}}
\newcommand{\bi}{\overline\imath}

\newcommand{\ellbar}{\overline{\ell}}

\usepackage{listings}
\lstdefinelanguage{Sage}[]{Python}
{morekeywords={False,sage,True},sensitive=true}
\definecolor{dblackcolor}{rgb}{0.0,0.0,0.0}
\definecolor{dbluecolor}{rgb}{0.01,0.02,0.7}
\definecolor{dgreencolor}{rgb}{0.2,0.4,0.0}
\definecolor{dgraycolor}{rgb}{0.30,0.3,0.30}

\usepackage{xparse}

\makeatletter
\protected\def\specialmergetwolists{%
  \begingroup
  \@ifstar{\def\cnta{1}\@specialmergetwolists}
    {\def\cnta{0}\@specialmergetwolists}%
}
\def\@specialmergetwolists#1#2#3#4{%
  \def\tempa##1##2{%
    \edef##2{%
      \ifnum\cnta=\@ne\else\expandafter\@firstoftwo\fi
      \unexpanded\expandafter{##1}%
    }%
  }%
  \tempa{#2}\tempb\tempa{#3}\tempa
  \def\cnta{0}\def#4{}%
  \foreach \x in \tempb{%
    \xdef\cnta{\the\numexpr\cnta+1}%
    \gdef\cntb{0}%
    \foreach \y in \tempa{%
      \xdef\cntb{\the\numexpr\cntb+1}%
      \ifnum\cntb=\cnta\relax
        \xdef#4{#4\ifx#4\empty\else,\fi\x#1\y}%
        \breakforeach
      \fi
    }%
  }%
  \endgroup
}
\makeatother

\usepackage{xparse}
\DeclareDocumentCommand\rpp{ m m g }{
	\foreach \x [count=\s from 1] in {#1}{
	        {\ifnum\s=1
	                \draw (0,-\s)--(\x,-\s);
	                \fi}
	   \draw (0,-\s-1) to (\x,-\s-1);
	   \foreach \y in {0, ..., \x} {\draw (\y,-\s)--(\y,-\s-1);}
	}
	\specialmergetwolists{/}{#1}{#2}\ziplist
	\foreach \x/\y [count=\yi from 1] in \ziplist{
	    \node[anchor=west,font=\small] at (\x,-\yi - .5) {$\y$};
	}
	\IfValueT {#3}
	{\foreach \z [count=\zi from 1] in {#3} {\node[anchor=east,font=\small] at (0,-\zi - .5) {$\z$};}}
	{}
}

\definecolor{darkred}{rgb}{0.7,0,0} 
\newcommand{\defn}[1]{{\color{darkred}\emph{#1}}} 

\theoremstyle{plain}
\newtheorem{thm}{Theorem}[section]
\newtheorem{lemma}[thm]{Lemma}
\newtheorem{conj}[thm]{Conjecture}
\newtheorem{prop}[thm]{Proposition}
\newtheorem{cor}[thm]{Corollary}
\theoremstyle{definition}
\newtheorem{dfn}[thm]{Definition}
\newtheorem{ex}[thm]{Example}
\newtheorem{remark}[thm]{Remark}
\numberwithin{equation}{section}

\begin{document}
\title{A crystal to rigged configuration bijection and the filling map for type $D_4^{(3)}$}
\author{Travis Scrimshaw}
\address{School of Mathematics, University of Minnesota, Minneapolis, MN 55455}
\email{tscrimsh@umn.edu}
\urladdr{http://www.math.umn.edu/~tscrimsh/}

\thanks{The author was partially supported by NSF grant OCI--1147247.}

\begin{abstract}
We give a bijection $\Phi$ from rigged configurations to a tensor product of Kirillov--Reshetikhin crystals of the form $B^{r,1}$ and $B^{1,s}$ in type $D_4^{(3)}$. We show that the cocharge statistic is sent to the energy statistic for tensor products $\bigotimes_{i=1}^N B^{r_i,1}$ and $\bigotimes_{i=1}^N B^{1,s_i}$. We extend this bijection to a single $B^{r,s}$, show that it preserves statistics, and obtain the so-called Kirillov--Reshetikhin tableaux model for $B^{r,s}$. Additionally, we show that $\Phi$ commutes with the virtualization map and that $B^{1,s}$ is naturally a virtual crystal in type $D_4^{(1)}$, thus defining an affine crystal structure on rigged configurations corresponding to $B^{1,s}$.
\end{abstract}

\maketitle

\section{Introduction}
\label{sec:introduction}

Rigged configurations are remarkable combinatorial objects that arose from the study of the Bethe Ansatz for the isotropic Heisenberg model by Kerov, Kirillov, and Reshetikhin in~\cite{KKR86, KR86}. Rigged configurations can be considered as action-angle variables of the box-ball systems, which arise from the study of the inverse scattering transform~\cite{HKOTY02, Takagi05, Yamada04}. Despite their origin in statistical mechanics, rigged configurations have been shown recently to have deep connections to crystal bases, a combinatorial framework to study representations of quantum groups pioneered by Kashiwara in the 1990's~\cite{K90,K91}.

Kerov, Kirillov, and Reshetikhin in~\cite{KKR86, KR86} also gave a bijection $\Phi$ from rigged configurations to highest weight elements in $B = \bigotimes_{i=1}^N B^{1,1}$ in type $A_n^{(1)}$, where $B^{r,s}$ denotes a Kirillov--Reshetikhin (KR) crystal. This was extended to the general case $B = \bigotimes_{i=1}^N B^{r_i, s_i}$ for type $A_n^{(1)}$~\cite{KSS02} and to $B = \bigotimes_{i=1}^N B^{1,1}$ in the non-exceptional types~\cite{OSS03} and $E_6^{(1)}$~\cite{OS12}. The bijection was also extended to $B = B^{r,s}$ in type $D_n^{(1)}$ in~\cite{OSS13} and the remaining non-exceptional types in~\cite{SS15}. It is an open conjecture that a bijection similar to that given in~\cite{KSS02} can be extended to the general case for all affine types. The bijection $\Phi$ is highly recursive, but despite this, it (conjecturally) sends a certain statistic called cocharge on rigged configurations to the energy statistic, which has connections to many aspects of mathematical physics. It also (conjecturally) transforms the intricate combinatorial $R$-matrix on $B$ to the identity map on rigged configurations.

A combinatorial model for KR crystals in non-exceptional types was given in~\cite{FOS09} and for $r = 1, 6, 2$ in type $E_6^{(1)}$ in~\cite{JS10}. Moreover, KR crystals were shown to be perfect in~\cite{FOS10} for non-exceptional types, for $B^{1,s}$ of type $D_4^{(3)}$ in~\cite{KMOY07}, and for $B^{2,s}$ of type $G_2^{(1)}$ in~\cite{Yamane98}. For the non-exceptional types, a combinatorial model was given using Kashiwara--Nakashima tableaux~\cite{KN94}, but this was not the natural image of $\Phi$. Thus for type $D_n^{(1)}$ in~\cite{OSS13} (the special case for $B^{r,1}$ was given in~\cite{S05}) and for the remaining non-exceptional types in~\cite{SS15}, a new tableaux model, coined Kirillov--Reshetikhin tableaux, was given along with a filling map from the Kashiwara--Nakashima tableaux.

Rigged configurations have also been extended to the full (classical) crystal $B$ by Schilling in simply-laced types~\cite{S06}, and later this was extended to all affine types in~\cite{SS15}. It has also been shown that the bijection $\Phi$ is a classical crystal isomorphism in types $A_n^{(1)}$~\cite{DS06} and $D_n^{(1)}$~\cite{Sakamoto13}. The affine crystal structure has also been given for type $A_n^{(1)}$ in full generality in~\cite{SW10} and in type $D_n^{(1)}$ for $B = B^{r,s}$ in~\cite{OSS13}. Additionally, the definition of rigged configurations was expanded to highest weight crystals and $B(\infty)$ of certain types, including all simply-laced, finite, and affine types, in~\cite{SalScrim15}.

The goal of this paper is to give the corresponding crystal isomorphism $\Phi$ for type $D_4^{(3)}$. We do this for tensor products of KR crystals containing factors of the form $B^{1,s}$ and $B^{r,1}$. Moreover, we show for tensor products $B = \bigotimes_{i=1}^N B^{r_i,1}$ and $B = \bigotimes_{i=1}^N B^{1,s_i}$ that the cocharge is sent to energy under $\Phi$, thus giving a bijective proof of the $X = M$ conjecture of~\cite{HKOTY99, HKOTT02} in these cases. We also describe the filling map, the map between the Kang--Misra tableaux~\cite{KM94,KMOY07,Yamada07} and the Kirillov--Reshetikhin tableaux (which is the explicit image under $\Phi$), for general $B^{r,s}$ of type $D_4^{(3)}$. Furthermore, we give some conjectures on an explicit description of the affine crystal structure on rigged configurations. In the process of obtaining our results, we also show that $\Phi$ commutes with the so-called virtualization map of type $D_4^{(3)}$ KR crystals of the form $B^{1,s}$ or $B^{r,1}$ into (a tensor product of) type $D_4^{(1)}$ KR crystals. This proves another special case of Conjecture~2.18 and Conjecture~6.3 (which is an extended version of Conjecture~7.2 in~\cite{OSS03III}) given in~\cite{SS15}. Additionally in the process of our proof, we also show that the combinatorial $R$-matrix for $B^{1,s} \otimes B^{1,1}$, described explicitly in~\cite{Yamada07}, goes to the identity map under $\Phi$.

We must note that $B^{2,s}$ of type $D_4^{(3)}$ for $s > 1$ is not currently known to be the crystal basis of the corresponding Kirillov--Reshetikhin module $W^{2,s}$. This is still an open conjecture~\cite{HKOTY99, HKOTT02} for the exceptional types in general. For non-exceptional types it was shown to be the case in~\cite{OS08}, and in~\cite{LNSSS14, LNSSS14II} for the crystals $B^{r,1}$ in all types. In order to describe the filling map, we give a classical decomposition of $B^{2,s}$ and an affine grading by cocharge that agrees with the results of~\cite{CM07, Hernandez10}, which gives further evidence that $W^{2,s}$ has $B^{2,s}$ as its crystal base.

This paper is organized as follows. In Section~\ref{sec:background}, we give some background on crystals, KR crystals, rigged configurations, the (virtual) Kleber algorithm, and the bijection $\Phi$ for type $D_4^{(1)}$. In Section~\ref{sec:bijection}, we describe the bijection $\Phi$ for type $D_4^{(3)}$. In Section~\ref{sec:filling_map}, we describe the filling map. In Section~\ref{sec:virtualization}, we describe the virtualization map and show that it commutes with $\Phi$. In Section~\ref{sec:affine}, we show that $B^{1,s}$ and $B^{2,1}$ virtualize in type $D_4^{(1)}$ and give conjectures on the affine crystal structure for rigged configurations of $B^{r,s}$. In Section~\ref{sec:main_results}, we prove our main results. In Section~\ref{sec:extensions_G2}, we give some extensions of our results to type $G_2^{(1)}$. We conclude in Section~\ref{sec:affine_conjectures} with some conjectures for the $U_q'(\g)$-crystal structure for rigged configurations for all affine types except $A_n^{(1)}$.

\section{Background}
\label{sec:background}

In this section, we give a background of abstract crystals, Kirillov--Reshetikhin crystals, rigged configurations, the (virtual) Kleber algorithm, associated statistics, and the bijection $\Phi$ and relevant facts for type $D_4^{(1)}$.

\subsection{Crystals}

For this paper, let $\g$ be the Kac--Moody algebra of type $D_4^{(3)}$ with index set $I = \{0, 1, 2\}$, generalized Cartan matrix $A = (A_{ij})_{i,j\in I}$, weight lattice $P$, root lattice $Q$, fundamental weights $\{\Lambda_i \mid i \in I\}$, simple roots $\{\alpha_i \mid i\in I\}$, and simple coroots $\{h_i \mid i \in I\}$ unless otherwise noted.  There is a canonical pairing $\langle\ ,\ \rangle \colon P^\vee \times P \longrightarrow \ZZ$ defined by $\langle h_i, \alpha_j \rangle = A_{ij}$, where $P^{\vee}$ is the dual weight lattice. Let $\g_0$ denote the classical subalgebra of type $G_2$ with index set $I_0 = \{1, 2\}$, weight lattice $\overline{P}$, root lattice $\overline{Q}$, fundamental weights $\{\clfw_1, \clfw_2\}$, and simple roots $\{\clsr_1, \clsr_2\}$. Let $(\ \mid \ ) \colon P \times P \longrightarrow \ZZ$ denote the symmetric bilinear form as in~\cite{kac90}. Let $U_q(\g)$ denote the corresponding quantum group. Let $\g' = [\g, \g]$ be the derived subalgebra of $\g$, and denote $U_q'(\g) := U_q(\g')$.

\begin{figure}[t]
\label{fig:dynkin_diagram}
\[
\begin{tikzpicture}[scale=0.6,baseline=-0.5cm]
\draw (2 cm,0) -- (4.0 cm,0);
\draw (0, 0.15 cm) -- +(2 cm,0);
\draw (0, -0.15 cm) -- +(2 cm,0);
\draw (0,0) -- (2 cm,0);
\draw (0, 0.15 cm) -- +(2 cm,0);
\draw (0, -0.15 cm) -- +(2 cm,0);
\draw[shift={(1.2, 0)}, rotate=0] (135 : 0.45cm) -- (0,0) -- (-135 : 0.45cm);
\draw[fill=white] (0, 0) circle (.25cm) node[below=4pt]{$2$};
\draw[fill=white] (2 cm, 0) circle (.25cm) node[below=4pt]{$1$};
\draw[fill=white] (4 cm, 0) circle (.25cm) node[below=4pt]{$0$};
\end{tikzpicture}
\]
\caption{Dynkin diagram of type $D_4^{(3)}$.}
\end{figure}

An \defn{abstract $U_q(\g)$-crystal} is a nonempty set $\mathcal{B}$ together with the \defn{weight function} $\wt \colon \mathcal{B} \longrightarrow P$, the \defn{crystal operators} $e_a, f_a \colon \mathcal{B} \longrightarrow \mathcal{B} \sqcup \{0\}$, and maps $\varepsilon_a, \varphi_a \colon \mathcal{B} \longrightarrow \ZZ \sqcup \{-\infty\}$ for $a \in I$, subject to the conditions
\begin{enumerate}
\item $\varphi_a(b) = \varepsilon_a(b) + \langle h_a, \wt(b) \rangle$ for all $a \in I$,

\item if $e_a b \in \mathcal{B}$, then
\begin{enumerate}
\item $\varepsilon_a(e_a b) = \varepsilon_a(b) - 1$,
\item $\varphi_a(e_a b) = \varphi_a(b) + 1$, and
\item $\wt(e_a b) = \wt(b) + \alpha_a$.
\end{enumerate}

\item if $f_a b \in \mathcal{B}$, then
\begin{enumerate}
\item $\varepsilon_a(f_a b) = \varepsilon_a(b) + 1$,
\item $\varphi_a(f_a b) = \varphi_a(b) - 1$, and
\item $\wt(f_a b) = \wt(b) - \alpha_a$.
\end{enumerate}

\item $f_a b = b^{\prime}$ if and only if $b = e_a b^{\prime}$ for $b,b^{\prime} \in \mathcal{B}$ and $a \in I$,

\item if $\varphi_a(b) = -\infty$ for $b \in \mathcal{B}$, then $e_a b = f_a b = 0$.
\end{enumerate}

In this paper, all abstract $U_q(\g)$-crystals will be \defn{regular} crystals, which means we define for all $b \in \mathcal{B}$
\begin{subequations}
\begin{align}
\varepsilon_a(b) & = \max\{ k \in \ZZ_{\geq 0} \mid e_a^k b \neq 0 \},
\\ \varphi_a(b) & = \max\{ k \in \ZZ_{\geq 0} \mid f_a^k b \neq 0 \}.
\end{align}
\end{subequations}

Let $\mathcal{B}_1$ and $\mathcal{B}_2$ be abstract $U_q(\g)$-crystals. The tensor product of crystals $\mathcal{B}_2 \otimes \mathcal{B}_1$ is defined to be the Cartesian product $\mathcal{B}_2 \times \mathcal{B}_1$ with the crystal structure
\begin{align*}
e_i(b_2 \otimes b_1) & = \begin{cases}
e_i b_2 \otimes b_1 & \text{if } \varepsilon_i(b_2) > \varphi_i(b_1), \\
b_2 \otimes e_i b_1 & \text{if } \varepsilon_i(b_2) \le \varphi_i(b_1),
\end{cases}
\\ f_i(b_2 \otimes b_1) & = \begin{cases}
f_i b_2 \otimes b_1 & \text{if } \varepsilon_i(b_2) \ge \varphi_i(b_1), \\
b_2 \otimes f_i b_1 & \text{if } \varepsilon_i(b_2) < \varphi_i(b_1),
\end{cases}
\\ \varepsilon_i(b_2 \otimes b_1) & = \max\big( \varepsilon_i(b_2), \varepsilon_i(b_1) - \inner{h_i}{\wt(b_2)} \bigr)
\\ \varphi_i(b_2 \otimes b_1) & = \max\big( \varphi_i(b_1), \varphi_i(b_2) + \inner{h_i}{\wt(b_1)} \bigr)
\\ \wt(b_2 \otimes b_1) & = \wt(b_2) + \wt(b_1).
\end{align*}

\begin{remark}
Our tensor product convention is the opposite to that given in~\cite{K91}.
\end{remark}


Let $\mathcal{B}_1$ and $\mathcal{B}_2$ be two abstract $U_q(\g)$-crystals.  A \defn{crystal morphism} $\psi \colon \mathcal{B}_1 \longrightarrow \mathcal{B}_2$ is a map $\mathcal{B}_1 \sqcup \{0\} \longrightarrow \mathcal{B}_2 \sqcup \{0\}$ with $\psi(0) = 0$ such that for $b \in \mathcal{B}_1$:
\begin{enumerate}
\item if $\psi(b) \in \mathcal{B}_2$, then $\wt(\psi(b)) = \wt(b)$, $\varepsilon_i(\psi(b)) = \varepsilon_i(b)$, and $\varphi_i(\psi(b)) = \varphi_i(b)$;
\item we have $\psi(e_i b) = e_i \psi(b)$ provided $\psi(e_ib) \neq 0$ and $e_i\psi(b) \neq 0$;
\item we have $\psi(f_i b) = f_i \psi(b)$ provided $\psi(f_ib) \neq 0$ and $f_i\psi(b) \neq 0$.
\end{enumerate}
A crystal \defn{embedding} or \defn{isomorphism} is a crystal morphism such that the induced map $\mathcal{B}_1 \sqcup\{0\} \longrightarrow \mathcal{B}_2 \sqcup \{0\}$ is an embedding or bijection respectively. A crystal morphism is \defn{strict} if it commutes with all crystal operators.

If an abstract $U_q(\g)$-crystal $\mathcal{B}$ is isomorphic to the crystal basis~\cite{K90, Lusztig90} of an integrable $U_q(\g)$-module, we simply say $\mathcal{B}$ is a \defn{$U_q(\g)$-crystal}. In particular, an irreducible highest weight $U_q(\g_0)$-module with highest weight $\lambda$, which we denote by $V(\lambda)$, admits a crystal basis~\cite{K90} and is denoted by $B(\lambda)$. Additionally, there is a unique element $u_{\lambda} \in B(\lambda)$ such that $\wt(u_{\lambda}) = \lambda$ and $e_a u_{\lambda} = 0$ for all $a \in I_0$.

For each dominant integral weight $\lambda = k_1 \clfw_1 + k_2 \clfw_2$, we can associate a partition $(k_1 + k_2, k_2)$. We define $\absval{\lambda} = k_1 + 2k_2$ as the usual size of the partition associated to $\lambda$. We can realize $B(\lambda)$ as certain semistandard tableaux of shape $\lambda$ filled with entries in $B(\clfw_1)$ (i.e., the elements $1 < 2 < 3 < 0 < \bth < \btw < \bon$), whose crystal structure is given by embedding into $B(\clfw_1)^{\otimes \absval{\lambda}}$ using the so-called reverse far-eastern reading word, where we read the tableau bottom to top, left to right. In particular, the tableaux are those generated by $f_1$ and $f_2$ from the unique tableau which contains all $1$'s in the first row and all $2$'s in the second row of $\lambda$. The resulting tableaux were explicitly classified by Kang and Misra~\cite{KM94} by giving a set of 1 and 2 adjacent column conditions, along with the fact that $0$ can only appear once in a row.

\subsection{Kirillov--Reshetikhin crystals}

\begin{figure}[t]
\label{fig:crystal}
\[
\begin{array}{|c|}\hline
\begin{tikzpicture}[baseline=-4,yscale=0.7]
\node (1) at (0,0) {$\young(1)$};
\node (2) at (1.5,0) {$\young(2)$};
\node (3) at (3,0) {$\young(3)$};
\node (0) at (4.5,0) {$\young(0)$};
\node (b3) at (6,0) {$\young(\bth)$};
\node (b2) at (7.5,0) {$\young(\btw)$};
\node (b1) at (9,0) {$\young(\bon)$};
\node (ep) at (4.5,2) {$\young(\emptyset)$};
\path[->,font=\tiny]
 (1) edge node[above]{$1$} (2)
 (2) edge node[above]{$2$} (3)
 (3) edge node[above]{$1$} (0)
 (0) edge node[above]{$1$} (b3)
 (b3) edge node[above]{$2$} (b2)
 (b2) edge node[above]{$1$} (b1)
 (b1) edge[out=130,in=0] node[above]{$0$} (ep)
 (ep) edge[out=180,in=50] node[above]{$0$} (1);
\path[->,font=\tiny]
 (b3) edge[out=140,in=40] node[below]{$0$} (2);
\path[->,font=\tiny]
 (b2) edge[out=220,in=320] node[above]{$0$} (3);
\end{tikzpicture}\\ \hline
\end{array}
\]
\caption{The KR crystal $B^{1,1}$ of type $D_4^{(3)}$, which is isomorphic to $B(\clfw_1) \oplus B(0)$ as $U_q(\g_0)$-crystals.}
\end{figure}
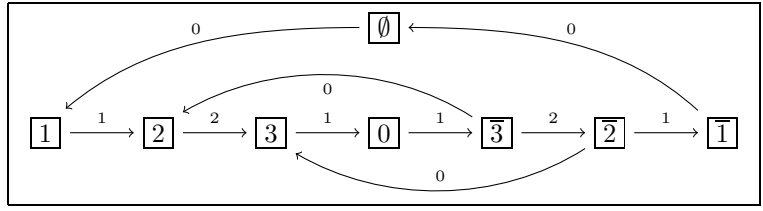

An important class of finite-dimensional $U_q'(\g)$-representations are the Kirillov--Reshetikhin (KR) modules $W^{r,s}$ indexed by $r \in I_0$ and $s \in \ZZ_{\geq 0}$. KR modules are characterized by their Drinfeld polynomials~\cite{CP95, CP98} and correspond to the minimal affinization of $B(s \clfw_r)$~\cite{Chari01}. It is conjectured that all KR modules admit a crystal basis.

\begin{conj}[{\cite{HKOTY99, HKOTT02}}]
\label{conj:crystal_basis}
Let $\g$ be of type $D_4^{(3)}$. The KR module $W^{r,s}$ admits a crystal basis $B^{r,s}$ and is a perfect crystal of level $s$.
\end{conj}

A crystal being perfect is a technical condition that implies we can use a semi-infinite tensor product of $B^{r,s}$ to realize highest weight $U_q(\g)$-crystals, which is known as the \defn{Kyoto path model}, see, e.g.,~\cite{HK02} for details.

Conjecture~\ref{conj:crystal_basis} is known to be true for $W^{1,s}$ and $W^{r,1}$~\cite{HKOTT02, KKMMNN92, KMOY07}, and the crystal corresponding to $W^{r,s}$ is called a \defn{Kirillov--Reshetikhin (KR) crystal}. As $U_q(\g_0)$-crystals, we have
\begin{equation}
\label{eq:classical_decomposition_1}
B^{1,s} \iso \bigoplus_{k=0}^s B(k \clfw_1).
\end{equation}

We now describe the $U_q'(\g)$-crystal structure of $B^{1,s}$. The $U_q(\g_0)$-crystal structure is the same for the Kang--Misra tableaux by embedding $B(k\clfw_1) \subseteq B^{1,s}$ into $B(\clfw_1)^{\otimes k}$. So we only need to describe $e_0$ and $f_0$ on any fixed $b \in B^{1,s}$. Following~\cite{KMOY07}, we first define $x_1, x_2, \overline{x}_2, \overline{x}_1$ as the number of $1, 2, \overline{2}, \overline{1}$ occurring in $b$ respectively, and $x_3$ and $\overline{x}_3$ as twice the number of $3, \overline{3}$ occurring in $b$ respectively plus the number of $0$ in $b$. Next, define
\begin{equation}
z_1 = \overline{x}_1 - x_1, \qquad z_2 = \overline{x}_2 - \overline{x}_3, \qquad z_3 = x_3 - x_2, \qquad z_4 = \frac{1}{2}(\overline{x}_3 - x_3).
\end{equation}
We describe conditions
\begin{align*}
(F_1) & \quad \quad z_1 + z_2 + z_3 + 3z_4 \leq 0, \quad z_1 + z_2 + 3z_4 \leq 0, \quad z_1 + z_2 \leq 0, \quad z_1 \leq 0, 
\\ (F_2) & \quad \quad z_1 + z_2 + z_3 + 3z_4 \leq 0, \quad z_2 + 3z_4 \leq 0, \quad z_2 \leq 0, \quad z_1 > 0,
\\ (F_3) & \quad \quad z_1 + z_3 + 3z_4 \leq 0, \quad z_3 + 3z_4 \leq 0, \quad z_4 \leq 0, \quad z_2 > 0, \quad z_1 + z_2 > 0,
\\ (F_4) & \quad \quad z_1 + z_2 + 3z_4 > 0, \quad z_2 + 3z_4 > 0, \quad z_4 > 0, \quad z_3 \leq 0, \quad z_1 + z_3 \leq 0,
\\ (F_5) & \quad \quad z_1 + z_2 + z_3 + 3z_4 > 0, \quad z_3 + 3z_4 > 0, \quad z_3 > 0, \quad z_1 \leq 0,
\\ (F_6) & \quad \quad z_1 + z_2 + z_3 + 3z_4 > 0, \quad z_1 + z_3 + 3z_4 > 0, \quad z_1 + z_3 > 0, \quad z_1 > 0,
\end{align*}
and conditions $(E_i)$ by replacing $>$ and $\leq$ with $\geq$ and $<$ respectively in $(F_i)$ for $1 \leq i \leq 6$. We define
\begin{equation}
\label{eq:affine_f}
f_0(x_1, x_2, x_3, \overline{x}_3, \overline{x}_2, \overline{x}_1) = \begin{cases}
(x_1+1, x_2, x_3, \overline{x}_3, \overline{x}_2, \overline{x}_1) & \text{if $(F_1)$ holds,} \\
(x_1, x_2, x_3+1, \overline{x}_3+1, \overline{x}_2, \overline{x}_1-1) & \text{if $(F_2)$ holds,} \\
(x_1, x_2, x_3+2, \overline{x}_3, \overline{x}_2-1, \overline{x}_1) & \text{if $(F_3)$ holds,} \\
(x_1, x_2+1, x_3, \overline{x}_3-2, \overline{x}_2, \overline{x}_1) & \text{if $(F_4)$ holds,} \\
(x_1+1, x_2, x_3-1, \overline{x}_3-1, \overline{x}_2, \overline{x}_1) & \text{if $(F_5)$ holds,} \\
(x_1, x_2, x_3, \overline{x}_3, \overline{x}_2, \overline{x}_1-1) & \text{if $(F_6)$ holds,}
\end{cases}
\end{equation}
and similarly $e_0$ by interchanging $+$ and $-$ and $E$ for $F$.

The \defn{combinatorial $R$-matrix} is the unique $U_q'(\g)$-crystal isomorphism $R \colon B \otimes B' \longrightarrow B' \otimes B$ defined by sending $b \otimes b' \mapsto b' \otimes b$ where $b,b'$ is the unique maximal weight element in $B, B'$, respectively. We require the following explicit description for the special case in the sequel. Let $a^n$ denote the row tableau with $a$ occurring $n$ times.

\begin{thm}[{\cite[Prop.~3.7]{Yamada07}}]
\label{thm:r_matrix}
The combinatorial $R$-matrix $R \colon B^{1,1} \otimes B^{1,s} \longrightarrow B^{1,s} \otimes B^{1,1}$ is given on classically highest weight elements by
\begin{align*}
1 \otimes 1^n & \mapsto \begin{cases} 1^{n+1} \bon \otimes 1 & \text{if } 0 \leq n \leq s-2, \\ 1^s \otimes \emptyset & \text{if } n = s-1, \\ 1^s \otimes 1 & \text{if } n = s, \end{cases}
\\ 2 \otimes 1^n & \mapsto \begin{cases} 1^{n-1}20 \otimes 1 & \text{if } 1 \leq n \leq s-1, \\ 1^{s-1} 2 \otimes 1 & \text{if } n = s, \end{cases}
\\ 0 \otimes 1^n & \mapsto 1^{n-1}0 \otimes 1,
\\ \bth \otimes 1^n & \mapsto 1^{n-2}2 \otimes 1,
\\ \bon \otimes 1^n & \mapsto \begin{cases} \bon \otimes 1 & \text{if } n = 1, \\ \emptyset \otimes 1 & \text{if } n = 2, \\ 1^{n-2} \otimes 1 & \text{if } 3 \leq n \leq s, \end{cases}
\\ \emptyset \otimes 1^n & \mapsto \begin{cases} 1^n \otimes \emptyset & \text{if } 0 \leq n \leq s-1, \\ 1^{s-1} \otimes 1 & \text{if } n = s, \end{cases}
\end{align*}
and extended as a $U_q(\g_0)$-crystal isomorphism.
\end{thm}

The KR module $W^{2,s}$ is known to have the following classical decomposition. The following is Theorem~9.2 in the arXiv version of~\cite{Hernandez10}.
\begin{thm}
\label{thm:decomposition_2}
We have the following decomposition as $U_q(\g_0)$-modules:
\[
W^{2,s} \iso \bigoplus_{\substack{m_1+m_2 \leq s \\ m_1,m_2 \geq 0}} (m_1 + 1) \min(1 + m_2, 1 + s - m_1 - m_2) V(m_1 \clfw_1 + m_2 \clfw_2)
\]
\end{thm}
In addition, the KR module $W^{2,s}$ has also been given a $\ZZ_{\geq 0}$ grading in~\cite{CM07}.
\begin{thm}
\label{thm:grading}
Define
\begin{align*}
\mcA & = \{r \in \ZZ_{\geq 0}^4 \mid r_3 \leq r_1 \text{ and } r_1 + r_2 + r_3 + r_4 \leq s\},
\\ \overline{\wt}(r) & = (r_1 + r_2 - r_3) \clfw_1 + (s - r_1 - r_2 - r_4) \clfw_2,
\\ \gr(r) & = r_1 + 2 r_2 + 2 r_3 + 3 r_4.
\end{align*}
Then we have
\[
\ch_t W^{2,s} = \sum_{r \in \mcA} t^{\gr(r)} \ch V\bigl( \overline{\wt}(r) \bigr),
\]
where $\ch_t$ is the graded character.
\end{thm}

\subsection{Rigged configurations}

Let $\HH_0 = I_0 \times \ZZ_{> 0}$.  Consider a multiplicity array $L = \bigl(L_i^{(a)} \in \ZZ_{\geq 0} \mid (a,i) \in \HH_0 \bigr)$
and a dominant integral weight $\lambda$ of $\g_0$. An \defn{$(L; \lambda)$-configuration} is a sequence of partitions $\nu = \{ \nu^{(a)} \mid a \in I_0 \}$ such that 
\begin{equation}\label{LL-config}
\sum_{(a,i)\in\HH_0} i m_i^{(a)} \clsr_a = \sum_{(a,i) \in \HH_0} i L_i^{(a)} \clfw_a - \lambda,
\end{equation}
where $m_i^{(a)}$ is the number of parts of length $i$ in the partition $\nu^{(a)}$. We denote the set of $(L; \lambda)$-configurations by $C(L; \lambda)$. The \defn{vacancy numbers} of $\nu \in C(L; \lambda)$ are defined as
\begin{equation}
\label{eq:vacancy_numbers}
p_i^{(a)} = \sum_{j \geq 1} \min(i,j) L_j^{(a)} - \sum_{(b,j) \in \HH_0} A_{ab} \min(i, j) m_j^{(b)}.
\end{equation}

A \defn{rigged configuration} of classical weight $\lambda$ is an $(L; \lambda)$-configuration $\nu$, along with a sequence of multisets of integers $J = \{ J_i^{(a)} \mid (a, i) \in \HH_0 \}$ such that $\absval{J_i^{(a)}} = m_i^{(a)}$ and $\max J_i^{(a)} \leq p_i^{(a)}$. (Often each $J_i^{(a)}$ will be sorted in weakly decreasing order.) So for each row of length $i$, we have an integer $x \in J_i^{(a)}$ and we call the pair $(i, x)$ a \defn{string}. We denote the set of strings associated to $\nu^{(a)}$ by $(\nu, J)^{(a)}$ (as opposed to $(\nu^{(a)}, J^{(a)})$). An integer $x \in J_i^{(a)}$ is called a \defn{label}, \defn{rigging}, or \defn{quantum number}, and we associate a label to a particular row in $\nu^{(a)}$ by considering $J_i^{(a)}$ sorted in weakly decreasing order. The \defn{colabel} or \defn{corigging} of a string $(i,x)$ is defined as $p_i^{(a)} - x$.

A rigged configuration is \defn{highest weight} if $\min J_i^{(a)} \geq 0$ for all $(a,i) \in \HH_0$. A string $(i, x)$ in $\nu^{(a)}$ is \defn{singular} if $x = p_i^{(a)}$ and is \defn{quasi-singular} if $x = p_i^{(a)} - 1$ and $\max J_i^{(a)} \neq p_i^{(a)}$ (i.e., there does not exist a singular string of length $i$ in $\nu^{(a)}$).

Denote by $\hwRC(L; \lambda)$ the set of highest weight rigged configurations. Rigged configurations have an abstract $U_q(\g_0)$-crystal structure~\cite{S06,SS15}. We begin by recalling the classical crystal operators.
\begin{dfn}
\label{def:rc_crystal_ops}
Let $\g$ be a Lie algebra of finite or affine type and $L$ a multiplicity array. Let $(\nu, J)$ be a rigged configuration. Fix $a\in I_0$ and let $x$ be the smallest label of $(\nu,J)^{(a)}$.  
\begin{enumerate}
\item If $x \geq 0$, then set $e_a(\nu,J) = 0$. Otherwise, let $\ell$ be the minimal length of all strings in $(\nu,J)^{(a)}$ which have label $x$. The rigged configuration $e_a(\nu,J)$ is obtained by replacing the string $(\ell, x)$ with the string $(\ell-1, x+1)$ and changing all other labels so that all colabels remain fixed.

\item If $x > 0$, then add the string $(1,-1)$ to $(\nu,J)^{(a)}$.  Otherwise, let $\ell$ be the maximal length of all strings in $(\nu,J)^{(a)}$ which have label $x$, and replace the string $(\ell, x)$ by the string $(\ell+1, x-1)$. In both cases, change all other labels so that all colabels remain fixed.  If the result is a rigged configuration, then it is $f_a(\nu, J)$.  Otherwise set $f_a(\nu, J) = 0$.
\end{enumerate}
\end{dfn}

\begin{remark}
The condition for highest weight rigged configurations matches the usual crystal theoretic definition; i.e., that $e_a(\nu, J) = 0$ for all $a \in I_0$ and any $(\nu, J) \in \hwRC(L; \lambda)$.
\end{remark}

Let $\RC(L; \lambda)$ denote the set generated by $\hwRC(L; \lambda)$ and crystal operators given in Definition~\ref{def:rc_crystal_ops}. Let $\RC(L)$ be the closure under the crystal operators of the set $\hwRC(L) = \bigsqcup_{\lambda\in P^+} \hwRC(L; \lambda)$. To obtain the weight, we first note that the classical weight is given by
\begin{equation}
\label{RC_weight}
\begin{split}
\overline{\wt}(\nu,J) & = \sum_{(a,i) \in \HH_0} i\big( L_i^{(a)} \clfw_a - m_i^{(a)} \clsr_a \big)
\\ & = \sum_{(a,i) \in \HH_0} i L_i^{(a)} \clfw_a - \lvert \nu^{(a)} \rvert \clsr_a.
\end{split}
\end{equation}
We can extend this to $\wt \colon \RC(L) \longrightarrow P$ by $\wt(\nu, J) = k_0 \Lambda_0 + \overline{\wt}(\nu, J)$, where $k_0$ is such that $\langle \wt(\nu, J), c \rangle = 0$ with $c$ being the canonical central element of $\g$. Explicitly, suppose $\overline{\wt}(\nu, J) = c_1 \clfw_1 + c_2 \clfw_2$, then we have
\begin{equation}
\label{eq:affine_weight}
\wt(\nu, J) = -(2c_1 + 3c_2) \Lambda_0 + c_1 \Lambda_1 + c_2 \Lambda_2.
\end{equation}

\begin{ex}
\label{ex:runningrig}
Rigged configurations will be depicted with vacancy numbers on the left and riggings on the right. For example, 
\[
(\nu, J) = \begin{tikzpicture}[scale=.35,anchor=top,baseline=-18]
 \rpp{4,1}{3,1}{5,1}
 \begin{scope}[xshift=8cm]
 \rpp{4}{-2}{-2}
 \end{scope}
\end{tikzpicture}
\]
is a rigged configuration in $\RC(L; 2\clfw_1 + \clfw_2)$ with $L$ is given by $L_1^{(1)} = L_2^{(1)} = L_1^{(2)} = 1$ with all other $L_i^{(a)} = 0$. Then $\wt(\nu, J) = 4\Lambda_0 + 5\Lambda_1 - 2\Lambda_2$ and
\begin{align*}
e_1(\nu,J) & = 0,
& e_2(\nu,J) & = \hspace{7.2pt}
\begin{tikzpicture}[scale=.35,baseline=-20]
 \rpp{4,1}{0,1}{2,1}
 \begin{scope}[xshift=8cm]
 \rpp{3}{-1}{-1}
 \end{scope}
\end{tikzpicture},
\\ f_1(\nu,J) & = 
\begin{tikzpicture}[scale=.35,baseline=-20]
 \rpp{4,1,1}{1,-1,-1}{3,-1,-1}
 \begin{scope}[xshift=8cm]
 \rpp{4}{-1}{-1}
 \end{scope}
\end{tikzpicture},
& f_2(\nu,J) & = 0.
\end{align*}
See Appendix~\ref{sec:sage} on how to construct this example in Sage~\cite{sage}.
\end{ex}

\begin{thm}[{\cite{SS15}}]
\label{thm:rc_crystal}
Let $\g_0$ be a Lie algebra of type $G_2$. For $(\nu, J) \in \hwRC(L; \lambda)$, let $X_{(\nu, J)}$ be the closure of $(\nu, J)$ under $e_a, f_a$ for $a \in I_0$. Then $X_{(\nu, J)} \iso B(\lambda)$ as $U_q(\g_0)$-crystals.
\end{thm}

Let $B = \bigotimes_{k=1}^N B^{r_k, s_k}$, and let $\RC(B)$ be $\RC(L)$ with $L_i^{(a)}$ being the number of factors $B^{a,i}$ occurring in $B$.

\begin{dfn}
\label{def:complement_rigging}
The \defn{complement rigging map} $\eta \colon \RC(B) \longrightarrow \RC(B_{rev})$, where $B_{rev}$ is the reverse ordering of $B$, is defined by $(\nu, J) \mapsto (\nu, J')$, where $J'$ is formed by replacing each label $x$ by its colabel $p_i^{(a)} - x$ on highest weight elements and extending as a $U_q(\g_0)$-crystal isomorphism.
\end{dfn}

\subsection{Virtual crystals}

Let $\virtual{\g}$ be the Kac--Moody algebra of type $D_4^{(1)}$ with index set $\virtual{I} = \{0, 1, 2, 3, 4\}$ and $\virtual{\g}_0$ be of type $D_4$ given by the index set $\virtual{I}_0 = \virtual{I} \setminus \{ 0 \}$. Highest weight crystals of type $D_4$ were described using tableaux by Kashiwara and Nakashima~\cite{KN94}. The KN tableaux were then used to describe KR crystals in~\cite{FOS09}. More specifically, for $r < n - 1$, the description of $B^{r,s}$ is given using the inherent classical crystal structure of KN tableaux, the classical decomposition $B^{r,s} \iso \bigoplus_{\lambda} B(\lambda)$, where the sum is over all $\lambda$ obtained from removing vertical dominoes from an $r \times s$ rectangle, and then defining $e_0$ and $f_0$ by using $\pm$-diagrams. Recently, a different tableaux model for the KR crystals was introduced called \defn{Kirillov--Reshetikhin (KR) tableaux} $\virtual{T}^{r,s}$ with the map between the two models $\virtual{\fillmap} \colon \virtual{B}^{r,s} \longrightarrow \virtual{T}^{r,s}$ called the \defn{filling map}~\cite{OSS13}. Moreover, the analog of Theorem~\ref{thm:rc_crystal} for type $\virtual{\g}$ (using the same crystal structure as given by Definition~\ref{def:rc_crystal_ops}) was given in~\cite{S06}.

We consider the diagram folding $\phi \colon \virtual{I} \searrow I$ defined by $\phi^{-1}(0) = \{0\}$, $\phi^{-1}(1) = \{2\}$, and $\phi^{-1}(2) = \{1,3,4\}$. The folding $\phi$ restricts to a diagram folding of type $\virtual{\g}_0 \searrow \g_0$ by $\virtual{I}_0 \searrow I_0$ and by abuse of notation also denote this folding by $\phi$.

\begin{figure}[t]
\label{fig:dynkin_diagram_D4_affine}
\[
\begin{tikzpicture}[scale=0.6,baseline=-0.5cm]
\draw (0,0.0 cm) -- (4 cm,0);
\draw (2 cm,0) -- (4 cm, 1 cm);
\draw (2 cm,0) -- (4 cm,-1 cm);
\draw[fill=white] (0 cm, 0 cm) circle (.25cm) node[left=3pt]{$0$};
\draw[fill=white] (2 cm, 0 cm) circle (.25cm) node[below=4pt]{$2$};
\draw[fill=white] (4 cm, 1 cm) circle (.25cm) node[right=3pt]{$1$};
\draw[fill=white] (4 cm, 0 cm) circle (.25cm) node[right=3pt]{$4$};
\draw[fill=white] (4 cm, -1 cm) circle (.25cm) node[right=3pt]{$3$};
\end{tikzpicture}
\]
\caption{Dynkin diagram of type $D_4^{(1)}$.}
\end{figure}

\begin{figure}[t]
\label{fig:crystal_D4_affine}
\[
\begin{array}{|c|}\hline
\begin{tikzpicture}[baseline=-4,yscale=0.7]
\node (1) at (0,0) {$\young(1)$};
\node (2) at (1.5,0) {$\young(2)$};
\node (3) at (3,0) {$\young(3)$};
\node (4) at (4.5,1) {$\young(4)$};
\node (b4) at (4.5,-1) {$\young(\bfo)$};
\node (b3) at (6,0) {$\young(\bth)$};
\node (b2) at (7.5,0) {$\young(\btw)$};
\node (b1) at (9,0) {$\young(\bon)$};
\path[->,font=\tiny]
 (1) edge node[above]{$1$} (2)
 (2) edge node[above]{$2$} (3)
 (3) edge node[above]{$3$} (4)
 (3) edge node[above]{$4$} (b4)
 (4) edge node[above]{$4$} (b3)
 (b4) edge node[above]{$3$} (b3)
 (b3) edge node[above]{$2$} (b2)
 (b2) edge node[above]{$1$} (b1);
\path[->,font=\tiny]
 (b2) edge[out=140,in=40] node[below]{$0$} (1);
\path[->,font=\tiny]
 (b1) edge[out=220,in=320] node[above]{$0$} (2);
\end{tikzpicture}\\ \hline
\end{array}
\]
\caption{The KR crystal $\virtual{B}^{1,1}$ of type $D_4^{(1)}$, which is isomorphic to $B(\virtual{\clfw}_1)$ as $U_q(\g_0)$-crystals.}
\end{figure}

\begin{remark}
To simplify our notation, for any object $X$ or $X_0$ of $\g$ or $\g_0$, we denote the corresponding object of $\virtual{\g}$ or $\virtual{\g}_0$, respectively, by $\virtual{X}$ or $\virtual{X}_0$, respectively.
\end{remark}

Furthermore, the folding $\phi$ induces an embedding of weight lattices $\Psi \colon P \longrightarrow \virtual{P}$ given by
\begin{equation}
\Lambda_a \mapsto \sum_{b \in \phi^{-1}(a)} \virtual{\Lambda}_b, \hspace{60pt}
\alpha_a \mapsto \sum_{b \in \phi^{-1}(a)} \virtual{\alpha}_b.
\end{equation}
This gives an embedding of crystals as sets $v \colon B(\lambda) \longrightarrow B\bigl( \Psi(\lambda) \bigr)$. We let $B^v(\lambda)$ denote the image of $v$. We can define a crystal structure on $B^v(\lambda)$ induced from the crystal $B\bigl( \Psi(\lambda) \bigr)$ by
\begin{equation}
\begin{gathered}
e^v := \prod_{b \in \phi^{-1}(a)} \virtual{e}_b, \hspace{60pt}
f^v := \prod_{b \in \phi^{-1}(a)} \virtual{f}_b,
\\ \hspace{15pt} \varepsilon_a^v := \virtual{\varepsilon}_x, \hspace{95pt}
\varphi_a^v := \virtual{\varphi}_x,
\\ \wt := \Psi^{-1} \circ \virtual{\wt},
\end{gathered}
\end{equation}
where we fix any $x \in \phi^{-1}(a)$. We say the pair $\Bigl(B^v(\lambda), B\bigl( \Psi(\lambda) \bigr) \Bigr)$ is a \defn{virtual crystal} and the isomorphism $v$ is the \defn{virtualization map}.

\begin{prop}[{\cite{SS15}}]
We have $B(\lambda) \iso B^v(\lambda)$.
\end{prop}
We can explicitly define a virtualization map on rigged configurations by
\begin{subequations}
\label{eq:RC_virtualization}
\begin{align}
\virtual{\nu}^{(b)} & = \nu^{(a)},
\\ \virtual{J}_i^{(b)} & = J_i^{(a)}
\end{align}
\end{subequations}
for all $b \in \phi^{-1}(a)$~\cite{SS15}. We also need the following fact.
\begin{prop}[{\cite[Prop.~6.4]{OSS03III}}]
\label{prop:virtual_tensor}
Virtual crystals form a tensor category.
\end{prop}

The following conjecture is a special case of Conjecture 3.7 in~\cite{OSS03II}.

\begin{conj}
\label{conj:KR_virtualization}
We have the following virtualizations of type $\g$ in type $\virtual{\g}$:
\begin{align*}
B^{1,s} & \longrightarrow \virtual{B}^{2,s},
\\ B^{2,s} & \longrightarrow \virtual{B}^{1,s} \otimes \virtual{B}^{3,s} \otimes \virtual{B}^{4,s}.
\end{align*}
\end{conj}

Next we define the \defn{virtual combinatorial $R$-matrix} $R^v$ for $B^{1,s} \otimes B^{1,s'}$ by the restriction of the $R$-matrix $\virtual{R}$ of type $\virtual{\g}$ to its image under $v$. This is well-defined since the classically highest weight elements exactly agrees with those in $\virtual{B}^{2,s} = v(B^{1,s})$ under $v$ and the $R$-matrix sends a classically highest weight element to a classically highest weight element. However, it is not clear that $R^v$ is well-defined for $B^{2,s} \otimes B$, where $B$ is a KR crystal of type $\g$. For KR crystals $B$ and $B'$ of type $\g$, if $R \colon B \otimes B' \longrightarrow B' \otimes B$ is well-defined, then the diagram
\[
\xymatrixrowsep{3.5pc}
\xymatrixcolsep{3pc}
\xymatrix{B \otimes B' \ar[r]^R \ar[d]^v & B' \otimes B \ar[d]^v
\\ v(B) \otimes v(B') \ar[r]^{R^v} & v(B') \otimes v(B)}
\]
commutes since $B \otimes B'$ is connected.

\begin{conj}[{\cite{OSS03II}}]
The virtual $R$-matrix is well-defined for $B^{r,s} \otimes B^{r',s'}$.
\end{conj}

Let $\widetilde{B}^{3,s}$ be the representation given by tableaux in a $3 \times s$ rectangle with the classical decomposition
\[
\widetilde{B}^{3,s} \iso \bigoplus_{k=0}^s B\bigl( k \clfw_1 + (s-k) (\clfw_3 + \clfw_4) \bigr)
\]
and the $U_q'(\g)$-crystal structure given by $\pm$-diagrams as in~\cite{FOS09}. The following is a special case for $D_4^{(1)}$ of Conjecture~5.16 in~\cite{SS15}.

\begin{conj}
\label{conj:spinor_columns}
We have $\widetilde{B}^{3,s} \iso \virtual{B}^{3,s} \otimes \virtual{B}^{4,s}$ as $U_q'(\g)$-crystals of type $D_4^{(1)}$.
\end{conj}

\begin{thm}[{\cite[Thm~3.3]{S05}}]
Conjecture~\ref{conj:spinor_columns} holds for $s = 1$.
\end{thm}

\subsection{(Virtual) Kleber Algorithm}

We now recall the \defn{Kleber algorithm}~\cite{Kleber98,OSS03II}, which is used to construct the classically highest weight components in a tensor product of KR crystals.
\begin{dfn}[Kleber algorithm]
\label{def:kleber_algorithm}
Consider a tensor product of KR crystals $\virtual{B}$ of type $\virtual{\g}$. We construct the \defn{Kleber tree} $T(\virtual{B})$, whose nodes will be labelled by weights in $\virtual{P}_0^+$, the dominant weight lattice of type $D_4$, and edges are labelled by $d_{xy} := x - y \in \virtual{Q}_0^+$, the dominant root lattice, such that $d_{xy} \neq 0$, recursively starting with $T_0$ consisting of a single node of weight $0$.
\begin{itemize}
\item[(K1)] Let $T^{\prime}_{\ell}$ be obtained from $T_{\ell-1}$ by adding $\sum_{a=1}^4 \virtual{\Lambda}_a \sum_{i \geq \ell} \virtual{L}_i^{(a)}$ to the weight of each node.
\item[(K2)] Construct $T_{\ell}$ from $T^{\prime}_{\ell}$ as follows. Let $x$ be a node at depth $\ell - 1$. Suppose there is a weight $y \in \overline{P}^+$ such that $x - y \in \virtual{Q}_0^+ \setminus \{0\}$. If $x$ is not the root node of $T'_{\ell}$, then let $w$ be the parent of $x$. If $(w - x)$ is larger than $(x - y)$ component-wise expressed as a sum of the simple roots $\virtual{\alpha}_i$ (equivalently we have $d_{wx} - d_{xy} \in \virtual{Q}_0^+ \setminus \{0\}$), then attach $y$ as a child of $x$.
\item[(K3)] If $T_{\ell} \neq T_{\ell-1}$, then repeat from~(K1); otherwise terminate and return $T(B) = T_{\ell}$.
\end{itemize}
Next we construct a highest weight rigged configuration from a node $x$ at depth $p$ in the Kleber tree $T(B)$ as follows. Let $x^{(0)}, x^{(1)}, \dotsc, x^{(p)} = x$ be the weights of nodes on the path from the root of $T(\virtual{B})$ to $x$. The resulting configuration $\nu$ is given by
\[
\virtual{m}_i^{(a)} = (x^{(i-1)} - 2 x^{(i)} + x^{(i+1)} \mid \virtual{\Lambda}_a)
\]
where we make the convention that $x = x^{(k)}$ for all $k > p$. In other words, there are $j$ rows of length $i$ in $\nu^{(a)}$ where $j$ is the coefficient of $\virtual{\alpha}_a$ in the difference of the corresponding edge labels. We then take the riggings over all possible values between $0$ and $\virtual{p}_i^{(a)}$.
\end{dfn}

We can compute the vacancy numbers just using the data in the Kleber tree by
\begin{equation}
\label{eq:kleber_vacancy}
\virtual{p}_i^{(a)} = ( \virtual{\alpha}_a \mid \lambda^{(i)} ) - \sum_{j > i} (j - i) \virtual{L}_j^{(a)}.
\end{equation}
We note there is a minor typo in~\cite[Eq.~(5.2)]{OSS03II} that has been corrected in the arXiv version.

For type $\g$, the algorithm is modified by using virtual rigged configurations and is known as the \defn{virtual Kleber algorithm}.

\begin{dfn}[Virtual Kleber algorithm]
\label{def:virtual_kleber}
Let $B$ be a tensor product of KR crystals of type $D_4^{(3)}$. The \defn{virtual Kleber tree} $\virtual{T}(B)$ is constructed by following the construction of the Kleber tree $T(\virtual{B})$ (of type $\virtual{\g}$) except we add a child in step~(K2) only if both of the following conditions are satisfied.
\begin{itemize}
\item[(V1)] We have $(y | \virtual{\alpha}_a) = (y | \virtual{\alpha}_b)$ for all $a, b \in \sigma^{-1}$.
\item[(V2)] If $\ell - 1 \notin \gamma_a \ZZ$, then for $w$ the parent of $x$, the $a$-th component of $d_{wx}$ and $d_{xy}$ must be equal.
\end{itemize}
To construct the rigged configurations from the nodes of the virtual Kleber tree, we take the devirtualization of the resulting (virtual) rigged configurations obtained by the usual Kleber algorithm given in Definition~\ref{def:kleber_algorithm}.
\end{dfn}

\subsection{Statistics}

There is a statistic called \defn{energy} defined on $B = \bigotimes_{i=1}^N B^{r_i,s_i}$~\cite{HKOTY99}. For the following definition, if we have $B^{2,s}$, it is understood that $s = 1$. First we define the \defn{local energy function} on $B^{r,s} \otimes B^{r',s'}$ as follows. Let $c' \otimes c = R(b \otimes b')$.
\begin{equation}
\label{eq:local_energy}
H\bigl( e_i(b \otimes b') \bigr) = H(b \otimes b') + \begin{cases}
-1 & \text{if } i = 0 \text{ and } e_0(b \otimes b') = b \otimes e_0 b' \text{ and } e_0(c' \otimes c) = c' \otimes e_0 c, \\
1 & \text{if } i = 0 \text{ and } e_0(b \otimes b') = e_0 b \otimes b' \text{ and } e_0(c' \otimes c) = e_0 c' \otimes c, \\
0 & \text{otherwise.}
\end{cases}
\end{equation}
The local energy function is defined up to an additive constant~\cite{KKMMNN91}, and so we normalize $H$ by the condition $H\bigl( u(B^{r ,s}) \otimes u(B^{r',s'}) \bigr) = 0$, where $u(B^{r,s})$ is the unique element of classical highest weight $s \clfw_r$. Next we define $D_{B^{r,s}} \colon B^{r,s} \longrightarrow \ZZ$ by
\[
D_{B^{r,s}}(b) = H(b \otimes b^{\sharp}) - H(u(B^{r,s}) \otimes b^{\sharp}),
\]
where $b^{\sharp} = \emptyset \in B(0) \subseteq B^{r,s}$ is the unique element such that $\varphi(b^{\sharp}) = \sum_{i \in I} \varphi_i(b^{\sharp}) = s \Lambda_0$. In particular, for $b \in B(k \clfw_1) \subseteq B^{1,s}$ we have
\begin{equation}
\label{eq:energy_single_row}
D_{B^{1,s}}(b) = s - k
\end{equation}
from~\cite[Cor.~3.8]{Yamada07} since $u(B^{1,s}) = 1^s$. From~\cite{Yamada07}, this is well-defined for $B^{2,1}$, and in fact, we can order the classical components $\bigl( B(\clfw_2), B(\clfw_1), B(\clfw_1), B(0) \bigr)$ and an element $b$ in the $i$-th component has $D_{B^{2,1}}(b) = i-1$. Then we define
\begin{equation}
\label{eq:energy_function}
D(b_N \otimes \cdots \otimes b_1) = \sum_{1 \leq i < j \leq N} H_i R_{i+1} R_{i+2} \cdots R_{j-1} + \sum_{j=1}^N D_{B^{1,s_j}} R_1 R_2 \cdots R_{j-1},
\end{equation}
where $R_i$ and $H_i$ are the combinatorial $R$-matrix and local energy function, respectively, acting on the $i$-th and $(i+1)$-th factors and $D_{B^{1,s_j}}$ acts on the rightmost factor. We say the \defn{energy} of an element $b \in B$ is $D(b)$.

There is a statistic called \defn{cocharge} on rigged configurations given by first defining cocharge on $(L; \lambda)$-configurations $\nu$.
\begin{equation}
\label{eq:cocharge_configurations}
\cc(\nu) = \frac{1}{2} \sum_{a,b \in I_0} \sum_{i,j \in \ZZ_{>0}} (\alpha_a | \alpha_b) \min(i, j) m_i^{(a)} m_j^{(b)}
\end{equation}
We then extend this to rigged configurations by adding all of the labels:
\begin{equation}
\label{eq:cocharge}
\cc(\nu, J) = \cc(\nu) + \sum_{(a,i) \in \HH_0} \sum_{x \in J_i^{(a)}} x.
\end{equation}
In addition, cocharge is invariant under $e_a$ and $f_a$ for $a \in I_0$.
\begin{prop}[{\cite[Prop.~3.11]{SS15}}]
\label{prop:cocharge_classical_invar}
Fix a classical component $X_{(\nu, J)}$ as given in Theorem~\ref{thm:rc_crystal}. The cocharge $\cc$ is constant on $X_{(\nu, J)}$.
\end{prop}

We recall (an equivalent form of) the $X = M$ conjecture of~\cite{HKOTY99, HKOTT02}. Let $\mathcal{P}(B; \lambda)$ denote the classically highest weight elements of $B$ of classical weight $\lambda$. The \defn{one-dimensional sum} is defined by
\begin{equation}
\label{eq:X}
X(B, \lambda; q) = \sum_{b \in \mathcal{P}(B; \lambda)} q^{D(b)}
\end{equation}
and the \defn{fermionic formula} by
\begin{equation}
\label{eq:M}
M(B, \lambda; q) = \sum_{\nu \in C(B; \lambda)} q^{\cc(\nu)} \prod_{a \in I_0}\prod_{i=1}^{\infty} \begin{bmatrix} m_i^{(a)} + p_i^{(a)} \\ m_i^{(a)} \end{bmatrix}_q = \sum_{(\nu, J) \in \hwRC(B; \lambda)} q^{\cc(\nu, J)},
\end{equation}
where the last equality comes from the fact that $J_i^{(a)}$ of a highest weight rigged configuration can be considered as partitions in a $p_i^{(a)} \times m_i^{(a)}$ box for all $(a, i) \in \HH_0$.

\begin{conj}
\label{conj:X=M}
Let $B$ be a tensor product of KR crystals of type $D_4^{(3)}$. Then we have
\[
X(B, \lambda; q) = M(B, \lambda; q).
\]
\end{conj}

Next we define virtual analogs of cocharge and energy by
\begin{align*}
D^v(b) & := \virtual{D}\bigl( v(b) \bigr),
\\ \cc^v(\nu, J) & := \virtual{\cc}(\virtual{\nu}, \virtual{J}).
\end{align*}
We define
\begin{align*}
X^v(B, \lambda; q) & = \sum_{b \in \mathcal{P}(B; \lambda)} q^{D^v(b)},
\\ M^v(B, \lambda; q) & = \sum_{(\nu, J) \in \hwRC(B; \lambda)} q^{\cc^v(\nu, J)}.
\end{align*}

\begin{prop}[{\cite{OSS03II}}]
\label{prop:virtual_statistics}
Let $B^v$ be a virtual crystal of $B$. Then we have
\begin{align*}
D^v(b) & = D(b),
\\ \cc^v(\nu, J) & = \cc(\nu, J).
\end{align*}
Moreover, we have
\begin{align*}
X^v(B, \lambda; q) & = X(B, \lambda; q),
\\ M^v(B, \lambda; q) & = M(B, \lambda; q).
\end{align*}
\end{prop}


\subsection{Bijection and filling map for $D_4^{(1)}$}

We now recall some facts about the bijection $\virtual{\Phi}$ and the filling map for type $D_4^{(1)}$.

\begin{thm}[{\cite{S05, SS2006, OSS13}}]
\label{thm:statistics_preserving_Daff}
Let $\virtual{B} = \bigotimes_{i=1}^N \virtual{B}^{r_i,1}$ or $\virtual{B} = \bigotimes_{i=1}^N \virtual{B}^{1,s_i}$ or $\virtual{B} = \virtual{B}^{r,s}$ be a $D_4^{(1)}$-crystal. Then
\[
\virtual{\Phi} \colon \RC(\virtual{B}) \longrightarrow \virtual{B}
\]
is a bijection such that $\virtual{\Phi} \circ \virtual{\eta}$ sends cocharge to energy.
\end{thm}

It is known in type $D_n^{(1)}$ that $\virtual{\Phi}$ commutes with the crystal operators.

\begin{thm}[{\cite{Sakamoto13}}]
\label{thm:crystal_isomorphism_D}
Consider a $D_4^{(1)}$-crystal $\virtual{B} = \bigotimes_{i=1}^N B^{r_i, s_i}$, and assume $\virtual{\Phi} \colon \RC(\virtual{B}) \longrightarrow \virtual{B}$ is a bijection on classically highest weight elements. Then $\virtual{\Phi}$ is a classical crystal isomorphism.
\end{thm}

Moreover, a $U_q'(\g)$-crystal structure was described for $\RC(B^{r,s})$ in~\cite[Thm.~4.9]{OSS13}. Thus with Theorem~\ref{thm:crystal_isomorphism_D}, we have the following.

\begin{thm}
\label{thm:affine_crystal_iso_Daff}
The bijection $\Phi \colon \RC(\virtual{B}^{r,s}) \longrightarrow \virtual{B}^{r,s}$ of type $D_n^{(1)}$ is a $U_q'(\g)$-crystal isomorphism.
\end{thm}

We also describe a special case of the filling map for type $D_4^{(1)}$ from~\cite{OSS13}. We refer the reader to~\cite{OSS13} for the general case.

\begin{dfn}
\label{def:filling_map_Daff}
Consider $b \in B(k\Lambda_r) \subseteq \virtual{B}^{r,s}$. We define $\fillmap \colon \virtual{B}^{r,s} \longrightarrow \virtual{T}^{r,s}$ as follows:
\begin{itemize}
\item If $r = 1,3,4$, then we must have $k = s$ and $\fillmap$ is the identity.
\item If $r = 2$, then fix some $0 \leq k \leq s$. We define $\fillmap(b)$ by adding $\left\lfloor \frac{s-k}{2} \right\rfloor$ times the columns $\young(\btw1,\bon2)$ to the right of $b$. Also if $s-k$ is odd, we add the additional column $\young(1,\bon)$.
\end{itemize}
\end{dfn}

Thus from~\cite[Thm.~5.9]{OSS13} and Theorem~\ref{thm:affine_crystal_iso_Daff}, we have the following.

\begin{thm}
\label{thm:filling_map_Daff}
Let $\virtual{B}^{r,s}$ be a KR crystal of type $D_4^{(1)}$ and $\virtual{\iota} \colon \RC(\virtual{B}^{r,s}) \longrightarrow \virtual{B}^{r,s}$ be the natural crystal isomorphism. Then
\[
\virtual{\Phi} = \virtual{\fillmap} \circ \virtual{\iota}.
\]
\end{thm}

%

\section{The bijection}
\label{sec:bijection}

In this section, we describe the KSS-type bijection for type $D_4^{(3)}$.

\subsection{Algorithm $\delta$}

We define the map $\delta \colon \RC(B^{1,1} \otimes B^*) \longrightarrow \RC(B^*) \times B^{1,1}$, where $B^*$ is a tensor product of KR crystals of type $\g$, by the following algorithm. We call the entry $b \in B^{1,1}$ the \defn{return value}.

Set $\ell_0 = 1$. Do the following process for $a = 1$. Find the minimal integer $i \geq \ell_{a-1}$ such that $\nu^{(a)}$ has a singular string of length $i$. If there is no such integer, then set $b = a$ and $\ell_a = \infty$ and terminate. Otherwise set $\ell_a = i$ and repeat this for $a = 2$.

Suppose the process has not terminated. We remove the selected (singular) string of length $\ell_1$ from consideration. If there are no singular or quasi-singular strings in $\nu^{(a)}$ larger than $\ell_2$ or if $\ell_2 = \ell_1$ and there is only one string of length $\ell_1$ in $\nu^{(1)}$, then set $b = 3$ and terminate. Otherwise find the smallest $i \geq \ell_2$ that satisfies one of the following three mutually exclusive conditions:

\begin{itemize}
\item[(S)] $J^{(1,i)}$ is singular and $i > 1$;
\item[(P)] $J^{(1,i)}$ is singular and $i = 1$;
\item[(Q)] $J^{(1,i)}$ is quasi-singular.
\end{itemize}

If~(P) holds, set $b = \emptyset$, and $\ell_3 = i$ and terminate. If~(S) holds, set $\ell_3 = i-1$, $\ellbar_3 = i$, say case~(S) holds for $a = n$, and continue. If~(Q) holds, find the minimal $j > i$ such that~(S) holds. If no such $j$ exists, set $b = 0$ and terminate. Else set $\ellbar_3 = j$ and say case~(Q,~S) holds and continue.

Suppose the process has not terminated, and let $a = 2$. If $\ell_a = \ellbar_{a+1}$, then set $\ellbar_a = \ell_a$, afterwards reset $\ell_a = \ellbar_a - 1$, and say case~(S2) holds for $a$. Otherwise find the minimal index $i \geq \ellbar_{a+1}$ such that $\nu^{(a)}$ has a singular string of length $i$. If no such $i$ exists, set $b = \overline{a+1}$ and terminate. Otherwise set $\ellbar_a = i$ and repeat this for $a = 1$ (there must exist at least two singular strings if $\ellbar_3 = \ellbar_1$ and case~(S2) does not hold). If the process has not terminated, set $b = \bon$.

Set all undefined $\ell_a$ and $\ellbar_a$ for $a = 1,2,3$ to $\infty$. Note that the return value $b \in B^{1,1}$. Next we describe how the rigged configuration changes under $\delta$.

We first remove a box from $\ellbar_a$ in $\nu^{(a)}$ for $a = 1,2$, and if case~(S2) holds for $a$, we remove another box from that particular row, otherwise we remove a box from $\ell_a$. Afterwards, we will make all the changed strings singular. If case~(S) holds, then remove two boxes from $\ellbar_3$ and make the resulting string singular. If case~(Q) holds, remove a box from $\ell_3$ and make the resulting string singular. If case~(Q,~S) holds, then we remove both boxes corresponding to $\ell_3$ and $\ellbar_3$, but we make the smaller one (corresponding to $\ell_3$) singular and the larger one quasi-singular.

Recall that $p_{\infty}^{(a)} = \inner{h_a}{\wt(\nu, J)}$. We compute the change in vacancy numbers. Let $\chi(S)$ denote the function which is 1 if the statement $S$ is true and 0 if false. Let $\widetilde{p}_i^{(a)}$ denote the vacancy numbers of $\delta(\nu, J)$.
\begin{subequations}
\label{eq:change_vacancy_nums}
\begin{align}
\widetilde{p}_i^{(1)} - p_i^{(1)} & = -1 + 2 \bigl( \chi(i \geq \ell_1) + \chi(i \geq \ell_3) + \chi(i \geq \ellbar_3) + \chi(i \geq \ellbar_1) \bigr) - 3 \bigl( \chi(i \geq \ell_2) + \chi(i \geq \ellbar_2) \bigr)
\\ \widetilde{p}_i^{(2)} - p_i^{(2)} & = 2 \bigl( \chi(i \geq \ell_2) + \chi(i \geq \ellbar_2) \bigr) - \bigl( \chi(i \geq \ell_1) + \chi(i \geq \ell_3) + \chi(i \geq \ellbar_3) + \chi(i \geq \ellbar_1) \bigr)
\end{align}
\end{subequations}
We break down the change in vacancy numbers depending on $i$ in Table~\ref{table:change_vac_nums}.

\begin{table}[h]
\[
\begin{array}{|c|c|c|c|c|c|c|c|}
\hline
i & [1, \ell_1) & [\ell_1, \ell_2) & [\ell_2, \ell_3) & [\ell_3, \ellbar_3) & [\ellbar_3, \ellbar_2) & [\ellbar_2, \ellbar_1) & [\ellbar_1, \infty)
\\ \hline
a = 1 & -1 & +1 & -2 & 0 & +2 & -1 & +1
\\ \hline
a = 2 & 0 & -1 & +1 & 0 & -1 & +1 & 0
\\ \hline
\end{array}
\]
\caption{Change in $p_i^{(a)}$ from $\delta$.}
\label{table:change_vac_nums}
\end{table}

\begin{ex}
\label{ex:delta}
Consider $(\nu, J)$ from Example~\ref{ex:runningrig} with $B = B^{1,1} \otimes B^{2,1} \otimes B^{1,2}$. We then apply $\delta$ to $(\nu, J)$
\[
\begin{tikzpicture}[scale=.35,baseline=-18]
\fill[lightgray] (0,-3) rectangle (1,-2);
\fill[lightgray] (11,-2) rectangle (12,-1);
\node[scale=.75] at (.6, -2.5) {$\ell_1$};
\node[scale=.75] at (11.6, -1.5) {$\ell_2$};
 \rpp{4,1}{3,1}{5,1}
 \begin{scope}[xshift=8cm]
 \rpp{4}{-2}{-2}
 \end{scope}
\end{tikzpicture}
\]
with the selected entries shaded in. Therefore $\delta$ returns $3$ and results in
\[
\begin{tikzpicture}[scale=.35,baseline=-18]
 \rpp{4}{3}{3}
 \begin{scope}[xshift=8cm]
 \rpp{3}{-2}{-2}
 \end{scope}
\end{tikzpicture} \in \RC(B^{2,1} \otimes B^{1,2}).
\]
We continue with some more examples from $\RC(B)$. Applying $\delta$ to
\[
\begin{tikzpicture}[scale=.35,baseline=-18]
\fill[lightgray] (0,-3) rectangle (1,-2);
\fill[lightgray] (2,-2) rectangle (4,-1);
\fill[lightgray] (10,-2) rectangle (12,-1);
\node[scale=.75] at (.6, -2.5) {$\ell_1$};
\node[scale=.75] at (2.6, -1.5) {$\ellbar_3$};
\node[scale=.75] at (3.6, -1.5) {$\ell_3$};
\node[scale=.75] at (10.6, -1.5) {$\ellbar_2$};
\node[scale=.75] at (11.6, -1.5) {$\ell_2$};
 \rpp{4,1}{5,1}{5,1}
 \begin{scope}[xshift=8cm]
 \rpp{4}{-2}{-2}
 \end{scope}
\end{tikzpicture}
\]
returns $\btw$ and results in
\[
\begin{tikzpicture}[scale=.35,baseline=-18]
 \rpp{2}{4}{4}
 \begin{scope}[xshift=8cm]
 \rpp{2}{-1}{-1}
 \end{scope}
\end{tikzpicture}.
\]
Applying $\delta$ to
\[
\begin{tikzpicture}[scale=.35,baseline=-18]
\fill[lightgray] (1,-3) rectangle (2,-2);
\fill[lightgray] (0,-4) rectangle (2,-3);
\fill[lightgray] (0,-4) rectangle (1,-5);
\fill[lightgray] (9,-2) rectangle (10,-1);
\fill[lightgray] (8,-3) rectangle (9,-2);
\node[scale=.75] at (0.6, -4.5) {$\ell_1$};
\node[scale=.75] at (0.6, -3.5) {$\ellbar_3$};
\node[scale=.75] at (1.6, -3.5) {$\ell_3$};
\node[scale=.75] at (1.6, -2.5) {$\ellbar_1$};
\node[scale=.75] at (9.6, -1.5) {$\ellbar_2$};
\node[scale=.75] at (8.6, -2.5) {$\ell_2$};
 \rpp{2,2,2,1,1}{-1,-1,-1,1,-1}{-1,-1,-1,1,1}
 \begin{scope}[xshift=8cm]
 \rpp{2,1,1}{1,0,0}{1,0,0}
 \end{scope}
\end{tikzpicture}
\]
returns $\bon$ and results in
\[
\begin{tikzpicture}[scale=.35,anchor=top,baseline=-18]
 \rpp{2,1,1}{-1,1,-1}{0,1,1}
 \begin{scope}[xshift=8cm]
 \rpp{1,1}{0,0}{0,0}
 \end{scope}
\end{tikzpicture}.
\]
\end{ex}

\subsection{Inverse algorithm}

We describe the algorithm for $\delta^{-1}$ for completeness, which is derived from $\delta$ by generally taking the largest singular strings. We consider that there always exists a length 0 singular string. Explicitly $\delta^{-1}$ is given as follows for a given $b \in B$. If $b \neq \emptyset$, do the following:
\begin{itemize}
\item[$b = 1$:] Do nothing.
\item[$b = 2$:] Select the largest singular string of $\nu^{(1)}$ of length $\ell_1$.
\item[$b = 3$:] Select largest singular string of $\nu^{(2)}$ of length $\ell_2$. Select a singular string of $\nu^{(1)}$ of length $\ell_1 \leq \ell_2$ (in other words, proceed as in the case $b = 1$ with $\ell_1 \leq \ell_2$).
\item[$b = 0$:] Select the largest singular string of $\nu^{(1)}$ of length $\ell_3$. Proceed as in the case $b = 2$ with $\ell_2 \leq \ell_3$ except $\ell_1$ must select a different string than $\ell_3$.
\item[$b = \bth$:] Find either the largest singular or quasisingular string in $\nu^{(1)}$ of length $\ellbar_3$. Select a singular string of $\nu^{(1)}$ of length $\ell_3 \leq \ellbar_3$. Proceed in the case $b = 3$ with $\ell_2 \leq \ell_3$ except if $\ell_1 = \ell_3$, then $\ell_1$ must select a different string than $\ell_3$.
\item[$b = \btw$:] Select the largest singular string of $\nu^{(2)}$ of length $\ellbar_2$ and proceed in the case $b = \bth$ with $\ellbar_3 \leq \ellbar_2$.
\item[$b = \bon$:] Select a singular string of $\nu^{(1)}$ of length $\ellbar_1$ and proceed in the case $b = \btw$ with $\ellbar_2 \leq \ellbar_1$ except $\ell_3$ and $\ellbar_3$ must be different strings that $\ellbar_1$ and/or $\ell_1$ (but $\ellbar_1$ and $\ell_1$ could possibly be the same string).
\end{itemize}
Now add a box to the strings corresponding to $\ell_1, \ell_2, \ell_3, \ellbar_3, \ellbar_2, \ellbar_1$ and make all string singular unless $\ell_3 < \ellbar_3$, in which case make the string corresponding to $\ell_3$ quasisingular.
For $b = \emptyset$, add two singular strings of length 1 to $\nu^{(1)}$, and a singular string of length 1 to $\nu^{(2)}$.

\subsection{Extending to Arbitrary Tensor Factors}

We now extend the process to arbitrary shapes $B^{r,s}$ by defining the maps
\begin{align*}
\ls & \colon \RC(B^{r,s} \otimes B^*) \longrightarrow \RC(B^{r,1} \otimes B^{r,s-1} \otimes B^*),
\\ \lt & \colon \RC(B^{2,1} \otimes B^*) \longrightarrow \RC(B^{1,1} \otimes B^{1,1} \otimes B^*),
\end{align*}
which are known as \defn{left-split} and \defn{left-box} respectively. On the rigged configurations, the map $\ls$ is the identity (but perhaps increases the vacancy numbers) and $\lt$ adds a singular string of length 1 to $\nu^{(1)}$. Thus it is clear the map $\ls$ is a strict crystal embedding. Furthermore it is easy to see that $\lt$ preserves the vacancy numbers, so $\lt$ is a strict crystal embedding as well.

We abuse notation and denote by $\ls$ and $\lt$ on $B$ the map which splits the left-most column of $B$ and the map which moves off the bottom box of the left-most $B^{2,1}$ respectively. In addition, we define $\delta'$ for the leftmost factor $B^{r,1}$ by starting the algorithm for $\delta$ at $a = r$ (or at $\nu^{(r)}$), then it is clear that $\delta' = \delta \circ \lt$. For simplicity, we will use $\delta'$ and write this as $\delta$ if there is no danger of confusion for the remainder of the paper.

We can now define the map $\Phi \colon \RC(B) \longrightarrow B$, where $B$ is a tensor product of KR crystals of type $\g$. We apply a sequence of the maps $\delta, \ls, \lt$ on a rigged configuration $(\nu, J)$ resulting in an element in $\RC(\emptyset)$. Then to construct the element in $B$, we apply the inverse sequence starting with the empty tensor product of KR crystals. The resulting element in $B$ is $\Phi(\nu, J)$.


Note that the resulting tableaux under $\Phi$ are always fully rectangular, as opposed to using Kang--Misra tableaux used to realize $B^{r,s}$. The resulting tableaux under of $\Phi$ of $B^{r,s}$ are the so-called \defn{Kirillov--Reshetikhin (KR) tableaux} $T^{r,s}$. The KR tableaux are related to the Kang--Misra tableaux by the filling map $\fillmap \colon B^{r,s} \longrightarrow T^{r,s}$, which we will describe in Section~\ref{sec:filling_map}. We extend the filling map to arbitrary tensor factors by
\[
\fillmap(B_1 \otimes \cdots B_N) = \fillmap(B_1) \otimes \cdots \fillmap(B_n)
\]
to give a KR tableaux representation of a tensor product of KR crystals.

The following conjecture is a special case of a conjecture given in~\cite{SS15}.

\begin{conj}
\label{conj:bijection}
Let $B = \bigotimes_{i=1}^N B^{r_i,s_i}$ be a tensor product of KR crystals of type $D_4^{(3)}$. The map $\Phi \colon \RC(B) \longrightarrow B$ is a bijection. In addition, $\Phi \circ \eta$ sends cocharge to energy, where $\eta$ is the complement map given in Definition~\ref{def:complement_rigging}.
\end{conj}

\begin{ex}
\label{ex:Phi}
Consider $(\nu, J)$ from Example~\ref{ex:runningrig} with $B = B^{1,1} \otimes B^{2,1} \otimes B^{1,2}$. Performing $\Phi$, we have
\[
\begin{tikzpicture}[scale=.35,baseline=-18]
 \rpp{4,1}{3,1}{5,1}
 \begin{scope}[xshift=8cm]
 \rpp{4}{-2}{-2}
 \end{scope}
\draw[->] (5,-3cm) -- (5,-6cm) node[midway,right] {$\delta$};
\draw (10,-4.5cm) node {(returns $3$)};
\begin{scope}[yshift=-6cm]
\rpp{4}{3}{3}
\begin{scope}[xshift=8cm]
\rpp{3}{-2}{-2}
\end{scope}
\end{scope} 
\draw[->] (5,-9cm) -- (5,-12cm) node[midway,right] {$\lt$};
\begin{scope}[yshift=-12cm]
\rpp{2,1}{5,2}{5,2}
\begin{scope}[xshift=8cm]
\rpp{2}{-2}{-2}
\end{scope}
\end{scope} 
\draw[->] (5,-15cm) -- (5,-18cm) node[midway,right] {$\delta$};
\draw (10,-16.5cm) node {(returns $\bth$)};
\begin{scope}[yshift=-18cm]
\rpp{2}{5}{5}
\begin{scope}[xshift=8cm]
\rpp{2}{-2}{-2}
\end{scope}
\end{scope} 
\draw[->] (5,-21cm) -- (5,-24cm) node[midway,right] {$\delta$};
\draw (10,-22.5cm) node {(returns $3$)};
\begin{scope}[yshift=-24cm]
\rpp{1}{2}{2}
\begin{scope}[xshift=8cm]
\rpp{1}{-1}{-1}
\end{scope}
\end{scope} 
\draw[->] (5,-27cm) -- (5,-30cm) node[midway,right] {$\ls$};
\begin{scope}[yshift=-30cm]
\rpp{1}{2}{3}
\begin{scope}[xshift=8cm]
\rpp{1}{-1}{-1}
\end{scope}
\end{scope} 
\draw[->] (5,-33cm) -- (5,-36cm) node[midway,right] {$\delta$};
\draw (10,-34.5cm) node {(returns $1$)};
\begin{scope}[yshift=-36cm]
\rpp{1}{2}{2}
\begin{scope}[xshift=8cm]
\rpp{1}{-1}{-1}
\end{scope}
\end{scope} 
\draw[->] (5,-39cm) -- (5,-42cm) node[midway,right] {$\delta$};
\draw (10,-40.5cm) node {(returns $3$)};
\begin{scope}[yshift=-42cm]
\node at (1, 0) {$\emptyset$};
\begin{scope}[xshift=8cm]
\node at (1, 0) {$\emptyset$};
\end{scope}
\end{scope} 
\end{tikzpicture}
\]
and the resulting element in $B$ is
\[
\young(3) \otimes \young(3,\bth) \otimes \young(13).
\]
\end{ex}

Note that Conjecture~\ref{conj:bijection} implies the $X = M$ conjecture by a constructing a bijection on classically highest weight elements. The following is a special case of Conjecture~6.3 in~\cite{SS15} (which is an extension of Conjecture~7.2 in~\cite{OSS03III}).

\begin{conj}
\label{conj:virtualization}
Consider $B = \bigotimes_{i=1}^N B^{r_i,s_i}$. The virtualization map $v$ commutes with the bijection $\Phi$.
\end{conj}

Lastly, we restate a special case of Conjecture~2.12 in~\cite{SS15} for $\g$ of type $D_4^{(3)}$.

\begin{conj}
\label{conj:isomorphism}
Consider $B = \bigotimes_{k=1}^N B^{r_k, s_k}$ of type $D_4^{(3)}$. There exists an affine crystal isomorphism $\Phi \colon \RC(B) \longrightarrow B$.
\end{conj}

\section{Filling Map}
\label{sec:filling_map}

In this section, we describe the KR tableaux for type $D_4^{(3)}$ and the associated filling map from $B^{r,s}$. We begin by determining the filling map by describing the highest weight rigged configurations, and we do so in two parts. The first part is the KR tableaux for $B^{1,s}$, which is ``easy'' because $\phi^{-1}(1) = \{2\}$. However, the second part dealing with the case for $B^{2,s}$ is ``hard'' because $\phi^{-1}(2) = \{1,3,4\}$.

\subsection{The easy case: $r = 1$}

We consider the case $r = 1$.


\begin{lemma}
\label{lemma:hw_1}
Consider the KR crystal $B^{1,s}$. We have
\[
\RC(B^{1,s}) = \bigoplus_{k=0}^s \RC(B^{1,s}; k \clfw_1).
\]
Moreover the highest weight rigged configurations in $\RC(B^{1,s}; k \clfw_1)$ are given by
\begin{align*}
\nu^{(1)} & = (s-k, s-k),
\\ \nu^{(2)} & = (s-k),
\end{align*}
with all riggings equal to $0$.
\end{lemma}

\begin{proof}
Follows from the virtual Kleber algorithm, Equation~\eqref{eq:RC_virtualization}, and the known type $D_4^{(1)}$ tree structure for $\virtual{B}^{2,s}$~\cite{Kleber98,SS15}.
\end{proof}

Using the notation of~\cite{OSS13,SS15} and considering $\clfw_1$ as being a column of height 2 and $\overline{\lambda}$ as being the complement shape in a $2 \times s$ box, we have $\nu^{(1)} = \overline{\lambda}$ and $\nu^{(2)} = \overline{\lambda}^{[1]}$.

\begin{dfn}
\label{def:filling_map_1}
Let $B^{1,s}$ be a KR crystal of type $D_4^{(3)}$ and consider the classical component $B(k\clfw_1) \subseteq B^{1,s}$. The filling map for $B^{1,s}$ is given by adding $\left\lfloor \frac{s - k}{2} \right\rfloor$ copies of the horizontal domino $\young(\bon1)$ and an additional $\young(\emptyset)$ if $s - k$ is odd.
\end{dfn}

\begin{ex}
Consider the element
\[
b = 
\begin{tikzpicture}[baseline=-3pt]
\matrix [matrix of math nodes,column sep=-.4, row sep=-.5,text height=8,text width=8,align=center,inner sep=2] 
{
\node[draw]{3};
& \node[draw]{0};
& \node[draw]{\btw};
& \node[draw]{\btw};
& \node[draw]{\bon}; \\
};
\end{tikzpicture}
\in B(5\clfw_1) \subseteq B^{1,9},
\]
then we have
\[
\fillmap(b) = 
\begin{tikzpicture}[baseline=-3pt]
\matrix [matrix of math nodes,column sep=-.4, row sep=-.5,text height=8,text width=8,align=center,inner sep=2] 
{
\node[draw]{3};
& \node[draw]{0};
& \node[draw]{\btw};
& \node[draw]{\btw};
& \node[draw]{\bon};
& \node[draw,fill=gray!30]{\bon};
& \node[draw,fill=gray!30]{1};
& \node[draw,fill=gray!30]{\bon};
& \node[draw,fill=gray!30]{1}; \\
};
\end{tikzpicture}.
\]
Now suppose $b \in B^{1,8}$, then we have
\[
\fillmap(b) = 
\begin{tikzpicture}[baseline=-3pt]
\matrix [matrix of math nodes,column sep=-.4,row sep=-.5,text height=8,text width=8,align=center,inner sep=2,nodes={anchor=south,minimum height=13}]
{
\node[draw]{3};
& \node[draw]{0};
& \node[draw]{\btw};
& \node[draw]{\btw};
& \node[draw]{\bon};
& \node[draw,fill=gray!30]{\bon};
& \node[draw,fill=gray!30]{1};
& \node[draw,fill=gray!30]{\emptyset}; \\
};
\end{tikzpicture}.
\]
\end{ex}

We can give an affine crystal structure to $T^{1,s}$ by using the ``coordinate representation'' and $e_0$ and $f_0$ defined in Section~\ref{sec:background} and Equation~\eqref{eq:affine_f}.

\begin{prop}
\label{prop:filling_map_isomorphism}
The filling map given in Definition~\ref{def:filling_map_1} is a $U_q'(\g)$-crystal isomorphism.
\end{prop}

\begin{proof}
It is clear the filling map commutes with the classical crystal operators. It is also clear that the filling map does not change the validity of conditions $(E_i)$ and $(F_i)$ used to define $e_0$ and $f_0$ as $z_i$ is invariant. Therefore the virtualization map commutes with $e_0$ and $f_0$ and the claim follows.
\end{proof}

\subsection{The hard case: $r = 2$}

We consider the case $r = 2$.

\begin{lemma}
\label{lemma:hw_2}
Consider the KR crystal $B^{2,s}$. We have
\[
\RC(B^{2,s}) = \bigoplus_{\lambda} \RC(B^{2,s}; \lambda)
\]
where $\lambda$ runs over all weights of the form
\begin{align*}
& s \clfw_2 - k_1 (3\alpha_1 + 2 \alpha_2) - k_2 (\alpha_1 + \alpha_2) - k_3 \alpha_1
\\ = & s \clfw_2 - k_1 (\clfw_2) - k_2 ( \clfw_2 - \clfw_1 ) - k_3 (2\clfw_1 - \clfw_2)
\\ = & (k_2 - 2k_3) \clfw_1 + (s - k_1 - k_2 + k_3) \clfw_2
\end{align*}
with $2 k_3 \leq k_2$ and $k_1 + k_2 \leq s$ for $k_1,k_2,k_3 \in \ZZ_{\geq 0}$. Moreover the highest weight rigged configurations in $\RC(B^{2,s})$ are given by
\begin{align*}
\nu^{(1)} & = (k_1 + k_2 + k_3, k_1, k_1),
\\ \nu^{(2)} & = (k_1 + k_2, k_1),
\end{align*}
and the multiplicity of the node is equal to $1 + k_2 - 2 k_3$. In particular, let $k = k_1 + k_2 + k_3$, then $p_k^{(1)} = k_2 - 2 k_3$ and $p_i^{(a)} = 0$ for all other $(a, i) \in \HH_0$.
\end{lemma}

\begin{proof}
We begin by stating that the only elements in $\virtual{Q}^+_0$ which give us symmetry in $\virtual{\Lambda}_1, \virtual{\Lambda}_3, \virtual{\Lambda}_4$ are those in the positive span of $\virtual{\alpha}_2$ and $\virtual{\alpha}^{(2)} := \virtual{\alpha}_1 + \virtual{\alpha}_2 + \virtual{\alpha}_3 + \virtual{\alpha}_4$. This symmetry is required by~(V2) of Definition~\ref{def:virtual_kleber}. Let $\Lambda = \virtual{\Lambda}_1 + \virtual{\Lambda}_3 + \virtual{\Lambda}_4$.

For the virtual Kleber tree, we start with $x = \Lambda$, and the only roots we can subtract are $\virtual{\alpha}^{(2)}$ and $\virtual{\alpha}^{(1)} := 2 \virtual{\alpha}_1 + 3 \virtual{\alpha}_2 + 2 \virtual{\alpha}_3 + 2 \virtual{\alpha}_4$. Thus the only children of $x$ are $x^{(1)} := x - \virtual{\alpha}^{(1)} = 0$ and $x^{(2)} := x - \virtual{\alpha}^{(2)}  = \virtual{\Lambda}_2$. Hence in the next step when we add $\Lambda$ to each node in the Kleber tree, the child $x^{(1)}$ gives a recursive structure, i.e. its subtree is that given by $\virtual{T}(B^{2,s-1})$, since $\virtual{\alpha}_2 \leq \virtual{\alpha}^{(2)} \leq \virtual{\alpha}^{(1)}$ component-wise. Thus the resulting subtree starting at $x^{(1)}$ can be shown by a straightforward induction on $s$.

Next if we consider the $x^{(2)}$ child, then there is only one possible node to add and is given by $y = x^{(2)} - \virtual{\alpha}^{(2)}$. Thus after adding $\Lambda$ to $y$, we have a node of weight $\Lambda + \virtual{\Lambda}_2$. So we can either subtract $\virtual{\alpha}^{(2)}$ or $\virtual{\alpha}_2$. If we subtract $\virtual{\alpha}_2$, this will be a leaf because all subsequent additions will not increase $\virtual{\Lambda}_2$. Hence we cannot subtract $\virtual{\alpha}_2$. If we subtract $\virtual{\alpha}^{(2)}$, we are in a similar situation as the node $y$.

Thus in general each node must be of the form:
\[
s \Lambda - k_1 \virtual{\alpha}^{(1)} - k_2 \virtual{\alpha}^{(2)} - k_3 \virtual{\alpha}_2
\]
with $2 k_3 \leq k_2$ and $k_2 + k_1 \leq s$. The paths from the root have edges labeled by $\virtual{\alpha}^{(1)}$, $\virtual{\alpha}^{(2)}$, and then $\virtual{\alpha}_2$. The claim follows from the definition the Kleber algorithm and devirtualization.
\end{proof}

\begin{remark}
\label{remark:dependency}
Let $\alpha^{(1)} = 3\clsr_1 + 2\clsr_2$ and $\alpha^{(2)} = \clsr_1 + \clsr_2$. The linear dependency $\alpha^{(1)} - 2\alpha^{(2)} - \clsr_1 = 0$ implies the weights are not uniquely determined by $(k_1, k_2, k_3)$. For example, we have $\alpha^{(1)} = 2\alpha^{(2)} + \clsr_1 = \clfw_2$ and both such nodes appear in $B^{2,2}$.
\end{remark}

Next we show the our classical decomposition agrees with that given in Theorem~\ref{thm:decomposition_2}, giving further evidence that every KR module admits a crystal basis.

\begin{prop}
\label{prop:decomposition_multiplicity}
As $U_q(\g_0)$-crystals, we have
\[
\RC(B^{2,s}) \iso \bigoplus_{\substack{m_1 + m_2 = s \\ m_1,m_2 \geq 0}} (m_1 + 1) \min(1 + m_2, 1 + s - m_1 - m_2) \RC(B^{2,s}; m_1 \clfw_1 + m_2 \clfw_2).
\]
\end{prop}

\begin{proof}
We will show the number of nodes of weight $\lambda = m_1 \clfw_1 + m_2 \clfw_2$ in the Kleber tree equals $\min(1 + m_2, 1 + s - m_1 - m_2)$. Fix some $m_1, m_2 \geq 0$ such that $m_1 + m_2 \leq s$. By the weight of a node stated in Lemma~\ref{lemma:hw_2}, we must have
\begin{subequations}
\label{eq:weight_coeffs}
\begin{align}
m_1 & = k_2 - 2 k_3,
\\ m_2 & = s - k_1 - k_2 + k_3,
\end{align}
\end{subequations}
for some $k_1, k_2, k_3 \geq 0$.
Now fix some $k_3 \geq 0$. Thus we have $k_2 = m_1 + 2k_3$, and hence
\[
k_1 = s - k_2 + k_3 - m_2 = s - m_1 - m_2 - k_3,
\]
giving the upper bound $k_3 \leq s - m_1 - m_2$ since $k_1 \geq 0$. We also have another upper bound on the choice of $k_3$ given by
\[
s \geq k_1 + k_2 = s - m_1 - m_2 - k_3 + m_1 + 2k_3 = s - m_2 + k_3,
\]
which implies $0 \geq k_3 - m_2$ or equivalently $m_2 \geq k_3$.

We also note that if we increase $k_3$ by 1, then $k_1$ decreases by 1 and $k_2$ increases by 2 and we still satisfy Equation~\eqref{eq:weight_coeffs}. Thus we are free to choose $0 \leq k_3 \leq \min(m_2, s - m_1 - m_2)$ as $k_3 \geq 0$ implies $k_2 \geq 0$, giving a total of $\min(1 + m_2, 1 + s - m_1 - m_2)$ nodes of weight $\lambda$ in the Kleber tree.

Now we note that every node of weight $\lambda$ occurs with multiplicity $1 + m_1 = 1 + k_2 - 2k_3$. Thus the claim follows.
\end{proof}

We refine the parameterization of $\hwRC(B^{2,s})$ given in Lemma~\ref{lemma:hw_2}, so that $k_4$ is the rigging on the largest row of $(\nu, J)^{(1)}$. Thus we must have $0 \leq k_4 \leq k_2 - 2k_3$, and hence the tuple $(k_1, k_2, k_3, k_4) \in \ZZ_{\geq 0}^4$ uniquely determines a $(\nu, J) \in \hwRC(B^{2,s})$. Moreover, we can explicitly compute the cocharge on classically highest weight elements.

\begin{prop}
Let $(\nu, J) \in \RC(B^{2,s})$, then we have
\[
\cc(\nu, J) = 3k_1 + k_2 + k_3 + k_4 = \lvert \nu^{(1)} \rvert + k_4.
\]
\end{prop}

\begin{proof}
A straightforward computation using the description of $(\nu, J)$ given in Lemma~\ref{lemma:hw_2} and the definition of cocharge given by Equation~\eqref{eq:cocharge}.
\end{proof}

We also have that $\RC(B^{2,s})$ graded by $\cc$ satisfies the grading conditions given in Theorem~\ref{thm:grading}.

\begin{prop}
\label{prop:affine_grading}
There exists a bijection $\Psi \colon \hwRC(B^{2,s}) \longrightarrow \mcA$ such that
\begin{align*}
\cc(\nu, J) & = \gr\bigl( \Psi(\nu, J) \bigr),
\\ \overline{\wt}(\nu, J) & = \overline{\wt}\bigl( \Psi(\nu, J) \bigr).
\end{align*}
\end{prop}

\begin{proof}
Let $\Psi$ be defined by
\begin{align*}
r_1 & = k_2 - k_3 - k_4,
\\ r_2 & = k_4,
\\ r_3 & = k_3,
\\ r_4 & = k_1,
\end{align*}
which is an invertible transformation with inverse given by
\begin{align*}
k_1 & = r_4,
\\ k_2 & = r_1 + r_2 + r_3,
\\ k_3 & = r_3,
\\ k_4 & = r_2.
\end{align*}
Thus we have
\begin{align*}
r_1 + r_2 - r_3 & = (k_2 - k_3 - k_4) + k_4 - k_3 = k_2 - 2k_3,
\\ s - r_1 - r_2 - r_4 & = s - (k_2 - k_3 - k_4) - k_4 - k_1 = s - k_1 - k_2 + k_3,
\\ r_1 + 2r_2 + 2r_3 + 3r_4 & = (k_2 - k_3 - k_4) + 2k_4 + 2k_3 + 3k_1 = 3k_1 + k_2 + k_3 + k_4,
\end{align*}
so the weights are preserved and $\cc$ goes to $\gr$. Also, it is straightforward to see that
\[
r_1 + r_2 + r_3 + r_4 = k_1 + k_2,
\]
so $r_1 + r_2 + r_3 + r_4 \leq s$ is equivalent to $k_1 + k_2 \leq s$. Also the fact that $0 \leq k_1, k_3, k_4$ is equivalent to $0 \leq r_2, r_3, r_4$. Next the following are equivalent:
\begin{align*}
r_3 & \leq r_1,
\\ k_3 & \leq k_2 - k_3 - k_4,
\\ 2k_3 + k_4 & \leq k_2,
\\ k_4 & \leq k_2 - 2k_3.
\end{align*}
It is clear that $r \in \mcA$ implies $k_2 \geq 0$. The last condition to verify is that $(\nu, J) \in \hwRC(B^{2,s})$ implies $r_1 \geq 0$. However this follows from the fact that
\[
k_4 + k_3 \leq k_4 + 2k_3 \leq k_2
\]
since $k_3 \geq 0$ and $k_4 \leq k_2 - 2k_3$.
\end{proof}

\begin{dfn}
\label{def:filling_map_2}
The crystal morphism $\fillmap \colon B^{r,s} \longrightarrow T^{r,s}$ is given by the filling procedure below on highest weight elements and extended as a crystal morphism. Consider $(k_1, k_2, k_3, x)$ where $0 \leq x \leq k_2 - 2 k_3$. Let $x^c$ denote the colabel of $x$. The image of $\fillmap$ on classically highest weight elements is given by the following algorithm.
\begin{enumerate}[(1)]
\item If $s = 1$, then fill the column by
\[
(0, 0, 0, 0) \mapsto \young(1,2)\, ,
\hspace{30pt} (0, 1, 0, 0) \mapsto \young(1,0)\, ,
\hspace{30pt} (0, 1, 0, 1) \mapsto \young(1,\emptyset)\, ,
\hspace{30pt} (1, 0, 0, 0) \mapsto \young(\emptyset,\emptyset)\, ,
\]
and terminate.

\item Fill the first $s - k_1 - k_2$ columns with $\young(1,2)$. Redefine $s = k_1 + k_2$.

\item If $k_2 = 0$ (so $k_1 = s$), then do the following.
\begin{enumerate}[(i)]
\item If $k_1 = 2$, then fill the remaining with $\young(01,\bon\emptyset)$ and terminate.

\item If $k_1 > 2$, then fill the leftmost unfilled column with $\young(\btw,\bon)$ and recurse with $(k_1-3, 2, 1, 0)$ with $B^{2,s-1}$.
\end{enumerate}

\item If $x^c > 2$, then fill the leftmost unfilled column with $\young(1,3)$ and fill the remaining columns by recursion on $(k_1, k_2 - 1, k_3 + 1, x)$ with $B^{2,s-1}$.

\item If $x^c = 2$, then fill the leftmost unfilled column with $\young(2,3)$ and fill the remaining columns by recursion on $(k_1, k_2 - 1, k_3, k_2 - 2k_3 - 1)$ with $B^{2,s-1}$.

\item If $x^c = 1$, then do the following.
\begin{enumerate}[(i)]
\item If $k_3 > 0$, then fill the leftmost unfilled column with $\young(2,0)$ and fill the remaining columns by recursion on $(k_1, k_2 - 1, k_3 - 1, k_2 - 2k_3 + 1)$ with $B^{2,s-1}$.

\item If $k_3 = 0$ and $k_2 > 1$, then fill the leftmost unfilled column with $\young(3,0)$ and fill the remaining columns by recursion on $(k_1, k_2 - 2, k_3, k_2 - 2k_3 - 2)$ with $B^{2,s-1}$.

\item If $k_2 = 1$ (so $k_1 \geq 1$), then fill the leftmost unfilled column with $\young(0,0)$ and fill the remaining columns by recursion on $(k_1-1, 1, 0, 0)$.
\end{enumerate}

\item If $x^c = 0$, then do the following.
\begin{enumerate}[(i)]
\item If $k_3 > 1$, then fill the leftmost unfilled column with $\young(2,\bth)$ and fill the remaining columns by recursion on $(k_1, k_2 - 1, k_3 - 2, k_2 - 2k_3 + 3)$ with $B^{2,s-1}$.

\item If $k_3 = 1$, then fill the leftmost unfilled column with $\young(3,\bth)$ and fill the remaining columns by recursion on $(k_1, k_2 - 2, k_3 - 1, k_2 - 2k_3)$ with $B^{2,s-1}$.

\item If $k_3 = 0$ and $k_2 > 2$, then fill the leftmost unfilled column with $\young(3,\btw)$ and fill the remaining columns by recursion on $(k_1, k_2 - 3, k_3, k_2 - 2k_3 - 3)$ with $B^{2,s-1}$.

\item If $k_3 = 0$, $k_2 = 2$, and $k_1 > 0$, then fill the leftmost unfilled column with $\young(0,\btw)$ and fill the remaining columns by recursion on $(k_1-1, 1, 0, 0)$ with $B^{2,s-1}$.

\item If $k_3 = 0$ and $k_2 = 2 = s$ (so $k_1 = 0$), then fill the remaining with $\young(11,\btw2)$ and terminate.

\item If $k_3 = 0$, $k_2 = 1$, and $k_1 > 1$, then fill the leftmost unfilled column with $\young(\bth,\btw)$ and fill the remaining columns by recursion on $(k_1-2, 2, 0, 0)$ with $B^{2,s-1}$.

\item If $k_3 = 0$, $k_2 = 1$, and $k_1 = 1$ (so $s = 2$), then fill the remaining with $\young(21,\btw0)$ and terminate.
\end{enumerate}
\end{enumerate}
\end{dfn}

\begin{ex}
We write $(k_1, k_2, k_3, x)$ where $x$ is the rigging on the top row of $\nu^{(1)}$ (we consider it to be 0 if $\nu^{(1)} = \emptyset$). The filling map on classically highest weight elements of $B^{2,1}$ is
\[
(0, 0, 0, 0) \mapsto \young(1,2)
\hspace{30pt} (0, 1, 0, 0) \mapsto \young(1,0)
\hspace{30pt} (0, 1, 0, 1) \mapsto \young(1,\emptyset)
\hspace{30pt} (1, 0, 0, 0) \mapsto \young(\emptyset,\emptyset)
\]

The filling map on classically highest weight elements of $B^{2,2}$ is
\[
\begin{array}{cccccc}
(0, 0, 0, 0) \mapsto \young(11,22)
& (0, 1, 0, 0) \mapsto \young(11,20)
& (0, 1, 0, 1) \mapsto \young(11,2\emptyset)
& (1, 0, 0, 0) \mapsto \young(1\emptyset,2\emptyset)
\vspace{10pt} \\
(0, 2, 1, 0) \mapsto \young(31,\bth2)
& (2, 0, 0, 0) \mapsto \young(01,\bon\emptyset)
& (0, 2, 0, 0) \mapsto \young(21,3\emptyset)
& (0, 2, 0, 1) \mapsto \young(31,02)
\vspace{10pt} \\
(0, 2, 0, 2) \mapsto \young(11,\btw2)
& (1, 1, 0, 0) \mapsto \young(01,00)
& (1, 1, 0, 1) \mapsto \young(21,\btw0)
\end{array}
\]
\end{ex}

Thus we have the following.

\begin{prop}
\label{prop:filling_map_2}
Let $B^{2,s}$ be a KR crystal of type $D_4^{(3)}$. Then there exists a natural crystal isomorphism $\iota \colon \RC(B^{2,s}) \longrightarrow B^{2,s}$ such that
\[
\Phi = \fillmap \circ \iota
\]
on classically highest weight elements with $\fillmap$ as in Definition~\ref{def:filling_map_2}.
\end{prop}

\begin{proof}
We proceed by induction on $s$. Consider a highest weight rigged configuration corresponding to $(k_1, k_2, k_3, x)$. Let $x^c$ denote the colabel of $x$. Let $(\nu_{\delta}, J_{\delta}) = \delta(\nu, J)$. Let $i = k_2 + k_3$, and so $p_i^{(1)} = k_2 - 2k_3$.
\begin{enumerate}[(1)]
\item Suppose $s = 1$, then this is a finite computation.

\item Suppose $s > k_1 + k_2$. Applying $\ls$ increases $p_i^{(2)}$ for all $i < s$ by $1$. We note that $\max \nu^{(2)} < s$, so all strings become non-singular and $\delta$ returns $2$. Now all strings on $(\nu_{\delta}, J_{\delta})^{(1)}$ are non-singular since $\max \nu^{(1)} < s$, so applying $\delta$ returns $1$.

\item Suppose $k_1 = s$ and $k_2 = 0$.
\begin{enumerate}[(i)]
\item Suppose $k_1 = 2$, then this is a finite computation.

\item Suppose $k_1 > 2$, then $\delta(\nu, J)$ returns $\bon$ and we have $\nu_{\delta}^{(1)} = (s, s-1, s-2)$ and $\nu_{\delta}^{(2)} = (s, s-2)$ with the largest string of $(\nu_{\delta}, J_{\delta})^{(1)}$ quasisingular and all other strings singular. Thus $\delta(\nu_{\delta}, J_{\delta})$ returns $\btw$. At this point we are in $B^{2,s-1}$ with $(k_1 - 2, 2, 1, 0)$.
\end{enumerate}

\item Suppose $x^c > 2$. Then $\delta(\nu, J)$ selects the singular string in $\nu^{(2)}$ and returns $3$ since there are no (quasi)singular strings of length at least $s$ in $\nu^{(1)}$. Next $\delta(\nu_{\delta}, J_{\delta})$ returns $1$ since there are no singular strings in $(\nu_{\delta}, J_{\delta})^{(1)}$ because $p_i^{(1)}(\nu_{\delta}, J_{\delta}) = p_i^{(1)}(\nu, J) - 2$ and our assumption. At this point we are now in $B^{2,s-1}$ with $(k_1, k_2 - 1, k_3 + 1, x)$.

\item Suppose $x^c = 2$. As in the previous case, $\delta(\nu, J)$ returns $3$, but in this case $(\nu_{\delta}, J_{\delta})^{(1)}$ has a singular string. Therefore $\delta(\nu_{\delta}, J_{\delta})$ returns $2$ since $(\nu_{\delta}, J_{\delta})^{(2)}$ only has a row of length $k_2 - 1 < k_2 + k_3$. At this point we are now in $B^{2,s-1}$ with $(k_1, k_2 - 1, k_3, k_2 - 2k_3 - 1)$.

\item Suppose $x^c = 1$.
\begin{enumerate}[(i)]
\item Suppose $k_3 > 0$. Then $\delta(\nu, J)$ selects the singular string in $(\nu, J)^{(2)}$, the quasisingular string in $(\nu, J)^{(1)}$, and returns $0$ since there are no singular strings in $(\nu, J)^{(1)}$. Then $\delta(\nu_{\delta}, J_{\delta})$ returns $2$ because it selects the singular string in $(\nu_{\delta}, J_{\delta})^{(1)}$, which is longer than the singular string in $(\nu_{\delta}, J_{\delta})^{(2)}$. At this point we are in $B^{2,s-1}$ with $(k_1, k_2 - 1, k_3 - 1, k_2 - 2k_3 + 1)$.

\item Suppose $k_3 = 0$ and $k_2 > 1$. Then $\delta(\nu, J)$ returns $0$ as in the previous case. In contrast, $(\nu_{\delta}, J_{\delta})^{(2)}$ has a singular string of the same length as in $(\nu_{\delta}, J_{\delta})^{(1)}$ and it gets selected. However there is not a different singular string in $(\nu_{\delta}, J_{\delta})^{(1)}$ of at least the same length, so $\delta(\nu_{\delta}, J_{\delta})$ returns $3$. At this point we are in $B^{2,s-1}$ with $(k_1, k_2 - 2, k_3, k_2 - 2k_3 - 2)$.

\item Suppose $k_2 = 1$ (so $k_1 > 1$). Then $\delta(\nu, J)$ returns $0$ as in the previous case. In this case, we have $\nu_{\delta}^{(1)} = (k_1, k_1, k_1)$ with 1 singular string and the other two strings quasisingular and $\nu_{\delta}^{(2)} = (k_1, k_1)$ and both strings singular. Thus $\delta(\nu_{\delta}, J_{\delta})$ returns $0$ since there is not a second singular string in $(\nu_{\delta}, J_{\delta})^{(1)}$. At this point we are in $B^{2,s-1}$ with $(k_1-1, 1, 0, 0)$.
\end{enumerate}

\item Suppose $x^c = 0$.
\begin{enumerate}[(i)]
\item Suppose $k_3 > 1$. Then $\delta(\nu, J)$ returns $\bth$ since there is a unique singular string in $(\nu, J)^{(2)}$ whose length is strictly larger than any singular string in $(\nu, J)^{(1)}$ except the longest string (which is singular). Next $\delta(\nu_{\delta}, J_{\delta})$ returns $2$ since the singular string of $(\nu_{\delta}, J_{\delta})^{(1)}$ is longer than that of $(\nu_{\delta}, J_{\delta})^{(2)}$ as all strings in $(\nu_{\delta}, J_{\delta})^{(1)}$ of length 1 are not singular. At this point we are not in $B^{2,s-1}$ with $(k_1, k_2 - 1, k_3 - 2, k_2 - 2k_3 + 3)$.

\item Suppose $k_3 = 1$. Then $\delta(\nu, J)$ returns $\bth$ as in the previous case. Next $\delta(\nu_{\delta}, J_{\delta})$ returns $3$ since there are unique singular strings in $(\nu_{\delta}, J_{\delta})^{(1)}$ and $(\nu_{\delta}, J_{\delta})^{(2)}$ of the same length. At this point we are in $B^{2,s-1}$ with $(k_1, k_2 - 2, k_3 - 1, k_2 - 2k_3)$.

\item Suppose $k_2 > 2$ and $k_3 = 0$. Then $\delta(\nu, J)$ returns $\btw$ since there are unique singular strings of length $i$ in $(\nu, J)^{(1)}$ and $(\nu, J)^{(2)}$. Next $\delta(\nu_{\delta}, J_{\delta})$ returns $3$ since there are unique singular strings in $(\nu_{\delta}, J_{\delta})^{(1)}$ and $(\nu_{\delta}, J_{\delta})^{(2)}$ of the same length. At this point we are in $B^{2,s-1}$ with $(k_1, k_2 - 3, k_3, k_2 - 2k_3 - 3)$.

\item Suppose $k_3 = 0$ and $k_2 = 2$ and $k_1 > 0$. The first $\delta$ returns $\btw$ similar to the above cases. Now $\delta(\nu_{\delta}, J_{\delta})$ returns $0$ since $\nu_{\delta}^{(1)}$ contains both a unique singular string and two quasi-singular strings of length $s-2$, and $(\nu_{\delta}, J_{\delta})^{(2)}$ contains a singular string of length $s-2$. At this point we are in $B^{2,s-1}$ with $(k_1 - 1, 1, 0, 0)$.

\item Suppose $k_2 = 2 = s$ (so $k_1 = 0$), and $k_3 = 0$. This is a finite computation.

\item Suppose $k_3 = 0$, $k_2 = 1$, and $k_1 > 1$. The first $\delta$ returns $\btw$ similar to the above cases. Now $\delta(\nu_{\delta}, J_{\delta})$ returns $\bth$ since $(\nu_{\delta}, J_{\delta})^{(1)}$ contains a singular string of length $s-2$ and two singular strings of length $s-1$ and $(\nu_{\delta}, J_{\delta})^{(2)}$ contains a singular string of length $s-2$ and a string of length $s-1$ which is not singular. At this point we are in $B^{2,s-1}$ with $(k_1 - 2, 2, 0, 0)$.

\item Suppose $k_3 = 0$, $k_2 = 1$, and $k_1 = 1$ (so $s = 2$). This is a finite computation.
\end{enumerate}
\end{enumerate}
\end{proof}

\section{Virtualization map}
\label{sec:virtualization}

In this section, we describe the virtualization map for types $D_4^{(3)} \lhook\joinrel\longrightarrow D_4^{(1)}$. Moreover, we show that the bijection $\Phi$ commutes with the virtualization map.

\begin{prop}
\label{prop:virtualization_1}
The virtualization map $v \colon B^{1,s} \longrightarrow \virtual{B}^{2,s}$ as $U_q(\g_0)$-crystals is given column-by-column by
\[ {\arraycolsep=15pt \def\arraystretch{2.2}
\begin{array}{cccc}
\young(1) \mapsto \young(1,2)\, ,
& \young(2) \mapsto \young(1,3)\, ,
& \young(3) \mapsto \young(2,\bth)\, ,
& \young(0) \mapsto \young(3,\bth)\, ,
\\ \young(\bth) \mapsto \young(3,\btw)\, ,
& \young(\btw) \mapsto \young(\bth,\bon)\, ,
& \young(\bon) \mapsto \young(\btw,\bon)\, ,
& \young(\emptyset) \mapsto \young(1,\bon)\, .
\end{array}
} \]
\end{prop}

\begin{proof}
For $B^{1,1}$, this can be seen by direct computation. For general $B^{1,s}$, this follows from Proposition~\ref{prop:virtual_tensor} and Proposition~\ref{prop:filling_map_1}.
\end{proof}

\begin{thm}
\label{thm:bijection_virtualization_1}
Consider a Kirillov--Reshetikhin crystal $B = \bigotimes_{i=1}^N B^{1,s_i}$. The virtualization map $v$ commutes with the bijection $\Phi$.
\end{thm}

\begin{proof}
It it sufficient to show this on the leftmost factor of $B^{1,s_1}$.
It is clear that $v$ commutes with $\ls$, thus it remains to show that $v$ commutes with $\delta$. Let $(\nu, J)$ be a rigged configuration and $b_r$ denote the return value of $\delta(\nu, J)$. Recall that $\delta^v = \virtual{\delta}^2$. Let $(\virtual{\nu}_{\delta}, \virtual{J}_{\delta}) = \virtual{\delta}(\virtual{\nu}, \virtual{J})$. We proceed on a case-by-case basis based on $b_r$. In general, the (non-)existence of singular strings for $(\nu_{\delta}, J_{\delta})^{(a)}$ follows from the change in vacancy numbers, see Table~\ref{table:virtual_change_vac_nums}.

\begin{table}[t]
\[
\begin{array}{|c|c|c|c|c|c|c|}
\hline
& [1, \ell_1) & [\ell_1, \ell_2) & [\ell_2, \ell_{34}) & [\ell_{34}, \ellbar_2) & [\ellbar_2, \ellbar_1) & [\ellbar_1, \infty)
\\ \hline
\virtual{\nu}^{(1)} & -1 & +1 & 0 & 0 & -1 & +1
\\ \hline
\virtual{\nu}^{(2)} & 0 & -1 & +1 & -1 & +1 & 0
\\ \hline
\virtual{\nu}^{(3)} & 0 & 0 & -1 & +1 & 0 & 0
\\ \hline
\end{array}
\]
\caption{The change in vacancy numbers from $\virtual{\delta}$. Since we always have $\virtual{\nu}^{(3)} = \virtual{\nu}^{(4)}$, we denote the selected strings of $\ell_3 = \ell_4$ by $\ell_{34}$ and only describe $\virtual{\nu}^{(3)}$.}
\label{table:virtual_change_vac_nums}
\end{table}

\pagebreak
\case{$b_r = 1$}

By our assumption, there are no singular strings in $(\nu, J)^{(1)}$, and so $(\virtual{\nu}, \virtual{J})^{(2)}$ has no singular strings and $\virtual{\delta}$ returns $2$ and increases all vacancy numbers of $(\virtual{\nu}, \virtual{J})^{(1)}$ by 1. Thus $(\virtual{\nu}_{\delta}, \virtual{J}_{\delta})^{(1)}$ has no singular strings and so the second application of $\virtual{\delta}$ returns $1$.

\case{$b_r = 2$}

There is a singular string of length $\ell_1$ in $(\nu, J)^{(1)}$ but no larger singular strings in $(\nu, J)^{(2)}$. Thus $\virtual{\delta}$ selects the singular string of length $\ell_1$ in $(\virtual{\nu}, \virtual{J})^{(2)}$, but since there are no larger singular strings in $(\virtual{\nu}, \virtual{J})^{(3)}$ and $(\virtual{\nu}, \virtual{J})^{(4)}$, the process returns $3$. Hence $(\virtual{\nu}_{\delta}, \virtual{J}_{\delta})^{(1)}$ has no singular strings since all vacancy numbers $\virtual{p}_i^{(1)}$ for $i < \ell_1$ were increased by 1 and there were no singular strings of length at least $\ell_1$ in $(\virtual{\nu}_{\delta}, \virtual{J}_{\delta})^{(1)} = (\virtual{\nu}, \virtual{J})^{(1)}$. Thus the second application of $\virtual{\delta}$ returns $1$.

\case{$b_r = 3$}

There is a singular string of length $\ell_1$ in $(\nu, J)^{(1)}$ and of length $\ell_2$ in $(\nu, J)^{(2)}$, but no (additional) (quasi)singular strings of length at least $\ell_2$ in $(\nu, J)^{(1)}$. Thus $\virtual{\delta}$ selects the singular string of length $\ell_1$ in $(\virtual{\nu}, \virtual{J})^{(2)}$ and of length $\ell_2$ in $(\virtual{\nu}, \virtual{J})^{(3)}$ and $(\virtual{\nu}, \virtual{J})^{(4)}$, but since there are no (additional) singular strings of length at least $\ell_2$ in $(\virtual{\nu}, \virtual{J})^{(2)}$, the process returns $\bth$. Thus $(\virtual{\nu}_{\delta}, \virtual{J}_{\delta})^{(1)}$ has a singular string of length $\ell_1$, but $(\virtual{\nu}_{\delta}, \virtual{J}_{\delta})^{(2)}$ has no singular strings of length at least $\ell_1$. Thus the second application of $\virtual{\delta}$ returns $2$.

\case{$b_r = 0$}

This is the same as the previous case for the first application of $\virtual{\delta}$, which returns $\bth$. The second application of $\virtual{\delta}$ is similar to the previous case except the existence of a quasisingular string in $(\nu, J)^{(1)}$ of length at least $\ell_2$ implies there exists a singular string in $(\virtual{\nu}_{\delta}, \virtual{J}_{\delta})^{(2)}$ of length $\ell_3 \geq \ell_2 \geq \ell_1$. However there are no singular strings of length at least $\ell_3$ in $(\virtual{\nu}_{\delta}, \virtual{J}_{\delta})^{(3)} = (\virtual{\nu}_{\delta}, \virtual{J}_{\delta})^{(4)}$ since no such string exists in $(\nu, J)^{(1)}$. Therefore the second application of $\virtual{\delta}$ returns 3.

\case{$b_r = \bth$}

We first assume case~(S) holds, and so there are singular strings of length $\ell_1$ and $\ellbar_3$ in $(\nu, J)^{(1)}$ and of length $\ell_2$ in $(\nu, J)^{(2)}$ with $\ellbar_3 > \ell_2$, but no (additional) singular strings of length at least $\ellbar_3$. Therefore $\virtual{\delta}$ selects a singular string of length $\ell_1$ in $(\virtual{\nu}, \virtual{J})^{(2)}$, of length $\ell_2$ in $(\virtual{\nu}, \virtual{J})^{(3)} = (\virtual{\nu}, \virtual{J})^{(4)}$, and of length $\ellbar_3$ in $(\virtual{\nu}, \virtual{J})^{(2)}$. It then terminates because there are no singular strings of length at least $\ellbar_3$ in $(\virtual{\nu}, \virtual{J})^{(1)}$ and returns $\btw$. Next we note by the change in vacancy numbers that there are no singular strings of length strictly less than $\ell_2$ in $(\virtual{\nu}_{\delta}, \virtual{J}_{\delta})^{(1)}$ (net change is $+1$), so the second application of $\virtual{\delta}$ selects a singular string of length $\ell_2$ in $(\virtual{\nu}_{\delta}, \virtual{J}_{\delta})^{(1)}$ and a singular string of length $\ellbar_3$, but there are no singular strings of length at least $\ellbar_3$ in $(\virtual{\nu}_{\delta}, \virtual{J}_{\delta})^{(3)} = (\virtual{\nu}_{\delta}, \virtual{J}_{\delta})^{(4)}$. Therefore the second application of $\virtual{\delta}$ returns $3$.

If case~(Q,S) holds, then this is similar to when case~(S) holds except that the quasi-singular string of length $\ell_3$ in $(\virtual{\nu}, \virtual{J})^{(2)}$ becomes singular in $(\virtual{\nu}_{\delta}, \virtual{J}_{\delta})^{(2)}$ and is selected on the second application of $\virtual{\delta}$.

\case{$b_r = \btw$}

Assume that case~(S2) holds for $(\nu, J)^{(2)}$ (thus case~(S) also holds). We note that there are no singular strings of length at least $\ellbar_2$ in $(\nu, J)^{(1)}$. Thus on applying $\virtual{\delta}$, we select singular strings of length $\ell_1$ and $\ellbar_3$ from $(\virtual{\nu}, \virtual{J})^{(2)}$, of length $\ell_2$ from $(\virtual{\nu}, \virtual{J})^{(3)}$ and $(\virtual{\nu}, \virtual{J})^{(4)}$, and of length $\ellbar_2$ from $(\virtual{\nu}, \virtual{J})^{(1)}$. Thus $\virtual{\delta}$ returns $\bon$. For the second application of $\virtual{\delta}$ we select singular strings of length $\ell_2$ from $(\virtual{\nu}_{\delta}, \virtual{J}_{\delta})^{(1)}$, of length $\ellbar_3$ from $(\virtual{\nu}_{\delta}, \virtual{J}_{\delta})^{(2)}$, and of length $\ellbar_2$ from $(\virtual{\nu}_{\delta}, \virtual{J}_{\delta})^{(3)}$ and $(\virtual{\nu}_{\delta}, \virtual{J}_{\delta})^{(4)}$. In particular, we note that there are no singular strings of length at least $\ellbar_2$ in $(\virtual{\nu}_{\delta}, \virtual{J}_{\delta})^{(2)}$, so the process terminates and returns $\bth$.

If case~(S2) does not hold, then this is similar to when case~(S) holds and is independent whether case~(S) or case~(Q,S) holds as noted when $b_r = \bth$.

\case{$b_r = \bon$}

This is similar to the previous case except we select a singular string of length $\overline{\ell}_1$ in $(\virtual{\nu}_{\delta}, \virtual{J}_{\delta})^{(2)}$, and note there are no singular strings of length at least $\ellbar_1$ in $(\virtual{\nu}_{\delta}, \virtual{J}_{\delta})^{(1)}$.

\case{$b_r = \emptyset$}

The rigged configuration $(\nu, J)$ has at least two singular strings of length $1$ in $(\nu, J)^{(1)}$ and at least one singular string of length $1$ in $(\nu, J)^{(2)}$. Thus $\virtual{\delta}$ selects both singular strings of length 1 in $(\virtual{\nu}, \virtual{J})^{(2)}$ and the singular strings of length 1 in $(\virtual{\nu}, \virtual{J})^{(1)} = (\virtual{\nu}, \virtual{J})^{(3)} = (\virtual{\nu}, \virtual{J})^{(4)}$ and returns $\bon$. Hence $(\virtual{\nu}_{\delta}, \virtual{J}_{\delta})^{(1)}$ does not contain any singular strings as the vacancy numbers were all increased by $1$, so the second application of $\virtual{\delta}$ returns $1$.
\end{proof}

\begin{prop}
\label{prop:virtualization_2}
The KR crystal $B^{2,1}$ virtualizes in $B^{1,1} \otimes B^{3,1} \otimes B^{4,1} \iso B^{1,1} \otimes \widetilde{B}^{3,1}$ as a $U_q'(\g)$-crystal, where the virtualization map is determined by:
\[ {\arraycolsep=15pt \def\arraystretch{4}
\begin{array}{cc}
\young(1,2) \mapsto \young(1) \otimes \young(+,+,+,-) \otimes \young(+,+,+,+) = \young(1) \otimes \young(1,2,3),
& \young(1,0) \mapsto \young(2) \otimes \young(+,-,-,-) \otimes \young(+,+,+,+) = \young(2) \otimes \young(1,2,\btw),
\\ \young(1,\emptyset) \mapsto \young(\bth) \otimes \young(+,+,+,-) \otimes \young(+,+,+,+) = \young(\bth) \otimes \young(1,2,3),
& \young(\emptyset,\emptyset) \mapsto \young(\bon) \otimes \young(+,-,-,-) \otimes \young(+,+,+,+) = \young(\bon) \otimes \young(1,2,\btw),
\end{array}
} \]
and extended as a crystal morphism.
\end{prop}

\begin{proof}
A direct finite computation.
\end{proof}

We define $\lt^v$ as the following composition map (omitting the right factors):
\[
\widetilde{B}^{3,1} \otimes B^{1,1} \xrightarrow[\hspace{30pt}]{\virtual{\lt}}
B^{1,1} \otimes B^{2,1} \otimes B^{1,1} \xrightarrow[\hspace{30pt}]{\virtual{R}}
B^{1,1}  \otimes B^{1,1} \otimes B^{2,1} \xrightarrow[\hspace{30pt}]{\virtual{\lt}^{-1}} B^{2,1} \otimes B^{2,1}.
\]

\begin{lemma}
\label{lemma:virtual_lt}
We have
\[
v \circ \lt = \lt^v \circ v.
\]
\end{lemma}

\begin{proof}
Recall that $\virtual{\lt}$ adds a singular string to $\virtual{\nu}^{(1)}$ and $\virtual{\nu}^{(2)}$ and $\virtual{R}$ acts trivially, thus $\virtual{\lt}^{-1}$ is well-defined and removes a singular string from $\virtual{\nu}^{(1)}$.
Hence the map $\lt^v$ adds a single singular string to $\virtual{\nu}^{(2)}$ and does not change the vacancy numbers. Thus this is equal to $v \circ \lt$ on rigged configurations. On KR tableaux, this can be reduced to the leftmost factor in type $D_4^{(3)}$ (or factors in type $D_4^{(1)}$), and thus is a finite computation (that can be done say, by computer). Thus we have $v \circ \lt = \lt^v \circ v$.
\end{proof}

\begin{thm}
\label{thm:bijection_virtualization}
Consider a tensor product of KR crystals $B$ which contain factors of the form $B^{1,s}$ or $B^{2,1}$ (possibly both). The virtualization map $v$ commutes with the bijection $\Phi$.
\end{thm}

\begin{proof}
We proceed by induction on the number of factors. Suppose that the leftmost factor is $B^{1,s}$, this is shown by Theorem~\ref{thm:bijection_virtualization_1}. Thus it remains to show this when the leftmost factor is $B^{2,1}$. From Lemma~\ref{lemma:virtual_lt}, we can apply $\lt$ and construct the desired virtual image. Thus the leftmost factor is $B^{1,1}$, so we can apply Theorem~\ref{thm:bijection_virtualization_1}. Therefore by the definition of $\Phi$, we have $v \circ \Phi = \virtual{\Phi} \circ v$.
\end{proof}

Thus we have shown a special case of Conjecture~\ref{conj:virtualization}.

\begin{remark}
\label{rem:type_Atwisted}
We note that we have the analog of Theorem~\ref{thm:bijection_virtualization} for $\virtual{\g}$ of type $A_7^{(2)}$ using the same virtualization map descriptions on tableaux (using $\virtual{B}^{1,1} \otimes \widetilde{B}^{3,1}$) and rigged configurations. This follows from the fact that the bijection $\Phi$ of type $A_7^{(2)}$ only differs from the type $D_4^{(1)}$ bijection by the identification of $\virtual{\nu}^{(3)} = \virtual{\nu}^{(4)}$. This is equivalent to the fact that the following virtualization maps must commute:
\[
\xymatrixrowsep{3pc}
\xymatrixcolsep{3.5pc}
\xymatrix{D_4^{(1)} \\ D_4^{(3)} \ar[r] \ar[u] & A_7^{(2)} \ar[lu]}
\]
\end{remark}

\begin{prop}
\label{prop:filling_map_1}
Let $\iota$ denote the natural crystal isomorphism $\iota \colon \RC(B^{1,s}) \longrightarrow B^{1,s}$ of type $D_4^{(3)}$. Then we have
\[
\Phi = \fillmap \circ \iota.
\]
\end{prop}

\begin{proof}
From Theorem~\ref{thm:bijection_virtualization}, the following diagram commutes:
\[
\xymatrixrowsep{3pc}
\xymatrixcolsep{3.5pc}
\xymatrix{\RC(B^{1,s}) \ar[r]^{\Phi} \ar[d]_v & T^{1,s} \ar[d]^v
\\ \RC(\virtual{B}^{2,s}) \ar[r]^{\virtual{\Phi}} & \virtual{T}^{2,s}}
\]
From Theorem~\ref{thm:filling_map_Daff}, the map $\virtual{\Phi}$ is a $U_q(\g_0)$-crystal isomorphism, and the claim follows from Proposition~\ref{prop:filling_map_isomorphism}.
\end{proof}

\begin{remark}
Proposition~\ref{prop:filling_map_1} can also be shown directly without appealing to the virtualization map and type $D_4^{(1)}$ similar to the proof of Proposition~\ref{prop:filling_map_2} or similar propositions in~\cite{SS15}. We leave the proof as an exercise for the interested reader.
\end{remark}

\section{Affine crystal structure}
\label{sec:affine}

In this section, we describe the affine crystal structure on rigged configurations $\RC(B^{r,s})$. In particular, we show that the KR crystal $B^{1,s}$ virtualizes in $\virtual{B}^{2,s}$. We first recall Theorem~6.1 in~\cite{KMOY07}.

\begin{thm}
\label{thm:unique_affine_structure}
Let $B$ and $B'$ be $U_q'(\g)$-crystals whose classical decompositions are $\bigotimes_{k=0}^s B(k\clfw_{N_0})$. Then $B \iso B'$.
\end{thm}

In other words, Theorem~\ref{thm:unique_affine_structure} states that there exists a unique $U_q'(\g)$-crystal structure on such a crystal.

We now describe an alternative $U_q'(\g)$-crystal structure on $B^{1,s}$.
\begin{dfn}
\label{def:affine_crystal_ops}
Let $V^{1,s} = \bigoplus_{k=0}^s B(k \clfw_1)$. Define
\begin{subequations}
\begin{align}
e_0 & := v^{-1} \circ \virtual{e}_0 \circ v,
\\ f_0 & := v^{-1} \circ \virtual{f}_0 \circ v.
\end{align}
\end{subequations}
\end{dfn}

\begin{prop}
\label{prop:virtual_structure}
Consider the crystal $V^{1,s}$ and the crystal operators from Definition~\ref{def:affine_crystal_ops}. Then $V^{1,s}$ is an abstract $U_q'(\g)$-crystal. Moreover $V^{1,s} \iso B^{1,s}$.
\end{prop}

\begin{proof}
Note that there is a virtualization map from $V^{1,s}$ into $\virtual{B}^{2,s}$ as $U_q(\g_0)$-crystals by Proposition~\ref{prop:filling_map_1}. We also note that $v(V^{1,s})$ is characterized by the elements $b \in \virtual{B}^{2,s}$ such that
\begin{gather*}
\virtual{\varepsilon}_1(b) = \virtual{\varepsilon}_3(b) = \virtual{\varepsilon}_4(b),
\\ \virtual{\varphi}_1(b) = \virtual{\varphi}_3(b) = \virtual{\varphi}_4(b).
\end{gather*}
Since $0$ is not adjacent to $1, 3, 4$ in type $D_4^{(1)}$, applying $\virtual{e}_0$ or $\virtual{f}_0$ does not change $\virtual{\varepsilon}_a$ and $\virtual{\varphi}_a$ for $a \in \{1, 3, 4\}$. Thus $v(V^{1,s})$ is closed under $e^v_a$ and $f^v_a$ for all $a \in I$. Furthermore, $\virtual{B}^{2,s}$ is a $U_q'(\g)$-crystal, and so the claim that $V^{1,s}$ is an abstract $U_q'(\g)$-crystal follows. Thus, from Theorem~\ref{thm:unique_affine_structure}, we have $V^{1,s} \iso B^{1,s}$.
\end{proof}

Thus we have the following special case of Conjecture~\ref{conj:KR_virtualization} as an immediate consequence of Proposition~\ref{prop:virtual_structure}.

\begin{cor}
\label{cor:KR_virtualization_1}
The KR crystal $B^{1,s}$ of type $D_4^{(3)}$ virtualizes in $\virtual{B}^{2,s}$ of type $D_4^{(1)}$.
\end{cor}

Thus we can describe a $U_q'(\g)$-crystal structure on rigged configurations by using the virtualization map and the affine crystal structure for $\RC(\virtual{B}^{2,s})$ of type $D_4^{(1)}$ given in~\cite{OSS13}.

\begin{remark}
We have the analog of Corollary~\ref{cor:KR_virtualization_1} for $\virtual{B}^{2,s}$ of type $A_7^{(2)}$ as in Remark~\ref{rem:type_Atwisted}.
\end{remark}

\begin{figure}[ht]
\label{fig:rc_B11}
\[
\begin{tikzpicture}[>=latex,line join=bevel, xscale=1, yscale=.75, every node/.style={scale=0.75}]
\input{rc_d43_b11}
\end{tikzpicture}
\]
\caption{The crystal $\RC(B^{1,1})$ of type $D_4^{(3)}$ generated using Sage~\cite{sage}.}
\end{figure}

\section{Main Results}
\label{sec:main_results}

We obtain our main results in this section.

\begin{thm}
\label{thm:filling_map}
Let $B^{r,s}$ be a KR crystal of type $D_4^{(3)}$. Then there exists a natural crystal isomorphism $\iota \colon \RC(B^{r,s}) \longrightarrow B^{r,s}$ such that
\[
\Phi = \fillmap \circ \iota.
\]
\end{thm}

\begin{proof}
This follows from Proposition~\ref{prop:filling_map_1} and Proposition~\ref{prop:filling_map_2}.
\end{proof}

Next we obtain the second part of our main results for single columns.

\begin{thm}
\label{thm:bijection_single_columns}
Let $B = \bigotimes_{i=1}^N B^{r_i,1}$ be a tensor product of KR crystals of type $D_4^{(3)}$. Then
\[
\Phi \colon \RC(B) \longrightarrow B
\]
is a bijection on classically highest weight elements such that $\Phi \circ \eta$ sends cocharge to energy.
\end{thm}

\begin{proof}
Consider the tensor product of KR crystals by $B_L \otimes B$, where $B_L = B^{r,1}$. Denote its virtual image in type $D_4^{(1)}$ by $B_L^v \otimes B^v$. Then the diagram
\[
\xymatrixrowsep{2.5pc}
\xymatrixcolsep{3pc}
\xymatrix{\RC(B^v_L \otimes B^v) \ar[rrr]^{\virtual{\Phi}} \ar[ddd]_{\virtual{\delta}^r} & & & B_L^v \otimes B^v \ar[ddd]^{\virtual{\delta}^r}
\\ & \RC(B_L \otimes B) \ar[r]^{\Phi} \ar[d]_{\delta^r} \ar[ul]^{v}  & B_L \otimes B \ar[d]^{\delta^r} \ar[ur]^{v} &
\\ & \RC(B) \ar[r]^{\Phi} \ar[dl]^{v} & B \ar[dr]^{v} &
\\ \RC(B^v) \ar[rrr]^{\virtual{\Phi}} & & & B^v}
\]
commutes from Theorem~\ref{thm:bijection_virtualization}. It is a finite computation to show that $v \circ H  = \virtual{H} \circ v$ on $B_L$. Thus the claim follows from Theorem~\ref{thm:statistics_preserving_Daff} and Theorem~\ref{thm:crystal_isomorphism_D}.
\end{proof}

To show the statistics are preserved for $B = \bigotimes_{i=1}^N B^{1,s_i}$, we follow~\cite[Sec.~8]{SS2006}. Let $[f, g] = 0$ denote that $f \circ g = g \circ f$, i.e., the maps $f$ and $g$ commute. Define $\widetilde{\delta} = \eta \circ \delta \circ \eta$. In particular, we show that the left-box and right-box maps commute with themselves and with the respective $\delta$ and $\widetilde{\delta}$, the combinatorial $R$-matrix gets sent to the identity on rigged configurations, and $[\delta, \widetilde{\delta}] = 0$.

\begin{lemma}
We have
\[
[\delta, \widetilde{\delta}] = 0
\]
on highest weight rigged configurations.
\end{lemma}

\begin{proof}
Let $\lt$ add a singular string of length 1 to $\nu^{(1)}$ and $\rb$ add a cosingular string. We can consider $\widetilde{\delta}$ as selecting smallest cosingular strings on highest weight rigged configurations. It is clear that $\virtual{\eta} \circ v = v \circ \eta$.

Let $\delta^v = \virtual{\delta} \circ \virtual{\lt} \circ \virtual{\delta}$ and $\widetilde{\delta}^v = \virtual{\widetilde{\delta}} \circ \virtual{\rb} \circ \virtual{\widetilde{\delta}}$. We have $v \circ \delta = \delta^v \circ v$ by Lemma~\ref{prop:virtualization_1}, and so $v \circ \widetilde{\delta} = \widetilde{\delta}^v \circ v$. The fact that $[\virtual{\lt}, \virtual{\rb}] = 0$ is clear from their definitions and they don't change the vacancy numbers. That $[\virtual{\lt}, \virtual{\widetilde{\delta}}] = 0$ and $[\virtual{\rb}, \virtual{\delta}] = 0$ are located in~\cite{S05}. Therefore we have
\begin{align*}
v \circ \delta \circ \widetilde{\delta}
& = \delta^v \circ \widetilde{\delta}^v \circ v
  = \virtual{\delta} \circ \virtual{\lt} \circ \virtual{\delta} \circ \virtual{\widetilde{\delta}} \circ \virtual{\rb} \circ \virtual{\widetilde{\delta}} \circ v
\\ & = \virtual{\delta} \circ \virtual{\lt} \circ \virtual{\widetilde{\delta}} \circ \virtual{\delta} \circ \virtual{\rb} \circ \virtual{\widetilde{\delta}} \circ v
     = \virtual{\delta} \circ \virtual{\widetilde{\delta}} \circ \virtual{\lt} \circ \virtual{\rb} \circ \virtual{\delta} \circ \virtual{\widetilde{\delta}} \circ v
\\ & = \virtual{\widetilde{\delta}} \circ \virtual{\delta} \circ \virtual{\rb} \circ \virtual{\lt} \circ \virtual{\widetilde{\delta}} \circ \virtual{\delta} \circ v
		     = \virtual{\widetilde{\delta}} \circ \virtual{\rb} \circ \virtual{\delta} \circ \virtual{\widetilde{\delta}} \circ \virtual{\lt} \circ \virtual{\delta} \circ v
\\ & = \virtual{\widetilde{\delta}} \circ \virtual{\rb} \circ \virtual{\widetilde{\delta}} \circ \virtual{\delta} \circ \virtual{\lt} \circ \virtual{\delta} \circ v
     = \widetilde{\delta}^v \circ \delta^v \circ v = v \circ \widetilde{\delta} \circ \delta.
\end{align*}
Therefore $[\delta, \widetilde{\delta}] = 0$.
\end{proof}

\begin{lemma}
\label{lemma:right_left_commute}
Consider $B = \bigotimes_{i=1}^N B^{r_i, s_i}$. The following hold:
\begin{itemize}
\item $[\ls, \rs] = 0$,
\item $[\ls, \rb] = 0$,
\item $[\lt, \rs] = 0$,
\item $[\lt, \rb] = 0$,
\end{itemize}
on both $B$ and $\RC(B)$.
\end{lemma}

\begin{proof}
For $B$, each of these statements are obvious. For $\RC(B)$, these statements follow immediately from the fact that $\ls$ and $\rs$ are the identity on rigged configurations and $\lt$ and $\rb$ preserve vacancy numbers.
\end{proof}

We also need to define an analog of the $\eta$ map on a $U_q'(\g)$-crystal $B$. We give a crystal morphism $\psi$ on $B$ which satisfies
\begin{subequations}
\label{eq:lusztig_involution}
\begin{align}
e_i\bigl( \psi(b) \bigr) & = \psi\bigl( f_i(b) \bigr),
\\ f_i\bigl( \psi(b) \bigr) & = \psi\bigl( e_i(b) \bigr),
\\ \overline{\wt}\bigl( \psi(b) \bigr) & = w_0 \overline{\wt}(b),
\end{align}
\end{subequations}
where $w_0$ is the longest element of the Weyl group of type $G_2$ (i.e., of $\g_0$) and $i \in I$. Note Equation~\eqref{eq:lusztig_involution} implies $\psi$ must preserve classical components. Moreover it is clear the following diagram commutes:
\begin{equation}
\label{eq:remove_ast_cd}
\xymatrixrowsep{3.5pc}
\xymatrixcolsep{3pc}
\xymatrix{B_2 \otimes B_1 \ar[r]^-{\delta} \ar[d]_{\psi} & B_1 \ar[d]^{\psi} \\ B_1 \otimes B_2 \ar[r]^-{\widetilde{\delta}} & B_1}.
\end{equation}

The following proposition is a combination of Proposition~2.9 and Remark~2.11 in~\cite{SS2006}.

\begin{prop}
For $U_q'(\g)$-crystals $B_1, B_2$ which have such a unique map $\psi$ satisfying Equation~\eqref{eq:lusztig_involution}, we can extend this to $B_1 \otimes B_2$ by 
\[
\psi(b_1 \otimes b_2) = R\bigl( \psi(b_2) \otimes \psi(b_1) \bigr).
\]
Moreover, we have
\[
\psi(B_1 \otimes B_2) \iso \psi(B_2) \otimes \psi(B_1).
\]
Thus the category of $U_q'(\g)$-crystals with a map $\psi$ satisfying Equation~\eqref{eq:lusztig_involution} form a tensor category.
\end{prop}

We explicitly describe $\psi$ on the unfilled tableaux of $B^{1,s}$ by interchanging $i \leftrightarrow \bi$ (with $0$ as a fixed point) and reversing the tableau. We note that this is the unique map satisfying Equation~\eqref{eq:lusztig_involution}, showing~\cite[Conj.~2.10]{SS2006} for type $D_4^{(3)}$.

\begin{remark}
There is an analogous (unique) map $\virtual{\psi}$ for $\virtual{B}^{2,s}$ satisfying a type $D_n$ version of Equation~\eqref{eq:lusztig_involution}, see, e.g., Section~3.6 of~\cite{S05} for more details, since the classical decomposition is multiplicity free. It is straightforward to show that $v \circ \psi = \virtual{\psi} \circ v$.
\end{remark}

Let $f^{HW}$ denote the composition of a function $f$ and then sending the result to the corresponding classically highest weight element. We note that $\delta^{HW}$ and $\widetilde{\delta}^{HW}$ are well-defined by Equation~\eqref{eq:remove_ast_cd}, analogous to Lemma~5.4 in~\cite{SS2006}.

\begin{prop}
\label{prop:correspondence}
Let $B$ be a tensor product of KR crystals of type $D_4^{(3)}$. Then the following hold:
\begin{itemize}
\item $[\lt, \Phi] = 0$,
\item $[\rb, \Phi] = 0$,
\item $[\ls, \Phi] = 0$ for $B = B^{1,s} \otimes B^*$,
\item $[\rs, \Phi] = 0$ for $B = B^* \otimes B^{1,s}$,
\item $\Phi \circ \eta = \psi^{HW} \circ \Phi$.
\end{itemize}
\end{prop}

\begin{proof}
$[\lt, \Phi] = 0$ follows from Lemma~\ref{lemma:virtual_lt}, Theorem~\ref{thm:bijection_virtualization}, and Theorem~\ref{thm:statistics_preserving_Daff}.

\vspace{12pt}
Next we show $[\ls, \Phi] = 0$. We write $B = B^{1,s} \otimes B^*$ where $s \geq 2$. Note that an element $b \in B$ is in the image of $\ls \colon B^{1,s} \otimes B^* \longrightarrow B^{1,1} \otimes B^{1,s-1} \otimes B^*$ if the leftmost elements $x,y$ of $B^{1,s}$ satisfy either $x \leq y$ or $x = \bon$ and $y = 1$. Let $\ell_i^{(x)}$ and $\overline{\ell}_i^{(x)}$ denote the strings selected by $\delta^{-1}(x)$. If $x \leq y$, then from the definition of $\delta^{-1}$, we must have $\ell_i^{(x)} \geq \ell_i$ and $\overline{\ell}_i^{(x)} \geq \overline{\ell}_i$ for all $i = 1,2,3$. Since $\ell_1 \geq s-1$ by assumption, we must have $\ell_1^{(x)} \geq s$, and so the image under $\Phi$ is in the image of $\ls$. If we have $x = \bon$ and $y = 1$, then we don't change the colabels of all strings of length at least $s - 1$, but decrease the colabels of all strings of length smaller than $s - 1$. In particular, there are at least 2 such strings of length exactly $s - 1$ which are singular (possibly of length 0), and we add 2 boxes to each such string. Hence the resulting length has length $s + 1$, and so the result is in the image of $\ls$.

A rigged configuration is in the image of $\ls \colon \RC(B^{1,s} \otimes B^*) \longrightarrow \RC(B^{1,1} \otimes B^{1,s-1} \otimes B^*)$ if there are no singular strings of length smaller than $s$ in $(\nu, J)^{(1)}$. Let $(\nu_{\delta}, J_{\delta}) = (\delta \circ \ls)(\nu, J)$, and let $\ell_i^{\delta}$ and $\overline{\ell}_i^{\delta}$ denote the strings selected by $(\delta \circ \ls)(\nu_{\delta}, J_{\delta})$. Now $\delta \circ \ls$ returns $x$ with $\ell_1 \geq s$. Thus unless case~(S1) holds, there exists a singular string of length $\ell_1 - 1 \geq s - 1$ in $(\nu_{\delta}, J_{\delta})^{(1)}$ by the definition of $\delta$. Similarly, there exists singular strings for $\ell_i - 1 \geq \ell_i^{\delta}$ and $\overline{\ell}_i - 1 \geq \overline{\ell}_i^{\delta}$ for all $i$. If case~(S1) holds, then if $\ell_1 - 1 \geq s$, we are in the previous case. Otherwise $\ell_1 = \overline{\ell}_1 = s$, we have that there are no singular strings in $(\nu_{\delta}, J_{\delta})^{(1)}$ by Table~\ref{table:change_vac_nums} and we return $1$. Additionally, if $s \geq 3$, the next such application of $(\delta \circ \ls)$ will select the singular strings of length $\ell_1 - 2$ and return $\bon$ with case~(S1) holding. Repeating this we obtain the desired KR tableau.

\vspace{12pt}
The proofs for $[\rb, \Phi] = 0$, $[\rs, \Phi] = 0$ and $\Phi \circ \eta = \psi^{HW} \circ \Phi$ are the similar to those given in~\cite[Thm.~8.6]{SS2006}. Note that Proposition~5.6 in~\cite{SS2006} is replaced by Lemma~\ref{lemma:right_left_commute} and the analogous statement to Proposition~8.5 in~\cite{SS2006} follows from Theorem~\ref{thm:bijection_virtualization} and that $\virtual{\lt}$ and $\virtual{\rb}$ preserve weights.
\end{proof}

\begin{thm}
\label{thm:bijection_classical_iso}
Let $\g$ be of type $D_4^{(3)}$. Consider a tensor product of KR crystals $B$ which contain factors of the form $B^{1,s}$ or $B^{2,1}$ (possibly both). The map
\[
\Phi \colon \RC(B) \longrightarrow B
\]
is a classical crystal isomorphism.
\end{thm}

\begin{proof}
We first show $\Phi$ is a bijection. This follows from the proof of Theorem~\ref{thm:bijection_single_columns} and Proposition~\ref{prop:correspondence}.

To show $\Phi$ commutes with $f_a$, we have
\begin{align*}
v \circ \Phi \circ f_a = \virtual{\Phi} \circ v \circ f_a = \virtual{\Phi} \circ f_a^v \circ v = f_a^v \circ \virtual{\Phi} \circ v = f_a^v \circ v \circ \Phi =  v \circ f_a \circ \Phi
\end{align*}
from Theorem~\ref{thm:bijection_virtualization}, Theorem~\ref{thm:crystal_isomorphism_D}, and that $v$ is a virtualization map (Proposition~\ref{prop:virtualization_1} and Proposition~\ref{prop:virtualization_2}). A similar statement holds for $e_a$. Since we are considering regular crystals, we have that $\Phi$ is a classical crystal isomorphism.
\end{proof}

\begin{thm}
\label{thm:R_matrix_id}
Let $\g$ be of type $D_4^{(3)}$. The following diagram commutes:
\[
\xymatrixrowsep{4pc}
\xymatrixcolsep{3pc}
\xymatrix{\RC(B^{1,1} \otimes B^{1,s}) \ar[r]^-{\Phi} \ar[d]_{\id} & B^{1,1} \otimes B^{1,s} \ar[d]^R
\\ \RC(B^{1,s} \otimes B^{1,1}) \ar[r]^-{\Phi} & B^{1,s} \otimes B^{1,1}}
\]
\end{thm}

\begin{proof}
It is sufficient to show this diagram commutes on highest weight elements since $\Phi$ is a classical crystal isomorphism by Theorem~\ref{thm:bijection_classical_iso}, and $\id$ and $R$ are (classical) crystal isomorphisms. We proceed case-by-case on classically highest weight elements given by Theorem~\ref{thm:r_matrix}.

\case{$1 \otimes 1^n \mapsto 1^{n+1} \bon \otimes 1$ if $0 \leq n \leq s-2$}

\nopagebreak
For the left hand side, we note that adding the filled portion under $\Phi^{-1}$ results in $\nu^{(1)} = (s-n, s-n)$ and $\nu^{(2)} = (s-n)$ with all riggings equal to 0 and all strings being singular. Then adding the final $\young(1)$ does not change the rigged configuration. For the right hand side, adding the filled portion under $\Phi^{-1}$ results in $\nu^{(1)} = (s-n-2, s-n-2)$ and $\nu^{(2)} = (s-n-2)$ with all strings being singular. Thus when adding $\young(\bon)$, the resulting rigged configuration is $\nu^{(1)} = (s-n, s-n)$ and $\nu^{(2)} = (s-n)$ with all riggings equal to 0. Adding in the remaining $1^{n+1}$ does not change the rigged configuration. Thus the results are equal.

\begin{ex}
Consider $B = B^{1,1} \otimes B^{1,5}$ with $n = 2$, we have
\begin{align*}
\young(11\bon1\emptyset) \xrightarrow[\hspace{40pt}]{\Phi^{-1}} &
\begin{tikzpicture}[scale=.35,anchor=top,baseline=-25pt]
 \rpp{3,3}{0,0}{0,0}
 \begin{scope}[xshift=6cm]
 \rpp{3}{0}{0}
 \end{scope}
\end{tikzpicture}
\\ \young(1) \otimes \young(11\bon1\emptyset) \xrightarrow[\hspace{40pt}]{\Phi^{-1}} &
\begin{tikzpicture}[scale=.35,anchor=top,baseline=-25pt]
 \rpp{3,3}{0,0}{1,1}
 \begin{scope}[xshift=6cm]
 \rpp{3}{0}{0}
 \end{scope}
\end{tikzpicture}
\\ \young(\emptyset) \otimes \young(1) \xrightarrow[\hspace{40pt}]{\Phi^{-1}} &
\begin{tikzpicture}[scale=.35,anchor=top,baseline=-25pt]
 \rpp{1,1}{1,1}{1,1}
 \begin{scope}[xshift=6cm]
 \rpp{1}{0}{0}
 \end{scope}
\end{tikzpicture}
\\ \young(\bon\emptyset) \otimes \young(1) \xrightarrow[\hspace{40pt}]{\Phi^{-1}} &
\begin{tikzpicture}[scale=.35,anchor=top,baseline=-25pt]
 \rpp{3,3}{0,0}{0,0}
 \begin{scope}[xshift=6cm]
 \rpp{3}{0}{0}
 \end{scope}
\end{tikzpicture}
\\ \young(111\bon\emptyset) \otimes \young(1) \xrightarrow[\hspace{40pt}]{\Phi^{-1}} &
\begin{tikzpicture}[scale=.35,anchor=top,baseline=-25pt]
 \rpp{3,3}{0,0}{1,1}
 \begin{scope}[xshift=6cm]
 \rpp{3}{0}{0}
 \end{scope}
\end{tikzpicture}
\end{align*}
\end{ex}

\case{$1 \otimes 1^{s-1} \mapsto 1^s \otimes \emptyset$}

\nopagebreak
Recall that the filling of the left hand side is $1 \otimes 1^{s-1}\emptyset$. We note that for $\Phi^{-1}$, the first step for both sides we add in $\young(\emptyset)$ and so our partitions are $\nu^{(1)} = (1,1)$ and $\nu^{(2)} = (1)$ with all riggings equal to 0. All other insertions of $\young(1)$ under $\Phi^{-1}$ do not change the rigged configuration, and therefore they are equal.

\case{$1 \otimes 1^s \mapsto 1^s \otimes 1$}

\nopagebreak
This is clear since this corresponds to $(\nu_{\emptyset}, J_{\emptyset})$.

\case{$2 \otimes 1^n \mapsto 1^{n-1}20 \otimes 1$ if $1 \leq n \leq s-1$}

\nopagebreak
For the left hand side, after adding the right factor under $\Phi^{-1}$, we get $\nu^{(1)} = (s-n, s-n)$ and $\nu^{(2)} = (s-n)$ with all riggings equal to 0 and all strings being singular. Thus when adding $\young(2)$, we get $\nu^{(1)} = (s-n+1, s-n)$ and $\nu^{(2)} = (s-n)$ with all riggings 0.

\begin{ex}
Consider $B = B^{1,1} \otimes B^{1,5}$ with $n = 3$, we have
\begin{align*}
\young(111\bon1) \xrightarrow[\hspace{40pt}]{\Phi^{-1}} &
\begin{tikzpicture}[scale=.35,anchor=top,baseline=-25pt]
 \rpp{2,2}{0,0}{0,0}
 \begin{scope}[xshift=6cm]
 \rpp{2}{0}{0}
 \end{scope}
\end{tikzpicture}
\\ \young(2) \otimes \young(111\bon1) \xrightarrow[\hspace{40pt}]{\Phi^{-1}} &
\begin{tikzpicture}[scale=.35,anchor=top,baseline=-25pt]
 \rpp{3,2}{0,0}{0,1}
 \begin{scope}[xshift=6cm]
 \rpp{2}{0}{0}
 \end{scope}
\end{tikzpicture}
\\ \young(\emptyset) \otimes \young(1) \xrightarrow[\hspace{40pt}]{\Phi^{-1}} &
\begin{tikzpicture}[scale=.35,anchor=top,baseline=-25pt]
 \rpp{1,1}{1,1}{1,1}
 \begin{scope}[xshift=6cm]
 \rpp{1}{0}{0}
 \end{scope}
\end{tikzpicture}
\\ \young(0\emptyset) \otimes \young(1) \xrightarrow[\hspace{40pt}]{\Phi^{-1}} &
\begin{tikzpicture}[scale=.35,anchor=top,baseline=-25pt]
 \rpp{2,2}{1,0}{1,1}
 \begin{scope}[xshift=6cm]
 \rpp{2}{0}{0}
 \end{scope}
\end{tikzpicture}
\\ \young(1120\emptyset) \otimes \young(1) \xrightarrow[\hspace{40pt}]{\Phi^{-1}} &
\begin{tikzpicture}[scale=.35,anchor=top,baseline=-25pt]
 \rpp{3,2}{0,0}{0,1}
 \begin{scope}[xshift=6cm]
 \rpp{2}{0}{0}
 \end{scope}
\end{tikzpicture}
\end{align*}
\end{ex}

\case{$2 \otimes 1^s \mapsto 1^{s-1} 2 \otimes 1$}

\nopagebreak
For the left hand side, the result under $\Phi^{-1}$ is $\nu^{(1)} = (1)$ with $J_1^{(1)} = (1)$ and $\nu^{(2)} = \emptyset$. For the right hand side, after adding the $\young(2)$ under $\Phi^{-1}$ results in $\nu^{(1)} = (1)$ with $J_{1}^{(1)} = (1)$ and $\nu^{(2)} = \emptyset$. Adding the remaining $1^{s-1}$ does not change the rigged configuration. Thus the results are equal.

\begin{ex}
Consider $B = B^{1,1} \otimes B^{1,5}$, we have
\begin{align*}
\young(2) \otimes \young(11111) \xrightarrow[\hspace{40pt}]{\Phi^{-1}} &
\begin{tikzpicture}[scale=.35,baseline=-18pt]
 \rpp{1}{1}{1}
 \begin{scope}[xshift=6cm]
 \node at (0, -40pt) {$\emptyset$};
 \end{scope}
\end{tikzpicture}
\\ \young(2) \otimes \young(1) \xrightarrow[\hspace{40pt}]{\Phi^{-1}} &
\begin{tikzpicture}[scale=.35,baseline=-18pt]
 \rpp{1}{1}{1}
 \begin{scope}[xshift=6cm]
 \node at (0, -40pt) {$\emptyset$};
 \end{scope}
\end{tikzpicture}
\\ \young(11112) \otimes \young(1) \xrightarrow[\hspace{40pt}]{\Phi^{-1}} &
\begin{tikzpicture}[scale=.35,baseline=-18pt]
 \rpp{1}{1}{1}
 \begin{scope}[xshift=6cm]
 \node at (0, -40pt) {$\emptyset$};
 \end{scope}
\end{tikzpicture}
\end{align*}
\end{ex}

\case{$0 \otimes 1^n \mapsto 1^{n-1}0 \otimes 1$}

\nopagebreak
Recall that $n \geq 1$. For the left hand side, after adding the right factor under $\Phi^{-1}$ we get $\nu^{(1)} = (s-n, s-n)$ and $\nu^{(2)} = (s-n)$ with all riggings 0 and all strings being singular. After adding $\young(0)$, we obtain $\nu^{(1)} = (s-n+1, s-n+1)$ with $J_{s-n+1}^{(1)} = (1, 0)$ and $\nu^{(2)} = (s-n+1)$ with $J_{s-n+1}^{(2)} = (0)$. For the right hand side, after adding the filled elements of the left hand factor under $\Phi^{-1}$, we have $\nu^{(1)} = (s-n, s-n)$ and $\nu^{(2)} = (s-n)$ with all riggings equal to $1$ and all strings being singular. Thus when adding $\young(0)$, we have $\nu^{(1)} = (s-n+1, s-n+1)$ with $J_{s-n+1}^{(1)} = (1, 0)$ and $\nu^{(2)} = (s-n+1)$ with $J_{s-n+1}^{(2)} = (0)$. Adding the remaining $1^{n-1}$ does not change the rigged configuration. Thus the results are equal.

\begin{ex}
Consider $B = B^{1,1} \otimes B^{1,5}$ with $n = 3$, we have
\begin{align*}
\young(111\bon1) \xrightarrow[\hspace{40pt}]{\Phi^{-1}} &
\begin{tikzpicture}[scale=.35,anchor=top,baseline=-25pt]
 \rpp{2,2}{0,0}{0,0}
 \begin{scope}[xshift=6cm]
 \rpp{2}{0}{0}
 \end{scope}
\end{tikzpicture}
\\ \young(0) \otimes \young(111\bon1) \xrightarrow[\hspace{40pt}]{\Phi^{-1}} &
\begin{tikzpicture}[scale=.35,anchor=top,baseline=-25pt]
 \rpp{3,3}{1,0}{1,1}
 \begin{scope}[xshift=6cm]
 \rpp{3}{0}{0}
 \end{scope}
\end{tikzpicture}
\\ \young(\bon1) \otimes \young(1) \xrightarrow[\hspace{40pt}]{\Phi^{-1}} &
\begin{tikzpicture}[scale=.35,anchor=top,baseline=-25pt]
 \rpp{2,2}{1,1}{1,1}
 \begin{scope}[xshift=6cm]
 \rpp{2}{0}{0}
 \end{scope}
\end{tikzpicture}
\\ \young(0\bon1) \otimes \young(1) \xrightarrow[\hspace{40pt}]{\Phi^{-1}} &
\begin{tikzpicture}[scale=.35,anchor=top,baseline=-25pt]
 \rpp{3,3}{1,0}{1,1}
 \begin{scope}[xshift=6cm]
 \rpp{3}{0}{0}
 \end{scope}
\end{tikzpicture}
\\ \young(110\bon1) \otimes \young(1) \xrightarrow[\hspace{40pt}]{\Phi^{-1}} &
\begin{tikzpicture}[scale=.35,anchor=top,baseline=-25pt]
 \rpp{3,3}{1,0}{1,1}
 \begin{scope}[xshift=6cm]
 \rpp{3}{0}{0}
 \end{scope}
\end{tikzpicture}
\end{align*}
\end{ex}

\case{$\bth \otimes 1^n \mapsto 1^{n-2}2 \otimes 1$}

\nopagebreak
Recall that $n \geq 2$. For the left hand side, after adding the right factor under $\Phi^{-1}$ we have $\nu^{(1)} = (s-n, s-n)$ and $\nu^{(2)} = (s-n)$ with all riggings equal to 0 and all strings being singular. Therefore after adding the $\young(\bth)$, we get $\nu^{(1)} = (s-n+2, s-n+1)$ with riggings $(0, 1)$ respectively, and $\nu^{(2)} = (s-n+1)$ with $J_{s-n+1}^{(2)} = (1)$. For the right hand side, after adding in the filling for the left factor (which has $s-n+1$ boxes), we have $\nu^{(1)} = (s-n+1, s-n+1)$ and $\nu^{(2)} = (s-n+1)$ with all riggings equal to $1$  and all strings being singular. Thus when adding $\young(2)$, we get $\nu^{(1)} = (s-n+2, s-n+1)$ with riggings $(0, 1)$ respectively, and $\nu^{(2)} = (s-n+1)$ with $J_{s-n+1}^{(2)} = (1)$. Next adding in the remaining $1^{n-2}$ does not change the rigged configuration. Thus the results are equal.

\begin{ex}
Consider $B = B^{1,1} \otimes B^{1,5}$ with $n = 3$, we have
\begin{align*}
\young(111\bon1) \xrightarrow[\hspace{40pt}]{\Phi^{-1}} &
\begin{tikzpicture}[scale=.35,anchor=top,baseline=-25pt]
 \rpp{2,2}{0,0}{0,0}
 \begin{scope}[xshift=6cm]
 \rpp{2}{0}{0}
 \end{scope}
\end{tikzpicture}
\\ \young(\bth) \otimes \young(111\bon1) \xrightarrow[\hspace{40pt}]{\Phi^{-1}} &
\begin{tikzpicture}[scale=.35,anchor=top,baseline=-25pt]
 \rpp{4,3}{0,1}{0,1}
 \begin{scope}[xshift=6cm]
 \rpp{3}{0}{0}
 \end{scope}
\end{tikzpicture}
\\ \young(\bon1\emptyset) \otimes \young(1) \xrightarrow[\hspace{40pt}]{\Phi^{-1}} &
\begin{tikzpicture}[scale=.35,anchor=top,baseline=-25pt]
 \rpp{3,3}{1,1}{1,1}
 \begin{scope}[xshift=6cm]
 \rpp{3}{0}{0}
 \end{scope}
\end{tikzpicture}
\\ \young(12\bon1\emptyset) \otimes \young(1) \xrightarrow[\hspace{40pt}]{\Phi^{-1}} &
\begin{tikzpicture}[scale=.35,anchor=top,baseline=-25pt]
 \rpp{4,3}{0,1}{0,1}
 \begin{scope}[xshift=6cm]
 \rpp{3}{0}{0}
 \end{scope}
\end{tikzpicture}
\end{align*}
\end{ex}

\case{$\bon \otimes 1 \mapsto \bon \otimes 1$}

\nopagebreak
We note that under $\Phi^{-1}$, just before adding $\young(\bon)$, we have $\nu^{(1)} = (s-1, s-1)$ and $\nu^{(2)} = (s-1)$ and all strings are singular for both sides (although the left hand side has different riggings and vacancy numbers than those from the right hand side). Thus when adding $\young(\bon)$, we get the same partitions. It is straightforward to see the resulting vacancy numbers are equal at this point, and hence the riggings are equal. Thus the results are equal.

\begin{ex}
Consider $B = B^{1,1} \otimes B^{1,5}$, we have
\begin{align*}
\young(1\bon1\bon1) \xrightarrow[\hspace{40pt}]{\Phi^{-1}} &
\begin{tikzpicture}[scale=.35,anchor=top,baseline=-25pt]
 \rpp{4,4}{0,0}{0,0}
 \begin{scope}[xshift=9cm]
 \rpp{4}{0}{0}
 \end{scope}
\end{tikzpicture}
\\ \young(\bon) \otimes \young(1\bon1\bon1) \xrightarrow[\hspace{40pt}]{\Phi^{-1}} &
\begin{tikzpicture}[scale=.35,anchor=top,baseline=-25pt]
 \rpp{6,6}{0,0}{0,0}
 \begin{scope}[xshift=9cm]
 \rpp{6}{0}{0}
 \end{scope}
\end{tikzpicture}
\\ \young(\bon1\bon1) \otimes \young(1) \xrightarrow[\hspace{40pt}]{\Phi^{-1}} &
\begin{tikzpicture}[scale=.35,anchor=top,baseline=-25pt]
 \rpp{4,4}{1,1}{1,1}
 \begin{scope}[xshift=9cm]
 \rpp{4}{0}{0}
 \end{scope}
\end{tikzpicture}
\\ \young(\bon\bon1\bon1) \otimes \young(1) \xrightarrow[\hspace{40pt}]{\Phi^{-1}} &
\begin{tikzpicture}[scale=.35,anchor=top,baseline=-25pt]
 \rpp{6,6}{0,0}{0,0}
 \begin{scope}[xshift=9cm]
 \rpp{6}{0}{0}
 \end{scope}
\end{tikzpicture}
\end{align*}
\end{ex}

\case{$\bon \otimes 11 \mapsto \emptyset \otimes 1$}

\nopagebreak
We note that the $\emptyset$ above refers to the empty tableau.
For the left hand side, once we've added the right factor under $\Phi^{-1}$, we get $\nu^{(1)} = (s-2, s-2)$ and $\nu^{(2)} = (s-2)$ with all riggings equal to 0 and strings being singular. Thus adding the $\young(\bon)$, we obtain $\nu^{(1)} = (s, s)$ and $\nu^{(2)} = (s)$ with all riggings equal to 1. For the right hand side, adding the filled elements results in $\nu^{(1)} = (s, s)$ and $\nu^{(2)} = (s)$ with all riggings equal to 1. Thus the results are equal.

\case{$\emptyset \otimes 1^n \mapsto 1^n \otimes \emptyset$ if $0 \leq n \leq s-1$}

\nopagebreak
For the left hand side, we note that after adding the right factor under $\Phi^{-1}$, we get $\nu^{(1)} = (s-n, s-n)$ and $\nu^{(2)} = (s-n)$ with all riggings equal to 0. Thus after adding $\young(\emptyset)$, we have $\nu^{(1)} = (s-n, s-n, 1, 1)$ and $\nu^{(2)} = (s-n, 1)$ with all riggings equal to 0. For the right hand side, after adding $\young(\emptyset)$, we note that all strings are singular and $\nu^{(1)} = (1,1)$ and $\nu^{(2)} = (1)$. Now if $n-s$ is odd, adding $\young(\emptyset)$ results in $\nu^{(1)} = (1,1,1,1)$ and $\nu^{(2)} = (1,1)$ with all riggings equal to $0$. Otherwise $n-s$ is even and adding $\young(1)$ makes all strings non-singular, so adding $\young(\bon)$ results in $\nu^{(1)} = (2,2,1,1)$ and $\nu^{(2)} = (2,1)$. Therefore as we add in the remaining filled portion of the left factor, and we get $\nu^{(1)} = (s-n, s-n, 1, 1)$ and $\nu^{(2)} = (s-n, 1)$ with all riggings equal to 0. Thus the results are equal.

\begin{ex}
Consider $B = B^{1,1} \otimes B^{1,5}$ with $n = 3$, we have
\begin{align*}
\young(111\bon1) \xrightarrow[\hspace{40pt}]{\Phi^{-1}} &
\begin{tikzpicture}[scale=.35,anchor=top,baseline=-25pt]
 \rpp{2,2}{0,0}{0,0}
 \begin{scope}[xshift=6cm]
 \rpp{2}{0}{0}
 \end{scope}
\end{tikzpicture}
\\ \young(\emptyset) \otimes \young(111\bon1) \xrightarrow[\hspace{40pt}]{\Phi^{-1}} &
\begin{tikzpicture}[scale=.35,anchor=top,baseline=-25pt]
 \rpp{2,2,1,1}{0,0,0,0}{0,0,0,0}
 \begin{scope}[xshift=6cm]
 \rpp{2,1}{0,0}{0,0}
 \end{scope}
\end{tikzpicture}
\\ \young(\emptyset) \xrightarrow[\hspace{40pt}]{\Phi^{-1}} &
\begin{tikzpicture}[scale=.35,anchor=top,baseline=-25pt]
 \rpp{1,1}{0,0}{0,0}
 \begin{scope}[xshift=6cm]
 \rpp{1}{0}{0}
 \end{scope}
\end{tikzpicture}
\\ \young(1) \otimes \young(\emptyset) \xrightarrow[\hspace{40pt}]{\Phi^{-1}} &
\begin{tikzpicture}[scale=.35,anchor=top,baseline=-25pt]
 \rpp{1,1}{0,0}{1,1}
 \begin{scope}[xshift=6cm]
 \rpp{1}{0}{0}
 \end{scope}
\end{tikzpicture}
\\ \young(111\bon1) \otimes \young(\emptyset) \xrightarrow[\hspace{40pt}]{\Phi^{-1}} &
\begin{tikzpicture}[scale=.35,anchor=top,baseline=-25pt]
 \rpp{2,2,1,1}{0,0,0,0}{0,0,0,0}
 \begin{scope}[xshift=6cm]
 \rpp{2,1}{0,0}{0,0}
 \end{scope}
\end{tikzpicture}
\end{align*}
\end{ex}

\case{$\emptyset \otimes 1^s \mapsto 1^{s-1} \otimes 1$}

\nopagebreak
For $\Phi^{-1}$, after inserting $\young(\emptyset)$ (note that the filled element on the right had side is $1^{s-1}\emptyset \otimes 1$) our partitions are $\nu^{(1)} = (1,1)$ and $\nu^{(2)} = (1)$ with all riggings equal to 1. All other insertions of $\young(1)$ under $\Phi^{-1}$ do not change the rigged configuration, therefore they are equal.
\end{proof}

\begin{thm}
\label{thm:statistics_preserving}
Let $B = \bigotimes_{i=1}^N B^{1,s_i}$ of type $D_4^{(3)}$. Then $\Phi \circ \eta$ sends cocharge to energy.
\end{thm}

\begin{proof}
We first note that it is sufficient to consider this on classically highest weight elements from Proposition~\ref{prop:cocharge_classical_invar} and that energy is invariant on each classical component.

There exists a sequence of maps $\rs$ and $R$ which transforms $B$ into $\bigotimes_{i=1}^{N'} B^{1,1}$. From Proposition~\ref{prop:correspondence}, and Theorem~\ref{thm:R_matrix_id}, we can do a corresponding sequence of $\rs$, $\rb$, and identity maps under $\widetilde{\Phi}$. Recall that $\rs$ on $B$ and $R$ preserve energy, and similarly $\rs$ on $\RC(B)$ and the identity map preserve cocharge. The result for $B' = \bigotimes_{i=1}^{N'} B^{1,1}$ was proven in Theorem~\ref{thm:bijection_single_columns}.
\end{proof}

\begin{cor}
\label{cor:X=M_type_D43}
Conjecture~\ref{conj:X=M} holds for $B = \bigotimes_{i=1}^N B^{r_i,1}$ or $B = \bigotimes_{i=1}^N B^{1,s_i}$ of type $D_4^{(3)}$.
\end{cor}

\begin{proof}
This follows immediately from Theorem~\ref{thm:bijection_single_columns} and Theorem~\ref{thm:statistics_preserving} and restricting to classically highest weight elements.
\end{proof}

Unfortunately, we cannot have a mix of factors of the form $B^{r,1}$ and $B^{1,s}$ at present. One may hope to use the virtualization to type $D_4^{(1)}$ and use the results of~\cite{OS10}, but this does not cover the necessary case of applying $\virtual{\rb}$ to $\virtual{B}^{2,s} \otimes \virtual{B}^{3,1}$. For the remainder of this section, we describe some partial results towards proving this and the general case following~\cite[Sec.~5]{S05}.

The $R$-matrix has been computed for $B^{1,1} \otimes B^{2,1}$ in~\cite[Appendix~B]{Yamada07}.

\begin{prop}
\label{prop:R_matrix_2}
Let $\g$ be of type $D_4^{(3)}$. The following diagram commutes:
\[
\xymatrixrowsep{4pc}
\xymatrixcolsep{3pc}
\xymatrix{\RC(B^{1,1} \otimes B^{2,1}) \ar[r]^-{\Phi} \ar[d]_{\id} & B^{1,1} \otimes B^{2,1} \ar[d]^R
\\ \RC(B^{2,1} \otimes B^{1,1}) \ar[r]^-{\Phi} & B^{2,1} \otimes B^{1,1}}
\]
\end{prop}

\begin{proof}
This is a finite computation.
\end{proof}

\begin{lemma}
\label{lemma:cocharge_left_top}
For $(\nu, J) \in \RC(B)$, we have
\[
\cc\bigl( \lt(\nu, J) \bigr) - \cc(\nu, J) = 1 + \sum_{j \in \ZZ_{>0}} L_j^{(1)}.
\]
\end{lemma}

\begin{proof}
Let $t_1^{\vee} = 1$ and $t_2^{\vee} = 3$. We can rewrite Equation~\eqref{eq:cocharge} as
\[
\cc(\nu) = \frac{1}{2} \sum_{(a,i) \in \HH_0} t_a^{\vee} \left( -p_i^{(a)} +  \sum_{j \in \ZZ_{>0}} \min(i,j) L_j^{(a)} \right) m_i^{(a)} + \sum_{(a,i) \in \HH_0} \sum_{x \in J_i^{(a)}} x
\]
since we can write Equation~\eqref{eq:vacancy_numbers} as
\[
p_i^{(a)} = \sum_{j \in \ZZ_{>0}} \min(i, j) L_j^{(a)} - \frac{1}{t_a^{\vee}} \sum_{(b,j) \in \HH_0} (\alpha_a | \alpha_b) \min(i, j) m_j^{(b)}.
\]

Recall that $\lt$ adds a singular string of length 1, and it is easy to verify that $\lt$ preserves the vacancy numbers. We also note that we increase $L_1^{(1)}$ by 2 and decrease $L_1^{(2)}$ by 1 when we apply $\lt$. We also note that
\[
p_1^{(1)} = \sum_{j \in \ZZ_{>0}} L_j^{(1)} - 2 m_j^{(1)} + 3 m_j^{(2)}.
\]
Thus we have
\begin{align*}
\cc\bigl( \lt(\nu, J) \bigr) - \cc(\nu, J) & = \frac{1}{2} \left( -p_1^{(1)} + 2 + \sum_{j \in \ZZ_{>0}} 2m_j^{(1)} - 3m_j^{(2)} + L_j^{(1)} \right) + p_1^{(1)}
\\ & = -p_1^{(1)} + 1 + \sum_{j \in \ZZ_{>0}} L_j^{(1)} + p_1^{(1)} = 1 + \sum_{j \in \ZZ_{>0}} L_j^{(1)},
\end{align*}
where the first terms of the sum inside the parentheses comes from the increase in $L_1^{(1)}$, the second from $t_2^{\vee} = 3$ and the decrease in $L_1^{(2)}$, the last terms and the $-p_1^{(1)}$ from the increase of $m_1^{(1)}$, and the additional 2 from the increase in $m_1^{(1)}$ and $L_1^{(1)}$.
\end{proof}

\begin{lemma}
\label{lemma:right_bottom_energy}
For $b \in \bigotimes_{i=1}^N B^{r_i, 1}$ with $r_N = 2$, we have
\[
D\bigl( \rb(b) \bigr) - D(b) = 1 + L_1^{(1)},
\]
where $L_1^{(1)}$ is the number of factors $B^{1,1}$ in $B$.
\end{lemma}

\begin{proof}
From the definition of energy from Equation~\eqref{eq:energy_function}, it is sufficient to consider the case when $B$ consists of at most two tensor factors and the right-most factor is $B^{2,1}$. For a single factor, this is a finite computation. For $B^{r,1} \otimes B^{2,1}$, let $\rb(b) = b_2 \otimes b_1$ where $b \in B^{2,1}$ and
\begin{align*}
R(b' \otimes b) & = b_R \otimes b_R', \\
R(b' \otimes b_2) & = \bullet \otimes \tilde{b}'_R, \\
R(\tilde{b}'_R \otimes b_1) & = \bullet \otimes \tilde{b}'_{RR},
\end{align*}
where the elements denoted by $\bullet$ will not affect the proof. It suffices to show that
\begin{gather*}
D(b_2 \otimes b_1) - D(b) = 1,
\\ H(b' \otimes b_2) + H(\tilde{b}'_R \otimes b_1) + D(\tilde{b}'_{RR}) - H(b' \otimes b) - D(b'_R) = \delta_{r,1}.
\end{gather*}
This is a finite computation using the results from~\cite{Yamada07} (note the $R$ matrix on $B^{2,1} \otimes B^{2,1}$ is the identity map).
\end{proof}

Thus to extend Corollary~\ref{cor:X=M_type_D43} to contain a mixture of factors, we would need an analog of Lemma~\ref{lemma:right_bottom_energy} or Theorem~\ref{thm:R_matrix_id} which can describe the case of $B^{1,s} \otimes B^{2,1}$.

\section{Extensions to type $G_2^{(1)}$}
\label{sec:extensions_G2}

In this section, we describe extensions of our results to type $G_2^{(1)}$. For this section, we assume $\g$ is of type $G_2^{(1)}$ and make all of the appropriate changes to our notation. In particular, the diagram folding $\phi \colon D_4^{(1)} \searrow G_2^{(1)}$ is given by $\phi^{-1}(0) = \{0\}$, $\phi^{-1}(1) = \{1,3,4\}$, and $\phi^{-1}(2) = \{2\}$ with scaling factors $\gamma_0 = 3$, $\gamma_1 = 1$, and $\gamma_2 = 3$. For the necessary description of type $G_2^{(1)}$ rigged configurations, their $U_q(\g_0)$-crystal structure, and the virtualization map, we refer the reader to~\cite{SS15}.

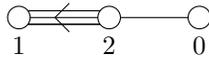
\begin{figure}[h]
\[
\label{fig:dynkin_diagram_Gaff}
\begin{tikzpicture}[scale=0.6,baseline=-0.5cm]
\draw (2 cm,0) -- (4.0 cm,0);
\draw (0, 0.15 cm) -- +(2 cm,0);
\draw (0, -0.15 cm) -- +(2 cm,0);
\draw (0,0) -- (2 cm,0);
\draw (0, 0.15 cm) -- +(2 cm,0);
\draw (0, -0.15 cm) -- +(2 cm,0);
\draw[shift={(0.8, 0)}, rotate=180] (135 : 0.45cm) -- (0,0) -- (-135 : 0.45cm);
\draw[fill=white] (0, 0) circle (.25cm) node[below=4pt]{$1$};
\draw[fill=white] (2 cm, 0) circle (.25cm) node[below=4pt]{$2$};
\draw[fill=white] (4 cm, 0) circle (.25cm) node[below=4pt]{$0$};
\end{tikzpicture}
\]
\caption{Dynkin diagram of type $G_2^{(1)}$.}
\end{figure}

\subsection{Extensions}

We can extend Lemma~\ref{lemma:hw_1} and Lemma~\ref{lemma:hw_2} to type $G_2^{(1)}$ as follows. 

\begin{lemma}
Consider the KR crystal $B^{2,s}$ of type $G_2^{(1)}$. We have
\[
\RC(B^{2,s}) = \bigoplus_{k=0}^s \RC(B^{2,s}; k \clfw_2).
\]
Moreover the highest weight rigged configurations in $\RC(B^{1,s})$ are given by
\begin{align*}
\nu^{(1)} & = (3k),
\\ \nu^{(2)} & = (k, k).
\end{align*}
with all riggings equal to $0$.
\end{lemma}

\begin{proof}
Here we have $\virtual{B}^{2,s} = B^{2, 3s}$, and by only selecting nodes at levels $3j$ by condition~(S2) of the virtual Kleber algorithm given in~\cite{OSS03II,SS15}, devirtualizing gives us our desired rigged configurations.
\end{proof}

\begin{lemma}
\label{lemma:hw_Gaff_1}
Consider the KR crystal $B^{1,s}$ of type $G_2^{(1)}$. We have
\[
\RC(B^{1,s}) = \bigoplus_{\lambda} \RC(B^{1,s}; \lambda)
\]
where $\lambda$ runs over all weights of the form
\begin{align*}
& s \clfw_1 - k_1 (2\alpha_1 + \alpha_2) - k_2 \left( \alpha_1 + \frac{\alpha_2}{3} \right) - k_3 \frac{\alpha_2}{3}
\\ & = s \clfw_1 - k_1 \clfw_1 - k_2 \left( \clfw_1 - \frac{\clfw_2}{3} \right) - k_3 \left( -\clfw_1 + \frac{2}{3} \clfw_2 \right)
\\ & = (s - k_1 - k_2 + k_3) \clfw_1 + \frac{k_2 - 2k_3}{3} \clfw_2
\end{align*}
where $k_1, k_2, k_3 \in \ZZ_{\geq 0}$ and satisfy:
\begin{itemize}
\item[(I)] $2 k_3 \leq k_2$ and $k_1 + k_2 \leq s$;
\item[(M)] $k_1 \equiv 0 \bmod 3$ and $k_2 + k_3 \equiv 0 \bmod 3$.
\end{itemize}
Moreover, the highest weight rigged configurations in $\RC(B^{2,s})$ are given by
\begin{align*}
\nu^{(1)} & = (k_1 + k_2, k_1)
\\ \nu^{(2)} & = \left( \frac{k_1 + k_2 + k_3}{3}, \frac{k_1}{3}, \frac{k_1}{3} \right)
\end{align*}
and the multiplicity of the node is equal to $1 + \frac{k_2 - 2 k_3}{3}$. Let $k = \frac{k_1 + k_2 + k_3}{3}$. Then $p_k^{(2)} = \frac{k_2 - 2 k_3}{3}$ and $p_i^{(a)} = 0$ for all other $(a, i) \in \HH_0$.
\end{lemma}

\begin{proof}
The proof is similar to Lemma~\ref{lemma:hw_2} with the following changes. We must have $k_1 \equiv 0 \bmod 3$ by condition~(V2) of Definition~\ref{def:virtual_kleber}, and we only select nodes at levels $3j$ by condition~(S2) of the virtual Kleber algorithm (see~\cite{OSS03II, SS15}), which implies $k_2 + k_3 \equiv 0 \bmod 3$. Thus devirtualization gives us our desired rigged configurations, noting that $k_2 + k_3 \equiv k_2 - 2k_3 \bmod 3$.
\end{proof}

We can also parameterize $\hwRC(B^{1,s})$ by an additional $k_4 \in \ZZ_{\geq 0}$ which satisfies $k_4 \leq \frac{k_2 - 3k_3}{3}$.
\begin{prop}
Fix some $(\nu, J) \in \hwRC(B^{1,s})$ of type $G_2^{(1)}$. Then we have
\[
\cc(\nu, J) = k_1 + \frac{k_2 + k_3}{3} + k_4 = \lvert \nu^{(2)} \rvert + k_4.
\]
\end{prop}

\begin{proof}
By direct computation using the cocharge for type $G_2^{(1)}$.
\end{proof}

There is another parameterization of $\hwRC(B^{1,s})$ that removes the fractions and the relation to the original parameterization is given by
\begin{align*}
k_1' & = k_1,
\\ k_2' & = \frac{k_2 + k_3}{3},
\\ k_3' & = \frac{k_2 - 2k_3}{3},
\\ k_4' & = k_4.
\end{align*}
Thus our conditions for $(\nu, J) \in \hwRC(B^{1,s})$ reduce to
\begin{itemize}
\item[(I)] $k_3' \geq 0$ and $k_1' + 2k_2' + k_3' \leq s$,
\item[(M)] $k_1' \equiv 0 \bmod 3$,
\end{itemize}
and our other conditions are $0 \leq k_4' \leq k_3' \leq k_2'$ and $0 \leq k_1'$. We also have
\begin{align*}
\overline{\wt}(\nu, J) & = \bigl( s - k_1' - (2k_2' + k_3') + (k_2' - k_3') \bigr) \clfw_1 + k_3' \clfw_2
\\ & = (s - k_1' - k_2' - 2k_3') \clfw_1 + k_3' \clfw_2,
\\ \cc(\nu, J) & = k_1' + k_2' + k_4'.
\end{align*}

We also prove the analogous statement to Proposition~\ref{prop:affine_grading} following~\cite{CM07} to obtain
\[
\ch_t W^{1,s} = \sum_{r \in \mcA} t^{\gr(r)} \ch V\bigl( \wt(r) \bigr)
\]
by using rigged configurations and Lemma~\ref{lemma:hw_Gaff_1}.

\begin{prop}
Define
\begin{align*}
\mcA & = \{r \in \ZZ_{\geq 0}^4 \mid r_4 \leq r_2, \text{ and } 2r_1 + 3r_2 + 3r_3 \leq s\},
\\ \wt(r) & = (s - r_1 - 3r_2 - 3r_3) \clfw_1 + (r_2 + r_3 - r_4) \clfw_2,
\\ \gr(r) & = r_1 + r_2 + 2 r_3 + 2 r_4.
\end{align*}
Then there exists a bijection $\Psi \colon \mcA \longrightarrow \RC(B^{1,s})$ such that
\begin{align*}
\gr(r) & = \cc\bigl( \Psi(r) \bigr),
\\ \overline{\wt}(r) & = \overline{\wt}\bigl( \Psi(r) \bigr).
\end{align*}
\end{prop}

\begin{proof}
We define $\Psi$ by
\begin{align*}
k_1' & = 3r_4,
\\ k_2' & = r_1 + r_2 + r_3 - r_4,
\\ k_3' & = r_2 + r_3 - r_4,
\\ k_4' & = r_3,
\end{align*}
and this is invertible with inverse defined by
\begin{align*}
r_1 & = k_2' - k_3',
\\ r_2 & = \frac{k_1'}{3} + k_3' + k_4',
\\ r_3 & = k_4',
\\ r_4 & = \frac{k_1'}{3}.
\end{align*}

A straightforward computation shows that the weights are preserved and $\cc$ goes to $\gr$. It is also clear that $k_1' \equiv 0 \bmod 3$. The condition that $r_2 \geq r_4$ and $r_i \geq 0$ implies that $0 \leq k_4' \leq k_3' \leq k_2'$ and $k_1' \geq 0$. Since $k_1' + 2k_2' + k_3' = 2r_1 + 3r_2 + 3r_3$, we have $k_1' + 2k_2' + k_3' \leq s$ is equivalent to $2r_1 + 3r_2 + 3r_3 \leq s$. It is clear that $k_i' \geq 0$ implies $r_2, r_3, r_4 \geq 0$ and $r_4 \leq r_2$. Additionally $k_2' \geq k_3'$ implies that $r_1 \geq 0$.
\end{proof}

We also note that $B^{2,s}$ virtualizes in $\virtual{B}^{1,3s}$ type $D_4^{(3)}$ with scaling factors $\gamma_0 = 3$, $\gamma_1 = 1$ and $\gamma_2 = 3$ by~\cite[Prop.~1]{MMO10} (note the difference in indexing of the Dynkin diagram). This gives us the following special case of Conjecture~3.7 in~\cite{OSS03II}.

\begin{cor}
\label{cor:KR_virtualization_G2aff}
We have the following virtualizations of type $G_2^{(1)}$ in type $D_4^{(1)}$:
\[
B^{2,s} \longrightarrow \virtual{B}^{2,3s}.
\]
\end{cor}

\begin{proof}
This follows from the composition of virtualization maps
\[
G_2^{(1)} \longrightarrow D_4^{(3)} \longrightarrow D_4^{(1)},
\]
which exists by~\cite[Prop.~1]{MMO10} and Corollary~\ref{cor:KR_virtualization_1}.
\end{proof}

\begin{remark}
We have the analog for Corollary~\ref{cor:KR_virtualization_G2aff} for $B^{2,s}$ virtualizing into type $A_7^{(2)}$ similar to Remark~\ref{rem:type_Atwisted}. However, we do not have an analogous statement to Remark~\ref{rem:type_Atwisted} for $G_2^{(1)}$ virtualizing into type $B_3^{(1)}$ as 2 does not divide 3 and the virtualization maps must commute.
\end{remark}

\subsection{Conjectures}

Defining the corresponding algorithm for $\delta$ in type $G_2^{(1)}$ is non-trivial. For example, in $B^{1,1}$ by considering the crystal structure, we must have:
\begin{gather*}
\begin{tikzpicture}[scale=.35,anchor=top]
 \rpp{2}{-1}{-1}
 \begin{scope}[xshift=6cm]
 \rpp{1}{0}{0}
 \end{scope}
\end{tikzpicture}
 \mapsto \young(3),
\\
\begin{tikzpicture}[scale=.35,anchor=top]
 \rpp{3}{-2}{-2}
 \begin{scope}[xshift=6cm]
 \rpp{1}{1}{1}
 \end{scope}
\end{tikzpicture}
 \mapsto \young(0).
\end{gather*}
Also we consider $B^{1,1} \otimes B^{1,1} \otimes B^{1,1}$ and the classical components isomorphic to $B(\clfw_1)$.
\begin{gather*}
\young(1) \otimes \young(\bon) \otimes \young(1)
\\ \young(\bon) \otimes \young(1) \otimes \young(1)
\\ \young(0) \otimes \young(0) \otimes \young(1)
\\ \young(\btw) \otimes \young(2) \otimes \young(1)
\end{gather*}
which must be in bijection with the rigged configurations
\begin{gather*}
\begin{tikzpicture}[scale=.35,anchor=top]
 \rpp{3,1}{j_1,j_2}{1,1}
 \begin{scope}[xshift=6cm]
 \rpp{1,1}{0,0}{0,0}
 \end{scope}
\end{tikzpicture},
\end{gather*}
where $j_1 = 0,1$ and $j_2 = 0,1$, by some procedure $\delta$ applied 3 times. This follows from wanting a classical crystal isomorphism which sends cocharge to energy. There will likely need to be some kind of modified case~(Q) similar to the description of $\delta$ in type $B_n^{(1)}$ given in~\cite{OSS03}. Once $\delta$ has been defined, the author expects similar techniques will show an analog of the results given here using the description of the combinatorial $R$-matrix given in~\cite{MOW12} and the $e_0$ and $f_0$ given in~\cite{MMO10}. Moreover, the virtualization map given in~\cite{MMO10} could be used to find results for $B^{2,s}$ of type $G_2^{(1)}$ using $B^{1,3s}$ of type $D_4^{(3)}$.

It is known that the KR module $W^{2,s}$ of type $G_2^{(1)}$ admits a crystal basis $B^{2,s}$, which is a perfect crystal of level $s$~\cite{Yamane98} (note in~\cite{Yamane98}, a different indexing convention is used). It is also known that $W^{1,1}$ admits a crystal basis $B^{1,1}$, which is a perfect crystal of level 1~\cite{LNSSS14,LNSSS14II}. However it is still an open conjecture that the KR module $W^{1,s}$ admits a crystal basis in general. 

The author also conjectures that similar techniques, along with the results of~\cite{MMO10, Yamane98}, can be used to prove a special case of the algorithm in~\cite{Mahathir} for $B^{2,s}$, and thereby obtaining similar results as this paper for type $G_2^{(1)}$.

\section{Conjectures for affine crystal structure on rigged configurations}
\label{sec:affine_conjectures}

We now give a conjecture on an explicit description of the $U_q'(\g)$-crystal structure on rigged configurations for $\g$ of any affine type except $A_n^{(1)}$.

We say an affine type is \defn{single-bonded} if there exists a unique $N_0 \in I_0$ such that $A_{i,0} = A_{0,i} = -\delta_{N_0,j}$ for all $i \in I_0$. In other words, there exists a unique simply-laced edge between $0$ and some $i$ in the Dynkin diagram.

We first recall the uniform construction of certain level 1 perfect crystals for arbitrary affine types given in~\cite{BFKL06}. We restrict ourselves when $\g$ is of untwisted type for simplicity of the exposition, but we note that there are analogous definitions for twisted types. Let $(c_a)_{a \in I}$ be the Kac labels (so the null root $\delta = \sum_{a \in I} c_a \alpha_a$). Define
\[
\theta = c_1 \alpha_1 + \cdots + c_n \alpha_n,
\]
and let $R$ denote the roots of the classical Lie algebra $\g_0$. We note that $B(\theta)$ is the crystal of the adjoint representation, so the vertices are $\{x_{\alpha} \mid \alpha \in R\} \sqcup \{y_i \mid i \in I\}$ and the $U_q(\g_0)$-crystal has $i$-edges
\begin{itemize}
\item $x_{\alpha} \longrightarrow x_{\beta}$ if and only if $\alpha - \alpha_i = \beta$, or
\item $x_{\alpha_i} \longrightarrow y_i \longrightarrow x_{-\alpha_i}$.
\end{itemize}
We then define a level 1 perfect crystal by the classical decomposition $B \iso B(\theta) \oplus B(0)$. Recall that we defined $\emptyset$ to be the unique element of $B(0)$. Now we define $0$-edges by
\begin{itemize}
\item $x_{\alpha} \longrightarrow x_{\beta}$ if and only if $\alpha + \theta = \beta$ and $\alpha,\beta \neq \pm \theta$, or
\item $x_{-\theta} \longrightarrow \emptyset \longrightarrow x_{\theta}$.
\end{itemize}

\begin{thm}[{\cite{BFKL06}}]
\label{thm:uniform_construction_level1}
Let $B$ be the $U_q'(\g)$-crystal defined above. Then $B$ is a perfect crystal of level $1$.
\end{thm}

In the single-bonded affine types, we have $B = B^{N_0,1}$. For the remaining types except $A_n^{(1)}$, this corresponds to $B^{1, \kappa}$ where
\[
\kappa = \begin{cases}
2 & \text{if } \g = C_n^{(1)}, \\
1 & \text{otherwise}.
\end{cases}
\]
For simplicity, let $N_0 = 1$ for the non-single-bonded types.

Let $(c_a^{\vee})_{a \in I}$ be the dual Kac labels, which are the Kac labels of the type obtained by reversing the arrows. Define $t_a = \max(c_a / c_a^{\vee}, c^{\vee}_0)$ and $t_a^{\vee} = \max(c_a^{\vee} / c_a, c_0)$.

\begin{dfn}
\label{def:RC_affine_ops}
Let $\g$ be of affine type. Consider a rigged configuration $(\nu, J) \in \RC(B^{N_0,s})$. We define $e_0, f_0$ as follows.
\begin{itemize}
\item[$f_0$:] If $\nu^{(a)}$ does not have $c_a / t_a$ rows for all $a \in I_0$, then $f_0(\nu, J) = 0$. Otherwise define $f_0(\nu, J)$ by removing $t_a$ boxes from each row of $\nu^{(a)}$ for all $a \in I_0$ and keeping the colabels fixed.

\item[$e_0$:] Add $t_a$ boxes to the first $c_a / t_a$ rows of $\nu^{(a)}$ and keeping the colabels fixed (we consider a row of length 0 to be singular). If the result is in $\RC(B^{N_0,s})$, then it is $e_0(\nu, J)$, otherwise $e_0(\nu, J) = 0$.
\end{itemize}
\end{dfn}

We propose the following generalization of Theorem~\ref{thm:uniform_construction_level1}.

\begin{conj}
\label{conj:RC_affine_crystal}
Let $\g$ be of affine type except $A_n^{(1)}$. Consider a rigged configuration $(\nu, J) \in \RC(B^{N_0,\kappa s})$. Then $\nu^{(a)}$ is contained in a $(c_a / t_a) \times (2 t_a s)$ rectangle for all $a \in I_0$ and the $U_q'(\g)$-crystal structure is given by Definition~\ref{def:RC_affine_ops}.
\end{conj}

\begin{conj}
\label{conj:classically_lowest_weight}
Let $\g$ be of affine type except $A_n^{(1)}$. The classically lowest weight element $(\nu, J) \in B(\kappa k \clfw_{N_0}) \subseteq B^{N_0, \kappa s}$ is given by
\[
\nu^{(a)} = (c_a / t_a)^{2t_a s - \kappa k}
\]
with all riggings 0 except for those in $(\nu, J)^{(N_0)}$, which are $-s - \kappa k$.
\end{conj}

Conjecture~\ref{conj:RC_affine_crystal} and Conjecture~\ref{conj:classically_lowest_weight} have been verified by computer using~\cite{sage} for $s = 1$ up to rank 8 and $s = 2$ up to rank 4.

We note that Theorem~\ref{thm:unique_affine_structure} was shown for all single-bonded types in~\cite{KMOY07}. Moreover, an algorithm for $\delta$ was proposed in~\cite{Mahathir} which could be used for all types except $A_n^{(1)}$. Therefore, combining this description of $\delta$ (which would extend to $\Phi$), Theorem~\ref{thm:uniform_construction_level1}, and Theorem~\ref{thm:unique_affine_structure} could lead to a partial \defn{type-independent} proof of Conjecture~\ref{conj:RC_affine_crystal}.

For the remainder of this section, we restrict ourselves to $\g$ of type $D_4^{(3)}$, so $N_0 = 1$, $c_1 = 2$, $c_2 = 1$, and $t_1 = t_2 = 1$. We first note that Conjecture~\ref{conj:RC_affine_crystal} is equivalent to Conjecture~\ref{conj:KR_virtualization}.

\begin{prop}
Conjecture~\ref{conj:classically_lowest_weight} holds in type $D_4^{(3)}$.
\end{prop}

\begin{proof}
This follows from the definition of the bijection $\Phi$ and classically lowest weight elements in $B^{1,s}$. We leave the details for the reader.
\end{proof}

We consider $(\nu, J) \in \RC(B^{1,s})$ of type $D_4^{(3)}$ for the remainder of this section. We note that if Conjecture~\ref{conj:RC_affine_crystal} holds, then $f_0$ increases a label on a string of $\nu^{(1)}$ if and only if the string has length at least $s$. Additionally, this conjecture implies that we have $\phi_0(\nu, J) \leq 2s - \max(\nu^{(2)}_1, \nu_2^{(1)})$ and $\varepsilon_0(\nu, J) \leq 2s - \nu_1^{(1)}$, however these bounds are not sufficient. We also note that we have
\begin{align*}
\overline{\wt}(\nu, J) & = s \clfw_1 -  \lvert \nu^{(1)} \rvert (2\clfw_1 - \clfw_2) -  \lvert \nu^{(2)} \rvert (-3\clfw_1 + 2\clfw_2)
\\ & = \bigl( s + 3 \lvert \nu^{(2)} \rvert - 2 \lvert \nu^{(1)} \rvert \bigr) \clfw_1 + \bigl( \lvert \nu^{(1)} \rvert - 2 \lvert \nu^{(2)} \rvert \bigr) \clfw_2.
\end{align*}
Hence
\[
\inner{\wt(\nu, J)}{h_0} = 2c_1 + 3c_2 = 2 \bigl( s + 3\lvert \nu^{(2)} \rvert - 2\lvert \nu^{(1)} \rvert \bigr) + 3 \bigl( \lvert \nu^{(1)} \rvert - 2 \lvert \nu^{(2)} \rvert \bigr) = 2s - \lvert \nu^{(1)} \rvert.
\]

We express Equation~\eqref{eq:affine_f} by
\[
f_0(b) = \begin{cases}
\young(1) + b \quad & \text{if $(F_1)$ holds,} \\ 
\young(\bon) \mapsto \young(0) & \text{if $(F_2)$ holds,} \\
\young(\btw) \mapsto \young(3) & \text{if $(F_3)$ holds,} \\
\young(\bth) \mapsto \young(2) & \text{if $(F_4)$ holds,} \\
\young(0) \mapsto \young(1) & \text{if $(F_5)$ holds,} \\
b - \young(\bon) \quad & \text{if $(F_6)$ holds.} 
\end{cases}
\]
where we change/add/remove one such box and reorder as necessary, and similarly for $e_0$. From the definition of the crystal operators and Proposition~\ref{prop:filling_map_1}, we know that for $(\nu, J) = \Phi^{-1}(b)$, we must have $\nu^{(1)}$ contained inside a $2 \times k$ box and $\nu^{(2)}$ inside a $1 \times k$ box. Thus from the description of $\Phi$, each of the above operations for $f_0$ must remove a box from each row of $\nu$. This is further evidence that Conjecture~\ref{conj:RC_affine_crystal} should be true.

\appendix
\section{Calculations using Sage}
\label{sec:sage}

Rigged configurations, Kirillov--Reshetikhin tableaux, and the bijection between in type $D_4^{(3)}$ them has been implemented by the author in Sage~\cite{sage}.  We conclude with examples. We begin by setting up the Sage environment to give a more concise printing.

\begin{lstlisting}
sage: RiggedConfigurations.global_options(display="horizontal")
\end{lstlisting}

We construct our the rigged configuration from Example~\ref{ex:runningrig} (in the $U_q^{\prime}(\g)$ setting).

\begin{lstlisting}
sage: RC = RiggedConfigurations(['D',5,1], [[1,2], [2,1], [3,1]])
sage: hw = RC(partition_list=[[1,1], [1]], rigging_list=[[1,0], [0]]); hw
1[ ]1   1[ ]0
1[ ]0
sage: hw.weight()
-7*Lambda[0] + 2*Lambda[1] + Lambda[2]
sage: elt = hw.f_string([2,1,1,1,2,2]); elt
5[ ][ ][ ][ ]3   -2[ ][ ][ ][ ]-2
1[ ]1
sage: elt.weight()
-4*Lambda[0] + 5*Lambda[1] - 2*Lambda[2]
sage: elt.e(1)
sage: elt.e(2)
2[ ][ ][ ][ ]0   -1[ ][ ][ ]-1
1[ ]1                         
sage: elt.f(1)
 3[ ][ ][ ][ ]1   -1[ ][ ][ ][ ]-1
-1[ ]-1                           
-1[ ]-1                           
sage: elt.f(2)
\end{lstlisting}
Alternatively, one could construct $(\nu,J)$ from Example~\ref{ex:runningrig} directly by specifying the partitions and corresponding labels.
\begin{lstlisting}
sage: elt = RC(partition_list=[[4,1],[4]], rigging_list=[[3,1], [-2]]); elt
5[ ][ ][ ][ ]3   -2[ ][ ][ ][ ]-2
1[ ]1
\end{lstlisting}

We then show the image under $\Phi$ as in Example~\ref{ex:Phi}.

\begin{lstlisting}
sage: elt.to_tensor_product_of_kirillov_reshetikhin_tableaux()
[[3]] (X) [[3], [-3]] (X) [[1, 3]]
\end{lstlisting}

\section*{Acknowledgements}

The author would like to thank Masato Okado for useful discussions and for the reference~\cite{Mahathir}. The author also would like thank Ben Salisbury for comments on an early draft of this paper. Additionally, the author would like to thank Anne Schilling for comments on an early draft of this paper and useful discussions. This work benefited from computations done in Sage~\cite{sage}. Finally, the author would like to thank the anonymous referee for valuable comments.

The majority of this work was done at University of California Davis.

\bibliographystyle{alpha}
\bibliography{biject}{}
\end{document}